\documentclass[11pt]{article}

\usepackage{geometry, fullpage, authblk}
\usepackage{xr-hyper}
\usepackage{times,enumitem}
\usepackage{hyperref}[]
\hypersetup{
	colorlinks=true,
	linkcolor=blue,
	filecolor=magenta,      
	urlcolor=cyan,
	citecolor=blue,
}
\usepackage{ulem}

\usepackage{url}

\usepackage{bigints}
\usepackage{amsfonts,amsmath,amssymb,amsthm, bbm}
\usepackage{wrapfig}
\usepackage{verbatim,float,url}
\usepackage{graphicx}
\usepackage{subcaption}
\usepackage[round]{natbib}   
\usepackage[english]{babel}
\usepackage{cancel}
\usepackage{color}
\usepackage[dvipsnames]{xcolor}
\usepackage{cleveref}

\usepackage{mathtools}

\newtheorem{theorem}{Theorem}
\newtheorem{lemma}[theorem]{Lemma}

\usepackage{mdwlist}	

\newtheorem{remark}{Remark}
\theoremstyle{definition}

\newtheorem{example}{Example}

\usepackage{tikz}  
\usetikzlibrary{trees}
\usetikzlibrary{arrows.meta}
\usetikzlibrary{decorations.pathreplacing}

\graphicspath{{figs/}}

\usepackage{algorithm}
\usepackage{algpseudocode}

\algnewcommand\algorithmicinput{\textbf{INPUT:}}
\algnewcommand\INPUT{\item[\algorithmicinput]}
\algnewcommand\algorithmicoutput{\textbf{OUTPUT:}}
\algnewcommand\OUTPUT{\item[\algorithmicoutput]}

\usepackage{mathtools}

\DeclarePairedDelimiter\floor{\lfloor}{\rfloor}

\theoremstyle{plain}
\newtheorem{assumption}{Assumption}
\allowdisplaybreaks
\usepackage{bm}
\usepackage{booktabs}

\title{Variance estimation in graphs with the fused lasso}

\author[1]{Oscar Hernan Madrid Padilla}

\affil[1]{\small Department of Statistics, University of California, Los Angeles}

\begin{document}
	\maketitle
	
	\begin{abstract}

We study the problem of variance estimation in general graph-structured problems. First, we develop a linear time estimator for the homoscedastic case that can consistently estimate the variance in general graphs. We show that our estimator attains minimax rates for the chain and 2D grid graphs when the mean signal has  total variation with canonical scaling. Furthermore, we provide general upper bounds on the mean squared error performance of the fused lasso estimator in general graphs under a  moment condition and a bound on the tail behavior of the errors. These upper bounds allow us to generalize for broader classes of distributions, such as sub-exponential,  many existing results on the fused lasso that are only known to hold with the assumption that errors are sub-Gaussian random variables. Exploiting our upper bounds, we then study a simple total variation regularization estimator for estimating the signal of variances in the heteroscedastic case. We also provide lower bounds  showing  that our heteroscedastic  variance estimator attains minimax rates for estimating signals of bounded variation in grid graphs, and $K$-nearest neighbor graphs, and the estimator  is consistent for estimating the variances in any connected graph.
	\end{abstract}
	

\section{Introduction}


Consider the problem of estimating signals $\theta^* \in \mathbb{R}^n$ and $v^*\in \mathbb{R}_{+}^n$, based on  data  $\{y_i\}_{i=1}^n \subset \mathbb{R}$ generated as
\begin{equation}
	\label{eqn:model}
	y_i \,=\, \theta_i^*   +   (v_i ^*)^{1/2}\epsilon_i,
\end{equation}
where  $\epsilon_1,\ldots,\epsilon_n$ are independent and  $\mathbb{E}(\epsilon_i) = 0$,  and  $\text{var}(\epsilon_i)  = 1$, and where $y_i$ is associated with node $i$  in a  conected graph $G = (V,E)$ where  $V =\{1,\ldots,n\}$ and  $E \subset  V\times  V$. 
This class of graph estimation problems has appeared in applications in biology  \citep{tibshirani2005sparsity}, image processing \citep{rudin1992nonlinear,tansey2014false}, traffic detection \citep{wang2016trend}, among others.

A common method for estimating the signal $\theta^*$ is the fused lasso over graphs, also known as (anisotropic) total variation denoising over graphs, independently introduced by   \cite{rudin1992nonlinear} and \citep{tibshirani2005sparsity}. This consists of solving  the optimization problem 
\begin{equation}
	\label{eqn:gfl}  
	\hat{\theta}  \,:=\, \underset{\theta \in  \mathbb{R}^n}{\arg\min} \,\left\{  \frac{1}{2}\|  y-\theta\|^2 \,+\,\lambda \|\nabla_G \theta\|_1         \right\},
\end{equation}
where    $y =(y_1,\ldots,y_n)^{\top}$, $\lambda >0$ is a tuning parameter, and  $\nabla_G \in  \mathbb{R}^{  \vert E\vert \times n }$ is the incidence matrix of $G$.  Specifically,  each row of  $\nabla_G$ corresponds to an edge $e = (e^{+},e^{-}) \in E$ and
\[
(\nabla_G)_{e,\ell}   \begin{cases}
	1 & \text{if}  \,\, \ell=e^{+},\\
	-1 & \text{if}  \,\, \ell=e^{-},\\
	0 & \text{otherwise.}
\end{cases}
\]
The motivation behind (\ref{eqn:gfl}) is to have an estimator that balances between fitting the data well, with the first term in the objective function in (\ref{eqn:gfl}), and having a small complexity in terms of the quantity $\|\nabla_G \theta\|_1$ which is known as the total variation of the signal $\theta$ along the graph $G$. Intuitively, if the graph $G$ is informative about the signals $\theta^*$ and $v^*$, then we would expect that $\|\nabla_G\theta^*\|_1,\|\nabla_G v^*\|_1 << n$.      
For instance, suppose that $G$ is constructed as a $K$-NN graph based on features $\{x_i\}_{i=1}^n \subset \mathbb{R}^d$, and assume that $\theta_i^* = f_0(x_i)$ for all $i=1,\ldots, n$, and for   a smooth function $f_0$. If $K$ is small, then for $\{i,j\}$ an edge in $G$, we have that $\vert \theta_i^* - \theta_j^*\vert  =  \vert f_0(x_i) – f_0(x_j)\vert  $ which would be a small quantity or zero for most edges. Then summing over all the edges, we obtain $\| \nabla_G \theta^*\|_1 <<n$. In fact, 
\cite{padilla2018adaptive} showed that $\| \nabla_G \theta^*\|_1 = O_{ \text{pr} }(n^{1-1/d} )$, ignoring logarithmic factors,  provided that $f_0 $ is a piecewise Lipschitz  function.

The estimator defined in (\ref{eqn:gfl})  has attracted a lot of attention in the literature. Specifically,  computationally efficient algorithms for chain graphs were developed by \cite{johnson2013dynamic}, for grid graphs by \cite{barbero2014modular}, and for general graphs by \cite{tansey2015fast,chambolle2009total}.  Moreover, several authors have studied the statistical properties of  (\ref{eqn:gfl}) in different settings. In particular,  \cite{mammen1997locally}, and \cite{tibshirani2014adaptive} studied slow rates of convergence in chain graphs with signals having bounded variation. \cite{dalalyan2017prediction,lin2017sharp,guntuboyina2020adaptive}, and \cite{ortelli2021prediction} proved fast rates  for piecewise constant signals.   \cite{hutter2016optimal}, \cite{sadhanala2016total}, \cite{ortelli2020adaptive}, and \cite{chatterjee2021new} studied statistical properties of total variation denoising in grid graphs. \cite{padilla2016dfs}, and  \cite{ortelli2018total} studied the fused lasso in general graphs.
\cite{wang2016trend}, and \cite{sadhanala2021multivariate} focused on developing higher order versions of  total variation denoising.

Despite the tremendous attention from the literature focusing on the fused lasso as defined in (\ref{eqn:gfl}), most of the statistical work assumes that the errors $\{\epsilon_i\}_{i=1}^n$ are sub-Gaussian when studying the estimator (\ref{eqn:gfl}). While some works have considered the model in (\ref{eqn:model}) with  more arbitrary distributions, such as \cite{madrid2020risk} and \cite{ye2021non}, these efforts have studied the quantile version of  (\ref{eqn:model}). Thus,  the performance of the estimator defined in (\ref{eqn:gfl}) is not understood beyond the sub-Gaussian errors assumption. 

Additionally, the literature has been silent about estimating the variances in (\ref{eqn:model}).  Thus, there is currently no estimator available in the literature for estimating the variances 
even in the homoscedastic case, where the $v^*_i$ are all equal to some $v_0^*>0$,  when $G$ is a general graph.  Estimation of the variance is an important problem because it would allow practitioners the possibility of quantifying the variability of the data in different regions of the graph. For instance, if $y_i$ is the crime rate at location $i$, then we could have two locations where $\mathbb{E}(y_i) = \mathbb{E}(y_j)$, however, knowing that $\mathrm{var}(y_i) > \mathrm{var}(y_j)$ would be informative about the nature of  crime at location $i$ versus location $j$. 

In this paper, we fill the gaps described above regarding mean and variance estimation in general graphs. Our main contributions are listed next.

\subsection{Summary of results}

We make the following contributions for the model described in (\ref{eqn:model}) with a connected graph $G$. 

\begin{enumerate}
	\item If the variances satisfy $v_i^* = v_0^*$ for all $i=1,\ldots,n$, then we show that, under a simple moment condition, there exists an  estimator $\hat{v}$  that can be found in linear time, $O(n +  \vert E\vert) $, and satisfies 
	\begin{equation}
		\label{eqn:u1}
		\vert \hat{v} -v_0^*\vert \,=\,O_{\mathrm{pr}}\left(   \frac{v_0^*}{n^{1/2}}  +  \frac{\|\nabla_G \theta^*\|_{1} }{n}   \right).
	\end{equation} The estimator $\hat{v}$ is based on first running depth-first search (DFS) on the graph $G$ and then using the differences of the $y_i$'s along the ordering. A detailed construction is given in Section \ref{sec:homo}. Notably, when $G$ is a 1D or 2D grid graph and  $\|\nabla_G \theta^*\|_{1}$ has a canonical scaling, the rate in (\ref{eqn:u1}) is minimax optimal.  Moreover, our estimator is the first for the problem of estimating the variance in the sequence model where the measurements are collected in a general graph.  We also  show with experiments in Appendix \ref{sec:model_selection} that the estimator $\hat{v}$ can be useful for model selection when the goal is to estimate $\theta^*$.
	
	\item For the fused lasso estimator defined in (\ref{eqn:gfl}), under a moment condition and an assumption stating that 
	\begin{equation}
		\label{eqn:tail_c}
		\underset{i=1,\ldots,n}{\max} \,\mathrm{pr}(  \vert \epsilon_i \vert >U_n )  \,\rightarrow\,0
	\end{equation}
	fast enough, where $U_n >0$ is a sequence,  we show that:
	\begin{enumerate}
		\item For any connected graph,  ignoring logarithmic factors, it holds that
		\begin{equation}
			\label{eqn:rate0.1}
			\frac{\|\hat{\theta}-\theta^*\|^2}{n}\,=\, O_{\mathrm{pr}}\left(\frac{U_n^{4/3}  \|\nabla_G\theta^*\|_1^{2/3}    }{n^{2/3}}\,+\, \frac{U_n^2 }{n} \right),
		\end{equation}
		and the same upper bound holds for an estimator that can be found in linear time. Thus, we generalize the conclusions in Theorems 2 and 3 from \cite{padilla2016dfs} to hold with noise beyond sub-Gaussian noise. For instance, for sub-Exponential noise the term $U_n$ would satisfy $U_n = O(\log n)$.  
		\item For the $d$-dimensional grid graph with $d>1$ and $n$ nodes, we show that 
		\begin{equation}
			\label{eqn:rate0.2}
			\frac{\|\hat{\theta}-\theta^*\|^2}{n}\,=\, O_{\mathrm{pr}}\left(\frac{U_n  \|  \nabla_G\theta^*\|_1   }{n}    \,+\,\frac{U_n^2}{n}  \right),
		\end{equation}
		if we disregard logarithmic factors. Thus, under the canonical scaling $\|  \nabla_G\theta^*\|_1    =  O(n^{1-1/d})$, see e.g  \cite{sadhanala2016total}, the upper bound is minimax optimal thereby generalizing the results from \cite{hutter2016optimal} to settings with error distributions that satisfy (\ref{eqn:tail_c}). 

		\item For $K$-nearest neighbor  ($K$-NN) graphs constructed with the assumptions from \cite{padilla2018adaptive}, we show that the fused lasso estimator satisfies that 
		\begin{equation}
			\label{eqn:rate0.3}
			\frac{\|\hat{\theta}-\theta^*\|^2}{n}\,=\, O_{\mathrm{pr}}\left( \frac{U_n }{n^{1/d}}\right),
		\end{equation}
		up to logarithmic factors. Hence, we generalize Theorem 2 from \cite{padilla2018adaptive} to models with more general error distributions. Moreover, if $U_n  = O\{\text{poly}(\log n)\}$ for a polynomial function  $\text{poly}(\cdot)$, then the rate in (\ref{eqn:rate0.3}) is minimax optimal for classes of  bounded variation. 

	\end{enumerate}
	\item  In the heteroscedastic setting, where some of the $v_i ^*$ can be different from each other, we are the first in the literature to develop an estimator for the vector of variances $v^* \in \mathbb{R}^n$ in general graph structured models. Specifically,
	we provide a simple estimator $\hat{v} $ of $v^* $ that can be found with the same computational complexity as that of  $\hat{\theta}$. For the proposed estimator we show that 
	there exists $U_n^{\prime} $ satisfying  $U_n^{\prime}   =O(1+U_n^2)$  for which the upper bounds in (\ref{eqn:rate0.1})--(\ref{eqn:rate0.3}) hold replacing $\|\hat{\theta}-\theta^*\|^2/n$ with  $\|\hat{v}-v^*\|^2/n$ and 
	$\|\nabla_G \theta^*\|_1  $ with $\|\nabla_G \theta^*\|_1  +  \|\nabla_G v^*\|_1 $. Our results hold with the same assumptions that those in 2), but with a stronger moment condition presented in Theorem \ref{thm5}. Moreover, when $U_n = O\{\text{poly}(\log n)\}$ and $\| \nabla_G \theta^*\|_1 \asymp  \|\nabla_G v^*\|_1$, our variance estimator attains, up to log factors, the same rates as $\hat{\theta}$ attains in (\ref{eqn:rate0.1})--(\ref{eqn:rate0.3}). We also show, save by logarithmic factors,  that the upper bounds in the case of grid and $K$-NN graphs are minimax optimal, see Lemmas \ref{lem:lower1}--\ref{lem:lower2}.
\end{enumerate}

\subsection{Other related work }

Besides total variation, other popular methods for mean estimation in graph problems include   kernels based methods  \citep{smola2003kernels,zhu2003semi,zhou2005learning},  wavelet  constructions  \citep{crovella2003graph,coifman2006diffusion,gavish2010multiscale,hammond2011wavelets,sharpnack2013detecting,shuman2013emerging},  tree based estimators \citep{donoho1997cart,blanchard2007optimal,chatterjee2021adaptive,madrid2021lattice}, and $\ell_0$-regularization approches \citep{fan2018approximate,yu2022optimal}.

As for variance estimation, some methods estimate the conditional mean and then compute the residuals before  proceeding to estimate the conditional variance. Some of these approaches include \cite{hall1989variance,fan1998efficient}. Other methods, as it is the case of our proposed approach, do not consider the residuals. Some of  such works include \cite{wang2008effect,cai2009variance}, which studied rates of convergence for univariate nonparametric regression with Lipschitz classes. \cite{cai2008adaptive} considered a wavelet thresholding approach also for univariate data. More recently, \cite{shen2020optimal} considered univariate H\"{o}lder functions classes  and some homoscedastic multivariate settings.





Finally, total variation denoising methods have become popular as a tool to tackle different statistics and machine learning problems.  \cite{ortelli2020adaptive} and \cite{sadhanala2019additive} studied additive models,  \cite{padilla2018graphon}  proposed a method for graphon estimation,  \cite{padilla2021causal} considered a method for interpretable causal inference,  \cite{dallakyan2022fused} developed a method for covariance matrix estimation.  More recently, \cite{tran2022generalized} proposed an $\ell_1$ + $\ell_2$  based penalty over graphs called the Generalized
Elastic Net aimed for problems where  features are associated with the nodes of graph.


\subsection{Notation}

Throughout, for a vector $v \in \mathbb{R}^n$, we define its $\ell_1$, $\ell_2$ and $\ell_{\infty}$ norms as $\|v\|_1 =\sum_{i=1}^{n} \vert v_i\vert$,   $\|v\|= (\sum_{i=1}^{n} \vert v_i\vert)^{1/2}$, $\|v\|_{\infty} =  \max_{i=1,\ldots,n} \vert v_i\vert$, respectively. Given a sequence of random variables $X_n$ and a squence of positive numbers $a_n$, we write $X_n = O_{\mathrm{pr}}(a_n)$ if for every $t>0$ there exists $C>0$ such that 
$\mathrm{pr}(  X_n > C a_n) < t$ for all $n$. For two sequences $a_n$ and $b_n$ we write  $a_n \asymp b_n$ if there exists positive constants $c$ and $C$ such that $c a_n \leq b_n \leq C a_n$ for all $n$.  A $d$-dimensional grid graph of size $n =  m^d$ is constructed  as the $d$-dimensional lattice $\{1,\ldots,m\}^d$, where  $i,j \in \{1,\ldots,m\}^d$ are connected if and only if $\|i-j\|_1 =1$. We also write $\mathbf{1} =  (1,\ldots,1 )^{\top}\in \mathbb{R}^n$ and $\bar{a} =  \frac{1}{n} \sum_{i=1}^{n} a_i $ for a vector $a \in \mathbb{R}^n$. For a function $f  \,:\,  [0,1]^d \rightarrow \mathbb{R}$,  we write $\| f \|_2  :=   \sqrt{\int_{ [0,1]^d     }  f(x)^2 dx  }  $.

\subsection{Outline}

The rest of the paper is organized as follows. In Section \ref{sec:homo} we introduce the estimator for the homoscedastic case and show an upper bound on its performance. In Section \ref{sec:hete} we start by defining our estimator for the heteroscedastic case. In Section \ref{sec:gen} we provide a general upper bound for the fused lasso estimator. Then we apply our new result in Section \ref{sec:variance} to obtain general upper bounds for our variances estimator in the heteroscedastic case, and conclude by providing matching lower bounds. Section \ref{sec:experiments}  contains numerical evaluations of the proposed methods in both simulated and real data.    All the proofs are deferred to the Appendix.


\section{Homoscedastic case}
\label{sec:homo}

This section considers the homoscedastic case, which means that $v_i^*=v_0^*$  for all $i$. We now give a motivation on how an estimator of the variance  in the homoscedastic setting can be used for model selection of (\ref{eqn:gfl}). Specifically, if $\hat{v}$ is an estimator of  $v_0^*$, then following  \cite{tibshirani2012degrees} and denoting $\hat{\theta}_{\lambda} $ the solution to (\ref{eqn:gfl}), we can define 
\[
\widehat{\mathrm{Risk}}(\lambda)  :=   \|  y - \hat{\theta}_{\lambda}\|^2 \,+\, 2 \hat{v} \widehat{\text{df}}_{\lambda},  
\]
where $\widehat{\text{df}}_{\lambda}$ is an estimator of the  degrees of freedom corresponding to the model associated with $\hat{\theta}_{\lambda}$,  see Equation (8) in \cite{tibshirani2012degrees}.  In fact, based on Equation 4 from  \cite{tibshirani2012degrees}, $\widehat{\text{df}}_{\lambda}$ can be taken as the number of connected components in $G$ induced by $\hat{\theta}_{\lambda}$ when removing the edges $(i,j)\in E$ satisfying $(\hat{\theta}_{\lambda})_i \neq (\hat{\theta}_{\lambda})_j $.  Hence, in practice one can choose the value of $\lambda$ that minimizes  $\widehat{\mathrm{Risk}}(\lambda)  $ or some variant of it, such as the one we consider in Section \ref{sec:tuning}. Therefore, for model selection, it is convenient to estimate $v_0^*$.

Before providing our estimator of  $v_0^*$, we state the statistical assumption needed to arrive at our main result of this section.

\begin{assumption}
	\label{as2} We assume that $v_i^* =v_0^*$ for $i=1,\ldots,n$, and $$  \underset{i=1,\ldots,n}{\max}  \mathbb{E}(\epsilon_i ^4)  = O(1).$$
\end{assumption}

Thus, we simply require that the fourth moments of the errors are uniformly bounded. We are now in positition to define our estimator. This is given as 
\begin{equation}
	\label{eqn:dfs_est}
	\hat{v} \,:=\,\displaystyle \frac{1}{2(  \lfloor{n/2\rfloor}-1)} \sum_{i=1}^{  \lfloor{n/2\rfloor}-1} \{ y_{ \sigma(2i) }  -    y_{ \sigma(2i-1) }\}^2,
\end{equation}
where $\sigma(1),\ldots,\sigma(n)$ are the  nodes in $G$ visited in order according to the DFS algorithm in the graph $G$, see \cite{tarjan1972depth}.  The DFS algorithm  proceeds as follows:

\noindent \textit{Procedure DFS}$(G,v)$: \\
\hspace{0.2in}$\textbf{Step 1:}$   Label $v$ as discovered.\\ 
\hspace{0.2in}$\textbf{Step 2:}$ For all $w$  such that $(w,v)  \in E$ do \\
If vertex $w$ is not label  then recursively call \textit{DFS}$(G,w)$.

Figure \ref{fig:dfs}  shows an example of a graph and a potential run of DFS.  Clearly,  by construction of DFS, the function $\sigma $ is a bijection from $\{1,\ldots,n\} $  onto itself, and the DFS ordering is not unique. Hence, we propose to select the DFS by randomly choosing the start of the algorithm.


\begin{figure}[t!]      
	\begin{center}
		\includegraphics[width=3in,height= 3in]{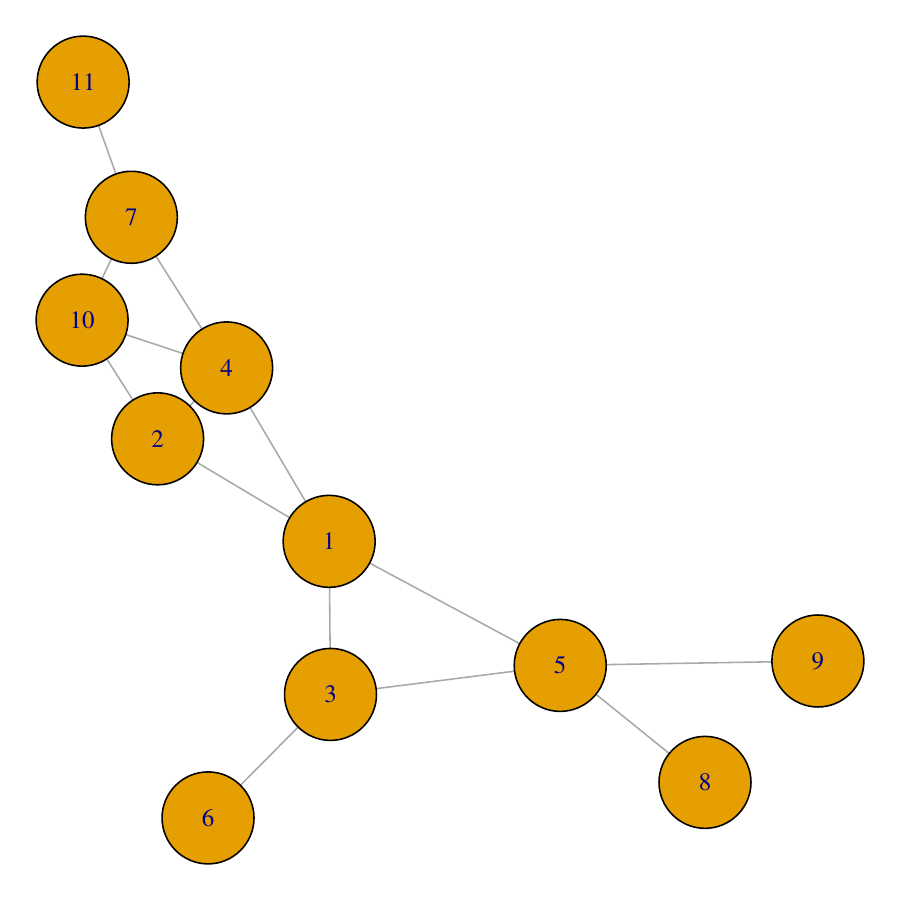}
	\end{center}
	\caption{\label{fig:dfs}An example of a graph $G$. Running DFS starting with the node $1$  produces the ordering $1,3,6,5,8,9,4,7,11,2, 10$. }
\end{figure}

Notice that the total computational complexity for computing $\hat{v}$ is $O(n+\vert E\vert)$, which comes from computing the DFS order.  Moreover, the estimator $\hat{v}$ does not require any tuning parameters to be specified.

The construction in (\ref{eqn:dfs_est})  can be motivated as follows. First, recall that by Lemma 1 in  \cite{padilla2016dfs}, it holds 
\[
\sum_{i=1}^{n-1} \vert  \theta_{\sigma(i)}^*- \theta_{\sigma(i+1)}^*\vert  \,\leq \, 2 \|\nabla_G \theta^*\|_1.
\]
Hence, if $\|\nabla_G \theta^*\|_1$ is small relative to $n$, then   the signal $\theta^*$ is well behaved in the order given by DFS. Our resulting  estimator defined in (\ref{eqn:dfs_est}) is then obtained by applying the idea of taking differences from \cite{rice1984bandwidth}, see also \cite{dette1998estimating} and \cite{tong2005estimating}.


\begin{theorem}
	\label{thm2}
	Suppose that Assumption \ref{as2} holds and   $\|\epsilon\|_{\infty} =O_{\mathrm{pr}}(U_n)$ for some positive sequence $U_n$. Then
	\begin{equation}
		\label{eqn:rate-0}
		\vert  v_0^* - \hat{v}\vert   \,:=\, O_{\mathrm{pr}}\left[   \frac{v_0^*}{n^{1/2} }  +  \frac{\{ U_n (v_0^*)^{1/2}  + \|\theta^*\|_{\infty}  \}   \| \nabla_G \theta^*\|_1  }{n}   \right].
	\end{equation}
\end{theorem}

\begin{remark}
	\label{rem1}
	Consider the case where  $G$ is  the chain graph, and suppose that $\theta_i^* = f^*(i/n)$ for $i=1,\ldots,n$, for a  function  $f^* \,:\,[0,1] \rightarrow \mathbb{R}$, bounded   and of bounded total variation. Thus, $f^*\in \mathcal{C}$ where
	\[
	\mathcal{C}  := \{  f \,:\,[0,1] \rightarrow  \mathbb{R}\,:\,      \|f\|_{\infty} \leq C_1,   \,\mathrm{TV}(f)  \,\leq \, C_2       \},
	\]
	where $C_1$ and $C_2$ are positive constants, and $\mathrm{TV}(f) $ is the total variation defined as
	\[
	\mathrm{TV}(f)   \,:=\, \underset{  0\leq a_1 <  \ldots <a_m  \leq 1, \,\,m \in  \mathbb{N} }{\sup}\,   \sum_{j=1}^{m-1}\vert f(a_j) - f(a_{j+1}) \vert,  \,
	\]
	see the discussion about functions  of bounded total variation in \cite{tibshirani2014adaptive}. Then 
	$\max\{   \|\theta^*\|_{\infty} ,   \| \nabla_G \theta^*\|_1 \} = O(1)$. Hence, provided that $v_0^* =O(1)$  and $ U_n = O\{\mathrm{poly}(\log n )\} $ for 
	$\mathrm{poly}(\cdot)$ some  polynomial function, we obtain that
	\begin{equation}
		\label{eqn:chain_homo}
		\vert  v_0^* - \hat{v}\vert   \,:=\, O_{\mathrm{pr}}(   n^{-1/2} ),
	\end{equation}
	if we ignore logarithmic factors. Therefore, from Proposition 3 in  \cite{shen2020optimal}, the rate  in (\ref{eqn:chain_homo}) is minimax optimal in the class  $\mathcal{C}$. 	This follows since $\mathcal{C}$ is a larger class than that considered in Proposition 3 in  \cite{shen2020optimal} for the case corresponding to bounded Lipschitz  continuous functions.
\end{remark}

\begin{remark}
	\label{rem2}
	If $G$ is the 2D grid graph, then it is well known that $\| \nabla_G  \theta^*\|_1  \asymp  n^{1/2}$ is the canonical scaling, see \cite{sadhanala2016total} and our discussion in Appendix \ref{sec:canonical}. Hence, if  $\max\{v_0^*,\| \theta^*\|_{\infty}\} = O(1)$, and $ U_n = O\{\mathrm{poly}(\log n )\} $, then (\ref{eqn:chain_homo})  holds. Therefore, as in Remark \ref{rem1}, by Proposition 3 from  \cite{shen2020optimal},  $\hat{v}$ attains minimax rates when $\theta^*$ is in  the class
	\[
	\{  \theta \,:\,    \|\theta \|_{\infty} \leq C_1,\,    \|\nabla_G \theta \|_{1} \leq C_1   n^{1/2} \}
	\]
	for positive constants $C_1$ and $C_2$.
	
\end{remark}

Finally, for a general graph $G$,  if the graph does capture smoothness of the true signal in the sense that $ U_n\|\nabla_G \theta^*\|_1 / n \rightarrow 0$, then, as long as $\max\{v_0^*,  \| \theta^*\|_{\infty} \} = O(1)$, the upper bound in Theorem \ref{thm2}  shows that $\hat{v}$ is a consistent estimator of $v_0^*$.

\section{Heteroscedastic case}
\label{sec:hete}

We now study the heteroscedastic setting. Hence, we do not longer require that all the variances are equal. To estimate the signal $v^* \in \mathbb{R}^n$,  we recall the identity 
\[
v_i^* \,=\,  \mathrm{var}(y_i)   \,= \,   \mathbb{E}(y_i^2)  -   \{\mathbb{E}(y_i)\}^2.
\]

\begin{figure}
	\begin{center}
		\includegraphics[width =2.5in,height = 2.3in]{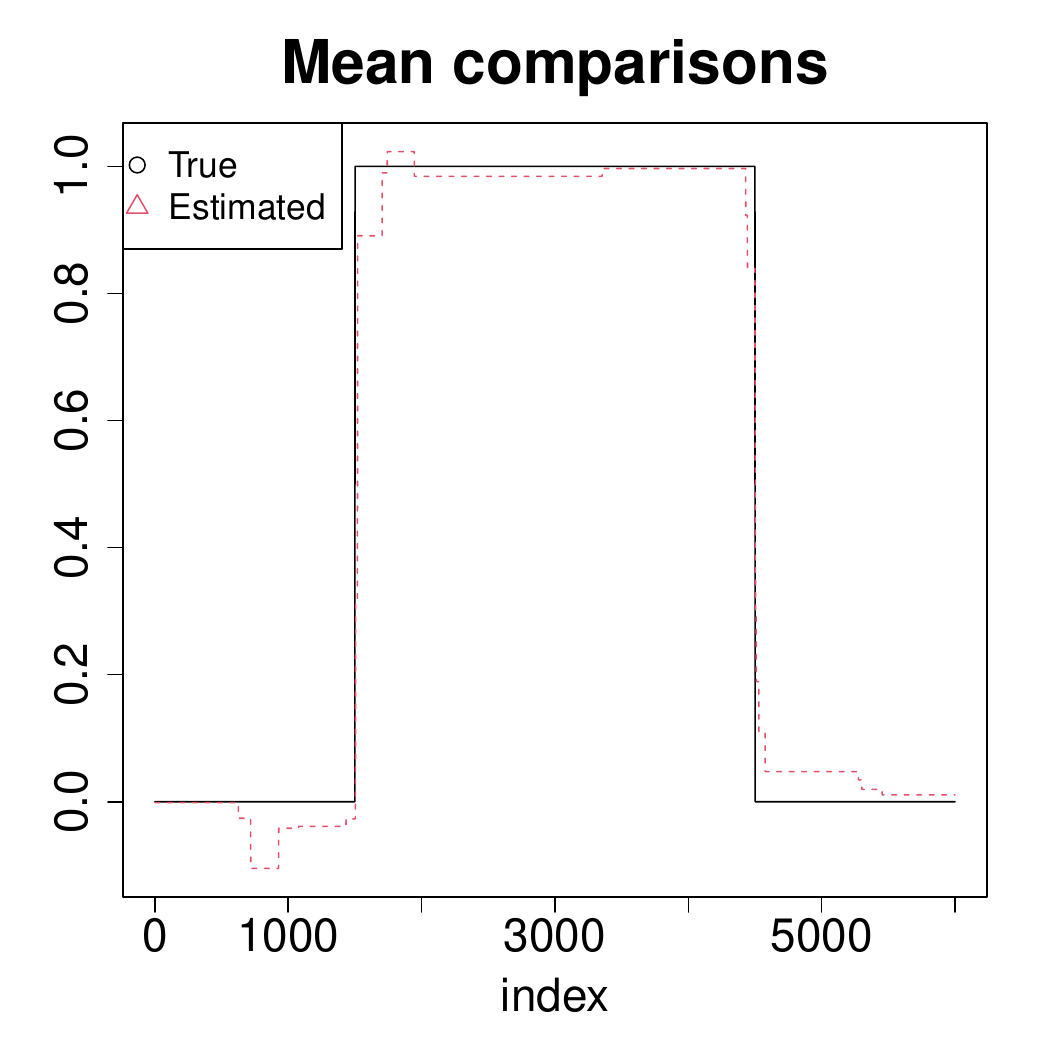}
		\includegraphics[width =2.5in,height = 2.3in]{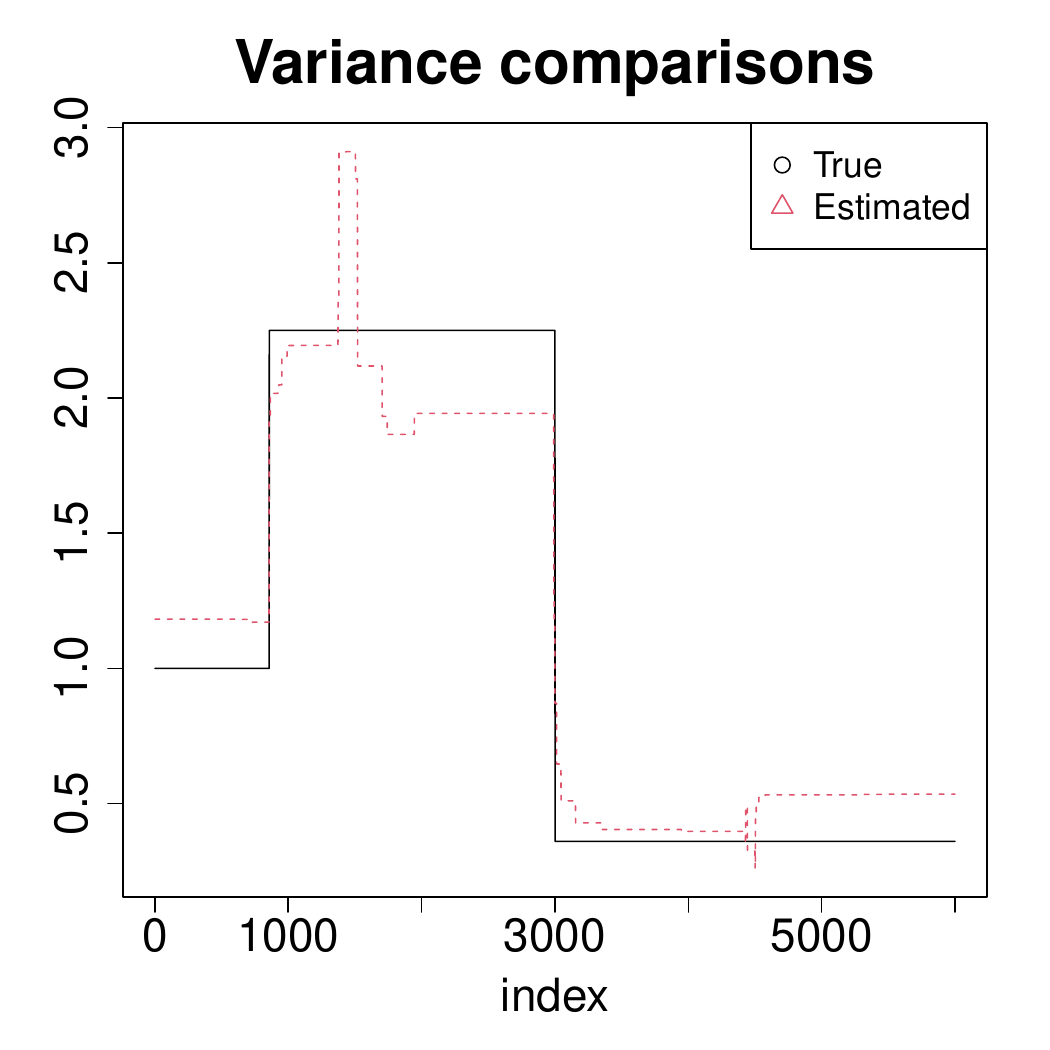}
		\caption{ 		\label{fig1} The left panel shows comparisons of the true and estimated means for Example \ref{ex1} in the text. The right panel shows the corresponding variance comparisons. }
	\end{center}
\end{figure}

Therefore, it is natural to estimate $v^*_i$ with 
\begin{equation}
	\label{eqn:def}
	\hat{v}_i \,=\,  \hat{\gamma}_i - (\hat{\theta}_i)^2,
\end{equation}
where $\hat{\gamma}_i$ is an estimator of $\gamma_i^* := \mathbb{E}(y_i^2)$, and $\hat{\theta}$  is the fused lasso estimator defined in (\ref{eqn:gfl}).  As an estimator for $\gamma^*$, we propose 
\begin{equation}
	\label{eqn:estimator2}
	\hat{\gamma} \,:=\,\underset{\gamma \in  \mathbb{R}^n}{\arg \min}\left\{ \frac{1}{2}\sum_{i=1}^{n}(y_i^2 - \gamma_i)^2   +   \lambda^{\prime}  \sum_{ (i,j)\in E } \vert\gamma_i -\gamma_j\vert   \right\}
\end{equation}
for a tuning parameter $\lambda^{\prime} >0$.

Notice that $\hat{v}$ can be found with the same order of computational cost that it is required for finding $\hat{\theta}$. In practice, this can be done using the algorithm from \cite{chambolle2009total}. As  for parameter tuning, we give details about choosing $\lambda^{\prime}$ in practice in Section \ref{sec:tuning}.

To illustrate the behavior of the estimator defined in (\ref{eqn:def})--(\ref{eqn:estimator2}), we now consider a simple numerical  example. More comprehensive evluations are given in Section \ref{sec:experiments}.

\begin{example}
	\label{ex1}
	We set  $n = 6000$ and generate data according to the model given by (\ref{eqn:model}) with   $\epsilon_i  \overset{\text{ind}}{\sim} N(0,1)$  for $i=1,\ldots,n$, and   $\theta^*,v^* \in  \mathbb{R}^n$ satisfying 
	\[
	\theta_i^* \,=\,\begin{cases}
		1 & \text{if} \,\,\,n/4 < i \leq  3n/4,\\
		0 & \text{otherwise},
	\end{cases}
	\]
	and 
	\[
	v_i^* \,=\,\begin{cases}
		
		1 & \text{if} \,\, i \leq  \lfloor{n/7\rfloor},\\
		1.5^2 & \text{if} \,\,\lfloor{n/7\rfloor}< i \leq  n/2,\\
		0.6^2 & \text{otherwise.}
	\end{cases}
	\]
	Given the data $\{y_i\}_{i=1}^n$, we run the estimator defined in (\ref{eqn:def})--(\ref{eqn:estimator2})   with tuning parameter choices as discussed in Section \ref{sec:tuning}. The results are displayed in Figure \ref{fig1}, where we see that the estimated means and variances are reasonably close to the corresponding true parameters.
\end{example}

\subsection{A general result for fused lasso estimator}
\label{sec:gen}

Before presenting our main result for the estimator  $\hat{v}$ defined in (\ref{eqn:def})--(\ref{eqn:estimator2}), we provide a general upper bound for the fused lasso estimator that holds under very weak assumptions and generalizes existing work in \cite{hutter2016optimal}, \cite{padilla2016dfs} and  \cite{padilla2018adaptive}.

\begin{theorem}
	\label{thm4}
	Consider data $\{o_i \}_{i=1}^n$  generated as $o_i = \beta_i^*+\varepsilon_i$ for some $\beta^*\in  \mathbb{R}^n$ and $\varepsilon_1,\ldots,\varepsilon_n$ independent random variables satisfying  satisfying $\mathbb{E}(\varepsilon_i) = 0$ for $i=1,\ldots,n$, and  $\max_{i=1,\ldots,n} \mathbb{E}(\varepsilon_i^4) =  O(1)$. Let  $\hat{\beta}$ be defined as 
	\begin{equation}
		\label{eqn:glf3}
		\hat{\beta} \,:=\,\underset{\beta \in \mathbb{R}^n}{\arg \min}\left\{ \frac{1}{2}\sum_{i=1}^{n}(o_i - \beta_i)^2   +   \lambda\sum_{ (i,j)\in E } \vert\beta_i -\beta_j\vert   \right\}.
	\end{equation}
	The following results hold: 
	
	\begin{enumerate}
		\item \textbf{General graphs.} For any connected graph $G$, if  for a positive sequence $U_n$ holds that 
		\begin{equation}
			\label{eqn:c1}
			n^{1/2}  U_n^{-1} \{\log (en )\}^{-1/2}\,\underset{i=1,\ldots,n}{\max}   \{\mathrm{pr}(\vert  \varepsilon_i\vert >  U_n) \}^{1/4}\,\rightarrow \, 0,
		\end{equation}
		then 
		\begin{equation}
			\label{eqn:rate1}
			\frac{\|\hat{\beta}-\beta^*\|^2}{n}\,=\, O_{\mathrm{pr}}\left\{\frac{U_n^{4/3}  (\log n)^{1/3} \|\nabla_G\beta^*\|_1^{2/3}    }{n^{2/3}}\,+\, \frac{U_n^2 \log n}{n} \right\},
		\end{equation}
		for a choice of $\lambda$ satisfying $\lambda \,\asymp \,U_n^{4/3}  (n\log n)^{1/3} \|\nabla_G\beta^*\|_1^{-1/3}.$
		\item  \textbf{Grid graphs.} Let  $G$ be the $d$-dimensional grid graph  with $d>1$. Suppose that for a positive sequence $U_n$ we have that 
		\begin{equation}
			\label{eqn:c2}
			\underset{i=1,\ldots,n}{\max}      U_n^{-1}\{  \mathrm{pr}(\vert  \varepsilon_i\vert >  U_n) \}^{1/4} \,\rightarrow \,0.
		\end{equation}
		Then 	there exists a choice of $\lambda$ satisfying 
		\[
		\lambda \,\asymp\,  U_n    \phi_n \,+\,   U_n \|  \nabla_G\beta^*\|_1^{-1/2}
		\]
		such that 
		
		\begin{equation}
			\label{eqn:rate2}
			\frac{\|\hat{\beta}-\beta^*\|^2}{n}\,=\, O_{\mathrm{pr}}\left(\frac{U_n  \phi_n  \|  \nabla_G\beta^*\|_1   }{n}    \,+\,\frac{U_n^2}{n}  \right),
		\end{equation}
		where $\phi_n  = 	C (\log n)^{1/2}  $ if $d=2$ and  $	\phi_n  = C $ otherwise, for some constant $C>0$.		
		for some constant $C>0$. 
		\item \textbf{K-NN graphs.} Suppose that in addition to the measurements $(o_1,\ldots,o_n)^{\top}$ we are also given covariates $\{x_i\}_{i=1}^n \subset \mathcal{X}$, where $x_i $ corresponds to $o_i$, and $\mathcal{X}$ is a metric space with metric $\mathrm{dist}(\cdot)$. Suppose that $\{(x_i,o_i)\}_{i=1}^n$ satisfy the		
		assumptions from \cite{padilla2018adaptive}, see Appendix \ref{sec:assump}. In particular, $\mathcal{X}$ is homeomorphic to $[0,1]^d$. In addition, assume that $K \asymp \log^{1+2r} n$ for some $r >0$ in the construction of the $K$-NN graph $G$, and 
		that for a positive sequence $U_n$ we have that
		\begin{equation}
			\label{eqn:c3}
			n^{1/2}    U_n^{-1} K^{-1/2}\,\underset{i=1,\ldots,n}{\max}  \{  \mathrm{pr}(\vert  \varepsilon_i\vert > U_n) \}^{1/4}\rightarrow  0.
		\end{equation}
		Consider\[
		\lambda\,\asymp\, \| \nabla_G \beta^*\|_1 ^{-1}\left[ (\mathrm{poly}(\log n)n^{1-1/d}  U_n)^{1/2}\,+\,  K^{1/2}  U_n \,+\,  (U_n  K^{1/2}  \mathrm{poly}(\log n)\,n^{1-1/d}\phi_n)^{1/2} \right]^2
		\]
		where $\mathrm{poly}(\cdot)$ is a polynomial function, and $\phi_n$ is defined as in the case of grid graphs above. Then
		\begin{equation}
			\label{eqn:rate3}
			\frac{\|\hat{\beta}-\beta^*\|^2}{n}\,=\, O_{\mathrm{pr}}\left\{ \frac{U_n  \mathrm{poly}_2(\log n)  }{n^{1/d}}\right\},
		\end{equation}
		where $ \mathrm{poly}_2(\cdot)$ is another polynomial function.
	\end{enumerate}
	
\end{theorem}

\begin{remark}
	\label{rem3}
	Let us now elaborate on (\ref{eqn:c1}), (\ref{eqn:c2}) and (\ref{eqn:c3}).   Suppose, for instance, that $\varepsilon_i$ is  sub-Exponential(a), for some constant $a>0$. Then the usual sub-Exponential tail inequality
	can be written as
	\[
	\mathrm{pr}(\vert \varepsilon_i \vert >t) \,\leq \, 2\exp(-t/a),\,\,\,\,  \text{for all }  t>0,
	\]
	see for instance Proposition 2.7.1 in \cite{vershynin2018high}. Hence, taking $U_n =   4a\log n$ it follows that (\ref{eqn:c1}), (\ref{eqn:c2}) and (\ref{eqn:c3}) immediately hold. More generally, if 
	\[
	\mathrm{pr}(\vert \varepsilon_i \vert >t) \,\leq \, c_1\exp(-t^{\alpha}/c_2),\,\,\,\text{for all }  t>0,
	\]
	for positive constants $c_1,c_2,$ and $\alpha$, then taking  $U_n = 4 c_2 (\log n)^{1/\alpha}$, we obtain that (\ref{eqn:c1}), (\ref{eqn:c2}) and (\ref{eqn:c3}) all hold.
\end{remark}

\begin{remark}
	\label{rem4}
	
	Remark \ref{rem3} gives a family of examples where $U_n$ can be taken as a power function of  $\log n $. More generally, if $U_n = O\{\mathrm{poly}(\log n)\}$, for a polynomial function $\mathrm{poly}$, then up to logarithmic factors, Theorem \ref{thm4} gives the same rates as in several existing works on the fused lasso, but now we allow for more general error distributions than sub-Gaussian. Specifically:
	\begin{enumerate}
		\item For a connected graph $G$, (\ref{eqn:rate1}) generalizes the upper bound in Theorem 3 in \cite{padilla2016dfs}. Moreover, the same upper bound in (\ref{eqn:rate1}) also holds if we replace the fused lasso estimator $\hat{\beta}$ (\ref{eqn:glf3}) with the DFS fused lasso estimator from \cite{padilla2016dfs}.
		\item For a $d$-dimensional grid graph $G$, the rate in (\ref{eqn:rate2})  matches that in Corollary 5  from \cite{hutter2016optimal}.
		\item For nonparameteric regression, (\ref{eqn:rate3}) gives the same minimax rate as Theorem 2 in \cite{padilla2018adaptive} for classes of piecewise functions. 
	\end{enumerate}
\end{remark}

\begin{remark}
	\label{remark:challenges}
	As stated before, Theorem \ref{thm4} is the first result for fused lasso in general graph models where the error terms can be non-Sub-Gaussian, yet the estimator still uses the $\ell_2$ loss. The    proof of Theorem \ref{thm4}  relies on Theorem \ref{thm3} 	in Appendix \ref{sec:general_bound}.  The latter basically alllow us to control the quantity  $\mathrm{pr}( \| \hat{\beta} -\beta^*\|  >\eta)   $, for $\eta>0$,  in terms of the process 
	\begin{equation}
		\label{eqn:key_quantity}
		\frac{1}{\eta^2}	\mathbb{E}\left[     \underset{\beta \in  \mathbb{R}^n \,:\, \|\beta - \beta^*\|\leq \eta , \, \| \nabla_G\beta\|_1 \lesssim \|\nabla_G\beta^*\|_1 }{\sup}  \sum_{i=1}^n   \xi_i\varepsilon_i  1_{  \{\vert  \varepsilon_i\vert \leq U_n\}   } (\beta_i - \beta_i^*)    \right]
	\end{equation}
	where  $\xi_1,\ldots,\xi_n$ are independent Rademacher random variables independent of $\{\varepsilon_i\}_{i=1}^n$. This general result holds for arbitrary sequences $U_n$ and it is key given that  the random variables $\xi_i\varepsilon_i  1_{  \{\vert  \varepsilon_i\vert \leq U_n\}   } $  for $i=1,\ldots, n$ are uniformly bounded. Hence, we do not need to control the standard process  
	\[
	\mathbb{E}\left\{    \underset{\beta \in \Lambda   \,:\, \|\beta - \beta^*\|\leq \eta , \, \| \nabla_G\beta\|_1 \lesssim \|\nabla_G\beta^*\|_1 }{\sup}  \varepsilon^{\top}(\beta - \beta^*) \right\}
	\]
	as it is the case in the analysis in   \cite{guntuboyina2020adaptive}, which is only able to handle sub-Gaussian random variables $\varepsilon_i$, $i=1,\ldots,n$. With this challenge overcome, the proof of Theorem \ref{thm3}  continues by controling additional terms that account for the case $\xi_i\varepsilon_i  1_{  \{\vert  \varepsilon_i\vert >U_n\}   } $  for $i=1,\ldots,n$ of the form
	\begin{equation}
		\label{eqn:key_quantity2}
		\frac{ n^{1/2}   \underset{i=1,\ldots,n}{\max}   \{\mathbb{E}( \varepsilon_i^4  )\}^{1/4}      \{  \mathrm{pr}(\vert  \varepsilon_i\vert > U_n) \}^{1/4}}{\eta}.
	\end{equation}
	With Theorem \ref{thm3} in hand, the proof of  Theorem \ref{thm4}  continues by deriving upper bounds for the quantities  (\ref{eqn:key_quantity})  and (\ref{eqn:key_quantity2}). The analysis for (\ref{eqn:key_quantity}) is done customizing for general graphs, grid graphs, and $K$-NN graphs.  
\end{remark}

\subsection{Fused lasso for variance estimation }
\label{sec:variance}

We are now ready to state our main result regarding the estimator $\hat{v}$  defined in (\ref{eqn:def})--(\ref{eqn:estimator2}).  Notably, our result shows that the estimator $\hat{v}$  enjoys similar properties as the original fused lasso in general graphs, $d$-dimensional grids, and $K$-NN graphs. The conclusion of our result  follows from an application of Theorem \ref{thm4} to  $\hat{\theta}$ defined in (\ref{eqn:gfl}) and $\hat{\gamma}$ defined in  (\ref{eqn:estimator2}).

\begin{theorem}
	\label{thm5}
	Consider data $\{y_i\}_{i=1}^n$ generated as in (\ref{eqn:model}) and suppose that $\mathbb{E}(\epsilon_i^8) < \infty$. Then the estimator $\hat{v}$  satisfies the following.
	\begin{itemize}
		\item  \textbf{General graphs.} Let $G$ be any connected graph and assume  that  (\ref{eqn:c1})  holds with $\{\epsilon_i\}_{i=1}^n$ instead of $\{\varepsilon_i\}_{i=1}^n$. Then for choices of $\lambda$ and $\lambda^{\prime} $  satisfying
		\[
		\lambda \,\asymp \,U_n^{4/3}  (n\log n)^{1/3} \|\nabla_G\theta^*\|_1^{-1/3} 
		\]
		and
		$\lambda^{\prime} \asymp  \{\|v^*\|_1^{1/2}\|\theta^*\|_{\infty} U_n   \,+\,  \| v^*\|_{\infty}(1+U_n^2)  \}^{4/3}(n\log n)^{1/3} \|\nabla_G \gamma^*\|_1^{-1/3}$, 
		we have that 
		\begin{equation}
			\label{eqn:rate4}
			\begin{array}{lll}
				\displaystyle 	\frac{1}{n}\|\hat{v} -v^*\|^2 &\,=\,&	\displaystyle  O_{\mathrm{pr}}\bigg\{\frac{ (\|\theta^*\|_{\infty}^2+1)(U_n^{\prime})^{4/3}  (\log n)^{1/3} (\|\nabla_Gv^*\|_1+ \|\theta^*\|_{\infty}\|\nabla_G\theta^*\|_1)^{2/3}    }{n^{2/3}}\,+\,\\
				& & \displaystyle \,\,\,\,\,\,\,\,\,\,\,\frac{(\|\theta^*\|_{\infty}^2+1)(U_n^{\prime})^2 \log n}{n} \bigg\}
			\end{array}
		\end{equation}
		where 
		\begin{equation}
			\label{eqn:uprime}
			U_n^{\prime}\,:=\, (2\|  v^*\|_{\infty}^{1/2}\|\theta^*\|_{\infty}+1) U_n+\, \|  v^*\|_{\infty}U_n^2 \,+\, \| v^*\|_{\infty}.
		\end{equation}
		\item   \textbf{Grid graphs.} Let  $G$ be the $d$-dimensional grid graph  with $d>1$.  Suppose that  the sequence $\{\epsilon_i\}_{i=1}^n$ satisfies (\ref{eqn:c2}).
		Then   there exists tuning parameter choices satisfying 
		\[
		\lambda \,\asymp\,  U_n    \phi_n \,+\,   U_n \|  \nabla_G\theta^*\|_1^{-1/2},\,\,\,\text{and}\,\,\,\lambda^{\prime} \,\asymp\,  U_n^{\prime}    \phi_n \,+\,   U_n^{\prime} \|  \nabla_G\gamma^*\|_1^{-1/2}
		\]
		for which 
		\begin{equation}
			\label{eqn:rate5}
			\frac{\|\hat{v}-v^*\|^2}{n}\,=\, O_{\mathrm{pr}}\left\{\frac{( \|\theta^*\|_{\infty}^2+1)U_n^{\prime}\phi_n \left( \|  \nabla_G v^*\|_1  +  \|\theta^*\|_{\infty} \|  \nabla_G\theta^*\|_1 \right)  }{n}    \,+\,\frac{ ( \|\theta^*\|_{\infty}^2+1)(U_n^{\prime})^2}{n}  \right\},
		\end{equation}
		with $U^{\prime}_n$ as in (\ref{eqn:uprime}) and $\phi_n$ as in  Theorem \ref{thm4}.

		\item  \textbf{$K$-NN graphs.} Suppose that in addition to the measurements $(y_1,\ldots,y_n)^{\top}$ we are also given covariates $\{x_i\}_{i=1}^n \subset \mathcal{X}$, where $x_i $ corresponds to $y_i$, and $\mathcal{X}$ is a metric space with metric $\mathrm{dist}(\cdot)$. Suppose that $\{(x_i,y_i)\}_{i=1}^n$ satisfy the		
		assumptions from \cite{padilla2018adaptive} stated in Appendix \ref{sec:assump}.  In addition, assume that $K \asymp \log^{1+2r} n$ for some $r >0$ in the construction of the $K$-NN graph $G$, and  (\ref{eqn:c3}) holds for $\{\epsilon_i\}_{i=1}^n$. 
		Then for  choices of $\lambda$ and  $\lambda^{\prime}$ satisfying 
		\[
		\lambda\,\asymp\, \| \nabla_G \theta^*\|_1 ^{-1}\left[ (\mathrm{poly}(\log n)n^{1-1/d}  U_n)^{1/2}\,+\,  K^{1/2}  U_n \,+\,  (U_n  K^{1/2}  \mathrm{poly}(\log n)\,n^{1-1/d}\phi_n)^{1/2} \right]^2
		\]
		and 
		\[
		\lambda^{\prime }\,\asymp\,  \| \nabla_G \gamma^*\|_1 ^{-1}\left[ (\mathrm{poly}(\log n)n^{1-1/d}  U_n^{\prime})^{1/2}\,+\,  K^{1/2}  U_n^{\prime} \,+\,  (U_n^{\prime}  K^{1/2}  \mathrm{poly}(\log n)\,n^{1-1/d}\phi_n)^{1/2} \right]^2
		\]
		for a polynomial function $\mathrm{poly}(\cdot)$, it holds that 
		\begin{equation}
			\label{eqn:rate6}
			\frac{\|\hat{v}-v^*\|^2}{n}\,=\, O_{\mathrm{pr}}\left\{ \frac{ (\|\theta^*\|_{\infty}^2+1)U_n^{\prime} \phi_n  \log^{1+2r} n  }{n^{1/d}}\right\},
		\end{equation}
		with $U^{\prime}_n$ as in (\ref{eqn:uprime}), and with $\phi_n$ as in the previous case of grid graphs.
	\end{itemize}
\end{theorem}

\begin{remark}
	\label{rem5}
	Consider  the setting in which $\max\{  \|\theta^*\|_{\infty}, \|v^*\|_{\infty} \} = O(1)$,  and $U_n = O\{\mathrm{poly}(\log n)\}$, for  $\mathrm{poly}(\cdot)$ a polynomial function. Then, ignoring logarithmic factors, Theorem \ref{thm5} implies the following:
	\begin{enumerate}
		\item For a connected graph $G$ the estimator $\hat{v}$ satisfies 
		\[
		\frac{\| \hat{v} -v^*\|^2}{n} \,=\, O_{\mathrm{pr}}\left\{    \frac{  ( \|\nabla_G \theta^* \|_1 + \|\nabla_G v^* \|_1  )^{2/3} }{n^{2/3}}  \right\}.
		\]
		Hence, for the chain graph and the canonical setting in which  $\max\{\|\nabla_G \theta^* \|_1 , \|\nabla_G v^* \|_1 \} = O(1)$,  the estimator $\hat{v}$ attains  the rate $n^{-2/3}$, which is minimax optimal in the class
		\[
		\{  (v,\theta)\,:\,    \max\{\|\nabla_G \theta^* \|_1 , \|\nabla_G v^* \|_1 \} \,\leq \, C_1,  \, \max\{\|\theta^* \|_{\infty} , \|v^* \|_{\infty} \} \,\leq \, C_1\}
		\]
		for some constants $C_1,C_2>0$, 		see Theorem 4  in \cite{shen2020optimal}.
		
		\item If $d>1$, then for the $d$-dimensional grid graph, we obtain that 
		\[
		\frac{\| \hat{v} -v^*\|^2}{n} \,=\, O_{\mathrm{pr}}\left(    n^{-1/d}  \right),
		\]
		under the canonical scaling, (\cite{sadhanala2016total}, see also Appendix \ref{sec:canonical})  \\
		$\| \nabla_G \theta^*\|_1,  \|\nabla_G v^*\|_1 \,\asymp \, n^{1-1/d}$. 	 Hence, from Lemma \ref{lem:lower1} below, for estimating $v^*$, $\hat{v}$ attains  the minimax rate under the canonical scaling.

		\item For the $K$-NN graph, $\hat{v}$ also attains the rate $n^{-1/d}$ for estimating piecewise Lipschitz functions, thereby maintaining the same adaptivity properties of $\hat{\theta}$ studied in \cite{padilla2018adaptive}. Moreover, from Lemma \ref{lem:lower2},  the rate $n^{-1/d}$ matches the minimax rate for estimating the signal of variances when this is constructed based on the evaluations of a piecewise Lipschitz function. 
	\end{enumerate}
	
\end{remark}

\begin{remark}
	\label{rem6}
	Given that our proposed estimator defined  in (\ref{eqn:def}) is based on estimating $ \gamma_i^* =  \mathbb{E}(  y_i^*) $ and $\theta_i^* =  \mathbb{E}(y_i)$ for $i = 1,\ldots, n$, the first step in the proof of Theorem \ref{thm5}  is to establish Lemma    \ref{lem1}  which states that the total variation of  $\gamma^*$  is bounded by the total variation of the variance signal $v^*$ and the total variation of the mean signal $\theta^*$:
	\[
	\|\nabla_ G   \gamma^* \|_1 \,\lesssim   	\|\nabla_ G   v^* \|_1    \,+\, 	\|\nabla_ G   \theta^* \|_1.
	\]
	Thus, if  $v^*$ and $\theta^*$  both have small total variation along the graph $G$, then the same can be said about the signal $\gamma^*$, which justifies our construction in   (\ref{eqn:estimator2}).  Then the proof of Theorem \ref{thm5} continues by showing that  
	\[
	\frac{1}{n}\|\hat{v} -v^*\|^2    \lesssim \frac{1}{n}\sum_{i=1}^{n}\big(\hat{\gamma}_i - \gamma^*_i \big)^2   \,+\, \frac{1}{n}\sum_{i=1}^{n}   \big(   \hat{\theta}_i-\theta^*_i\big)^2, 
	\]
	and then applying Theorem \ref{thm4} separately with the choices $\beta^* = \theta^* $ and $\beta^* = \gamma^*$. The latter has an additional small challenge, addressed in Lemma \ref{lem2},  concerning the behavior of the tails of the random variables $y_i^2$.
	
\end{remark}


Next we justify  second conclusion in Remark \ref{rem5} concerning the minimax optimality of $\hat{v}$  under canonical scaling. This is presented in the next lemma. 

\begin{lemma}
	\label{lem:lower1}
	Let  $G$ be the $d$-dimensional grid graph and $c \in (0, 1)$ a constant and  let
	\[
	K\,=\, \{   \theta \in \mathbb{R}^n \,:\,   \| \nabla_G\theta\|_1 \leq  c n^{1-1/d} , \,\,\|\theta\|_{\infty} \leq  c\}.
	\]
	Consider the collection of estimators given as
	\[
	\begin{array}{lll}
		\mathcal{F}& \,:=\, &\big\{    v \,:\,\mathbb{R}^n \rightarrow \mathbb{R}^n\,\text{measurable}\big\}.
	\end{array}
	\]
	Then there exists a constant $C >0$ depending on $c$ and $d$ such that 
	\[
	\underset{ \tilde{v} \in \mathcal{F} }{\inf}\,\,\,\underset{ \theta^*,v^* \in K,\,   v_i^* \in (\frac{c^2}{8} ,\frac{3c^2}{8})  }{\sup} \,\mathbb{E}\left(   \frac{1}{n}\|   \tilde{v}(y) - v^*  \|^2  \right) \,\geq \,   \frac{ C    }{n^{1/d}  },
	\]
	for data generated as $y_i =  \theta^*_i   +  \sqrt{v^*_i} \epsilon_i $, with $\epsilon_i  \overset{\text{ind} } {\sim}N(0,1)$,   for $i=1,\ldots,n$.
\end{lemma}

Finally, we conclude our theory section with a lower bound that justifies our assertion that $\hat{v}$ is minimax optimal when using a $K$-NN graph for estimating a piecewise Lipschitz signal. 

\begin{lemma}
	\label{lem:lower2}
	Consider the class of piecewise Lipschitz  functions $\mathcal{F}(L_0)$,  defined in Appendix \ref{sec:assump}, for a constant $L_0  \in (0,1)$.  Suppose that, for functions $f_0, g_0 \in \mathcal{F}$ with $g_0 \geq 0$,  the data are generated as
	\[
	y_i  \,=\,   f_0(x_i)   \,+\,  \sqrt{g_0(x_i)}\epsilon_i,  \,\,\,  
	\]
	where $\epsilon_i \overset{\text{ind} } {\sim}N(0,1)$ and $x_i  \overset{\text{ind} } {\sim} U[0,1]^d$,  for $i=1,$\ldots$,n$. Then for a constant $C>0$ depending on $L_0$, we have that  
	\[
	\underset{ \tilde{g} \,\text{estimator}  }{\inf}\,\,\,\underset{  f_0 , g_0 \in \mathcal{F}(L_0)  \, }{\sup} \,\mathbb{E}\left(   \|   \tilde{g} - g_0  \|^2_2  \right) \,\geq \,   \frac{ C }{n^{1/d}  }.
	\]
	
\end{lemma}

\section{Experiments}
\label{sec:experiments}

\subsection{Heteroscedastic estimator:  Tuning parameters}
\label{sec:tuning}


We now discuss how to choose the tuning parameters for the estimator $\hat{v}$ defined in (\ref{eqn:def})--(\ref{eqn:estimator2}).  Let $\hat{\theta}(\lambda)$ and  $\hat{v}(  \lambda^{\prime}) $ the estimates based on  choices  $\lambda$ and $\lambda^{\prime}$. Notice that $\hat{v}(\lambda^{\prime})$ depends on $\lambda$ but we do not make this dependence explicit to avoid overloading the notation. 

To choose $\lambda$, inspired by \cite{tibshirani2012degrees}, we use a Bayesian information criterion given as 
\begin{equation}
	\label{eqn:bic_mean}
	\widehat{\mathrm{BIC}}(\lambda)\,: =\,    \|y -\hat{\theta}(\lambda)\|^2  +  \widehat{\mathrm{df}}(\lambda) \log n   
\end{equation}
where  $\widehat{\mathrm{df}}(\lambda) $ is the number of connected components induced by $\hat{\theta}(\lambda)$ in the graph $G$. Then we select the value of $\lambda$ that minimizes  $\widehat{\mathrm{BIC}}(\lambda)$.

Once $\hat{\theta}(\lambda)$ has been computed, we proceed to select $\lambda^{\prime } $ for (\ref{eqn:estimator2}).  We let $\hat{\gamma}(\lambda^{\prime})  $  the solution to (\ref{eqn:estimator2}) and 
$\widetilde{\mathrm{df}}(\lambda^{\prime}) $ be the number of connected components in $G$ induced by $\hat{\gamma}(\lambda^{\prime})$.  Then we define 
\begin{equation}
	\label{eqn:score}
	\widetilde{\mathrm{BIC}}(\lambda^{\prime })\,: =\,    \sum_{i=1}^n  [ \min\{q, y_i^2 \}  - \hat{\gamma}(\lambda^{\prime})_i  ]^2 + \widetilde{\mathrm{df}}(\lambda^{\prime}) \log n 
\end{equation}
where $q$ is the $0.95$-quantile of the data $\{y_i^2\}_{i=1}^n$. We use $\min\{q, y_i^2 \} $ in (\ref{eqn:score}) to avoid the influence of outliers in the model selection step.  With the above score in hand, we choose the value of $\lambda^{\prime }$ that minimizes 
$\widetilde{\mathrm{BIC}}(\lambda^{\prime })$. In all our experiments, we select $\lambda$ and $\lambda^{\prime}$ from the set $\{10^1,10^2,10^3,10^4,10^5\}$.

\subsection{Homoscedastic case simulations}
\label{sec:homos}

We start by  considering settings where the variance, denoted as $v_0^*$, is constant across the different nodes $i$. As benchmarks, we consider the estimator  defined in (\ref{eqn:dfs_est})  which we refer as homoscedastic estimator  (Hom.), the heteroscedastic estimator (Het.) defined  (\ref{eqn:def})--(\ref{eqn:estimator2}), and  the U-statistic based local polynomial estimator  defined in \cite{shen2020optimal}   (U-LP).  


For our comparisons,  we generate data from the model in (\ref{eqn:model}) with 
$\epsilon_i \overset{\text{ind}}{\sim }   N(0,1)$ and $v_i^* = v_0^*$ for $i=1,\ldots,n$. We consider $2$-dimensional grid graphs $G $ with $n \in \{ 100^2,200^2, 300^2, 400^2\}$, and we identify the nodes of $G$ with elements of the set $\{1,\ldots\,n^{1/2}\} \times \{1,\ldots,n^{1/2} \}$. Then we consider values of $v_0^* $ in $\{0.5,1,1.5,2\}$ and three different scenarios for the signal $\theta^*$.  Next, we describe the choices of $\theta^*$ that we consider.

\textit{Scenario 1.} For $k,l \in \{ 1,\ldots, n^{1/2}\}$, we let 
\[
\theta^*_{k,l} \,=\, \begin{cases}
	1 & \text{if}\,\,\, \vert k - n/2\vert < n/4 ,\,\,\text{and}\,\,\vert l- n/2\vert < n/8,\\
	0 & \text{otherwise.}
\end{cases}
\]

\textit{Scenario 2.} We set 
\[
\theta^*_{k,l} \,=\, \begin{cases}
	1 & \text{if}\,\,\, (k-n/4)^2 + (l-n/4)^2< (n/5)^2,\\
	0 & \text{otherwise.}
\end{cases}
\]

\textit{Scenario 3.} In this scenario we set
\[
\theta^*_{k,l} \,=\, \begin{cases}
	1 & \text{if}\,\,\,    i < n/2   \,\,\text{and}\,\, j <n/2 ,\\
	0 & \text{otherwise.}
\end{cases}
\]

\begin{table}[tbp]
	\centering
	\setlength{\tabcolsep}{10.7pt}
	\begin{scriptsize}
		\begin{tabular}{|cc|ccc|ccc|ccc|}
			\toprule
			&	& \multicolumn{3}{c}{Scenario 1} \vline  & \multicolumn{3}{c}{Scenario 2} \vline & \multicolumn{3}{c}{Scenario 3} \vline \\
			\hline
			$n$ & $v_0$ & U-LP & Hom. & Het. & U-LP & Hom.& Het. & U-LP & Hom.& Het.\\
			\midrule
			$100^2$  &  0.25  & 0.33 &  \textbf{0.26} &1.15      & 0.17& \textbf{0.16}& 0.17  &0.43 &\textbf{0.28} &2.25\\ 
			$200^2$  &  0.25  & 0.14 &  \textbf{0.12} &1.121 & 0.21& \textbf{0.10}&  0.11 &0.39 &\textbf{0.11} &0.94\\ 
			$300^2$  &  0.25  & 0.15 &\textbf{0.08}& 0.97 &0.19 & \textbf{0.09}& \textbf{0.09}  & 0.44& \textbf{0.08}&0.47\\ 
			$400^2$  &  0.25  &0.14  &\textbf{0.07} &0.70 &0.22 &0.09 &  \textbf{0.08} & 0.49&\textbf{0.06}&0.28\\ 
			\midrule
			$100^2$  &  0.5  &     1.12      &  \textbf{1.10}     &1.11 &1.13 & \textbf{1.11} & 1.24&5.22&\textbf{1.12}&2.65\\ 
			$200^2$  &  0.5  &     0.52    & \textbf{0.44}     & 1.23& 1.21&  \textbf{0.62}& 1.34&4.99&\textbf{0.44}&0.94\\  
			$300^2$  &  0.5  &     0.62     & \textbf{0.32}      & 0.97& 1.19&  \textbf{0.34}&1.15 &4.84&\textbf{0.36}&0.48\\ 
			$400^2$  &  0.5  &    0.61       & \textbf{0.20}       &0.70 &1.38  &  \textbf{0.25}& 0.91&4.96 &\textbf{0.23} &0.29\\ 
			\midrule
			$100^2$  &  0.75  &  2.79&2.72 &\textbf{1.24} &1.36 &2.40 &\textbf{1.23}  &5.06 &3.06 &  \textbf{2.52}\\ 
			$200^2$  &  0.75  &1.20  &1.19 &\textbf{1.18} &1.31 &\textbf{1.27} & 1.40 &  4.89 &1.49 & \textbf{0.91}\\ 
			$300^2$  &  0.75  & 0.69 &\textbf{0.68} & 1.01& 1.43&\textbf{0.78} & 1.24  &4.76 &0.78 & \textbf{0.48}\\ 
			$400^2$  &  0.75  & 0.68 & \textbf{0.55}&0.69& 1.33&\textbf{0.58} &  0.92& 5.30&0.59 &  \textbf{0.30}\\ 
			\midrule
			$100^2$  &  1.0  &3.42  &3.94 & \textbf{1.18}&2.73 & 2.63 & \textbf{1.28} &5.46 &3.76 &\textbf{2.60} \\ 
			$200^2$  &  1.0  & 2.31 & 2.22&\textbf{1.28} & 1.58& 2.23&  \textbf{1.38} & 4.81& 2.03&\textbf{0.96}\\ 
			$300^2$  &  1.0  & 0.76 &\textbf{0.65} &0.94 &1.35 &\textbf{0.94} &1.22 &.4.94 & 1.06& \textbf{0.51}\\ 
			$400^2$  &  1.0  &0.57  &0.47 & \textbf{0.74}& 1.29& \textbf{0.89}&0.95 &4.81 &1.03 &\textbf{0.29}  \\  
			\bottomrule
		\end{tabular}
	\end{scriptsize}
	\caption{
		\label{tab:ex1}  Performance evaluations of  the competing methods  for the different settings described in the text. We report  $100$ multiplied by the average
		mean squared error,  averaging over 200 Monte Carlo simulations.
	}
\end{table}

For each scenario and value of the model  parameters, we generate 200 data sets and, for each data set, compute the different estimators.  We compute the  Hom. estimator  with a random DFS, and  the Het. estimator with tuning parameters chosen as in  Section \ref{sec:tuning}.   As for the U-LP estimator,  we follow the construction in  Section 4.1 from \cite{shen2020optimal}. First, we identiy the nodes of the $2$-dimensional grid graph with elements of the interval $[0,1]^2$, such that $(i,j)$ in the grid graph corresponds to $X_{  (j-1)n^{1/2}+ i    } := (i/n^{1/2},j/n^{1/2})  \in [0,1]^2 $ for $(i,j)   \in    \{ 1\ldots, n^{1/2} \}\times \{ 1,\ldots, n^{1/2}\}$.  We also let  $Y_{(j-1)n^{1/2}+ i   } =  y_{ (i,j) }$ where $y_{(i,j)}$ is the observation associated with $(i,j)$ in the 2-dimensional grid graph.  	Then we recall that estimator in \cite{shen2020optimal} in this context becomes 

\begin{equation}
	\label{eqn:estimator_shen}
	\hat{v}_{ \mathrm{original}}\,=\,	\frac{   {n\choose 2}^{-1}     \sum_{   k< \ell }   K_{h_1}( X_{k,1}  - X_{\ell,1}  )   \cdot K_{h_2}( X_{k,2}  - X_{\ell,2}  ) (Y_k -  Y_{\ell})^2  }{{n\choose 2}^{-1}     \sum_{   k< \ell }   K_{h_1}( X_{k,1}  - X_{\ell,1}  )   \cdot K_{h_2}( X_{k,2}  - X_{\ell,2}  ) },
\end{equation}
where $K  \,:\, \mathbb{R} \rightarrow \mathbb{R}$ is a kernel, $h_1, h_2>0$ are bandwidths, and $K_h(\cdot ) :=   K(\cdot /h)/h$.  Notice that computing $\hat{v}_{\mathrm{original}}$ involves $O(n^2)$ which quickly becomes intractable. Hence, we approximate  (\ref{eqn:estimator_shen}) with 
\begin{equation}
	\label{eqn:estimator_shen2}
	\hat{v}_{ \mathrm{approx}}\,=\,	\frac{  N^{-1}     \sum_{   s=1}^N   K_{h_1}( X_{k_s,1}  - X_{\ell_s,1}  )   \cdot K_{h_2}( X_{k_s,2}  - X_{\ell_s,2}  ) (Y_{k_s} -  Y_{\ell_s })^2  }{N^{-1}  \sum_{  s=1}^N   K_{h_1}( X_{k_s,1}  - X_{\ell_s,1}  )   \cdot K_{h_2}( X_{k_s,2}  - X_{\ell_s,2}  ) },
\end{equation}
where    $(k_1,\ell_1),\ldots, (k_N,\ell_N)$ are independent draws from the uniform distribution in $\{1,\ldots, n\}\times \{1,\ldots,n\}$.    The resulting estimator is the one that we consider as competitor in representation of the method from \cite{shen2020optimal}. In our simulations, we set $N= 5000$, $K$ is the Gaussian kernel, and $h_1 = h_2 = h$. We  allow $h \in \{ 2^{-10}, 2^{-9},\ldots, 2^{-1} \}$  and report results for the choice of $h$  that gives the best performance in terms of estimating the true parameter $v_0^*$.


We use the mean squared error as a measure of performance for the different estimators. For the methods Hom. and U-LP which only compute a single estimator, denoting the output of the method as $\hat{v} \in \mathbb{R}$, we  compute the average of $\vert \hat{v} -v_0^*\vert^2$ across the 200 replicates. For the method Het. that produces a vector $\hat{v} \in \mathbb{R}^n$, we compute the average of
\begin{equation}
	\label{eqn:mse}
	\frac{1}{n} \sum_{i=1}^{n}   (\hat{v}_i -  v_0^*)^2
\end{equation}
over the 200 Monte Carlo simulations. The results can be seen in Table \ref{tab:ex1}, where observe that our proposed estimators  Hom.  and Het. outperform  the competitor in all of the instances cinsidered. This does not come as a surprise since  the true mean in each scenario is piecewise constant, making it challenging for the kernel based method from \cite{shen2020optimal}, while both of our proposed methods are better suited for handling  piecewise constant signals for both the mean and variance vectors.

\subsection{Heteroscedastic case:  2D Grid graphs}
\label{sec:2d}

In our next experiment, we consider generative models where the true graph is a 2D grid graph of size $n^{1/2} \times n^{1/2}$. We generate data similarly to  Section \ref{sec:homos} with the difference that the variance is now not constant. Specifically, the data are generated as 
\[
y_i \,=\,   \theta^*_i  \,+\, \sqrt{ v_i^* }\epsilon_i
\]
with  $\epsilon_i \overset{ind} {\sim} N(0,1)$.   The scenarios we consider are:

\textit{Scenario 4.} The signal $\theta^*$ is taken as in Scenario 3 from Section \ref{sec:homos}, and we let
\[
v^*_{k,l} \,=\, \begin{cases}
	1.75 & \text{if}\,\,\, \vert k - n/2\vert < n/3 ,\,\,\text{and}\,\,\vert l- n/2\vert < n/3,\\
	1 & \text{otherwise,}
\end{cases}
\]
for $k,l \in \{ 1,\ldots, n^{1/2}\}$.

\begin{figure}
	\begin{center}
		\includegraphics[width =1.95in,height = 1.95in]{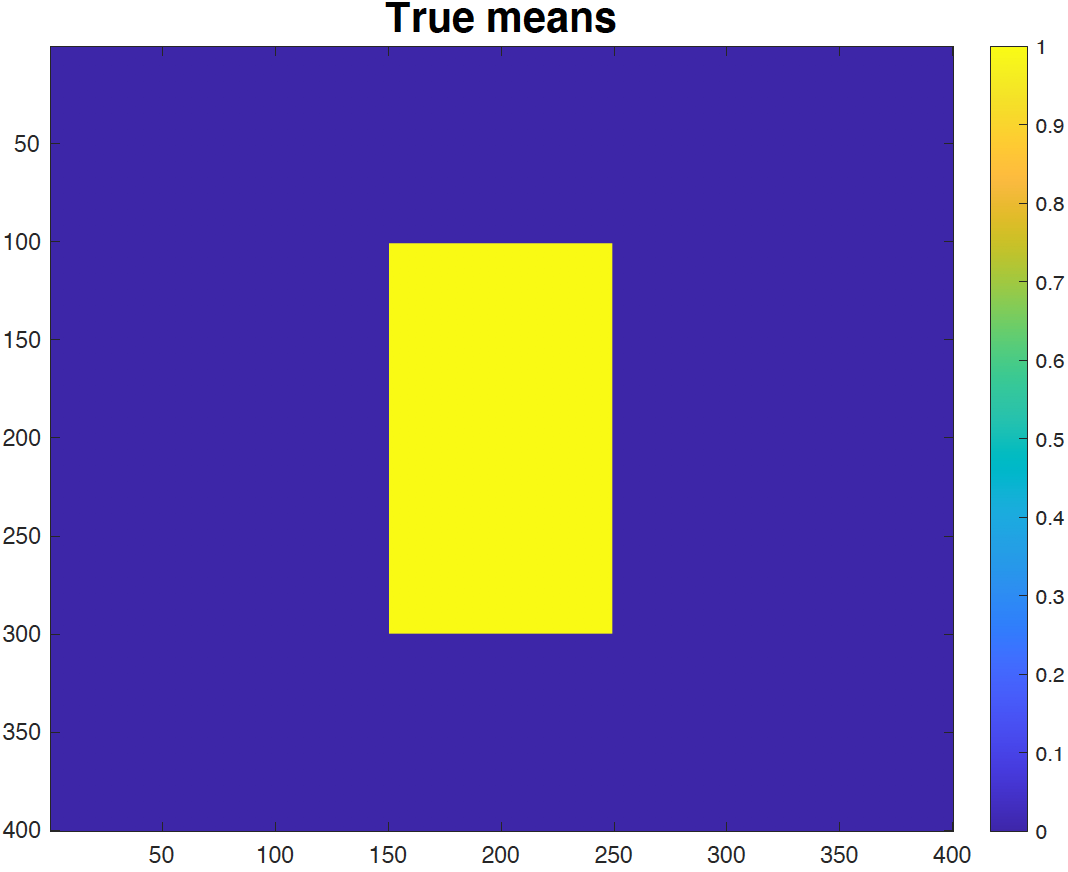}
		\includegraphics[width =1.95in,height = 1.95in]{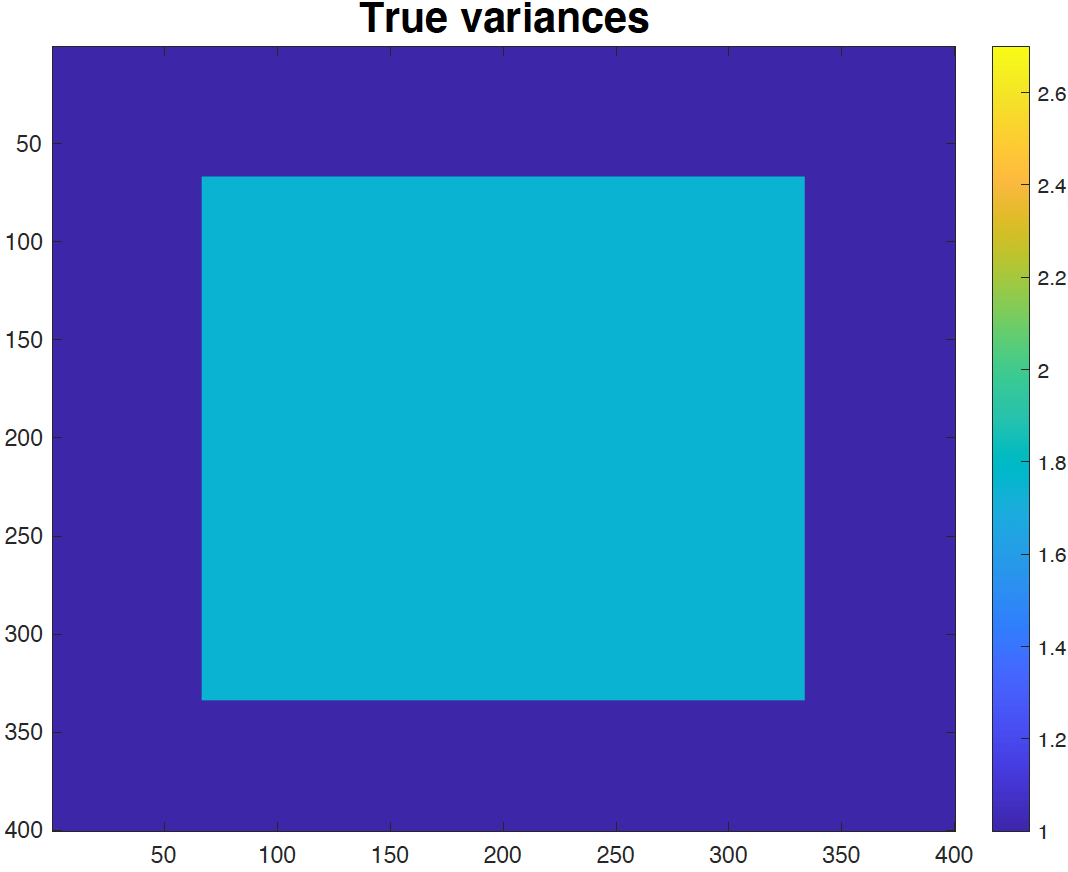}
		\includegraphics[width =1.95in,height = 1.95in]{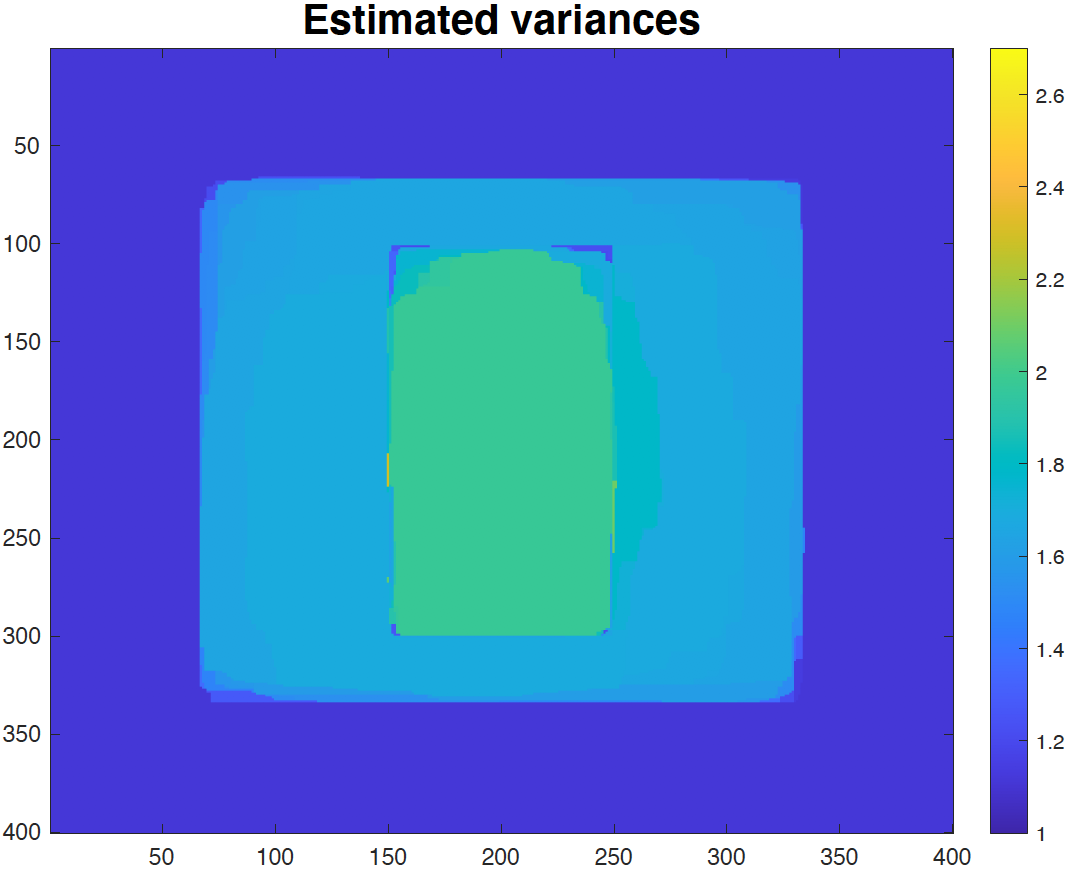}
		\includegraphics[width =1.95in,height = 1.95in]{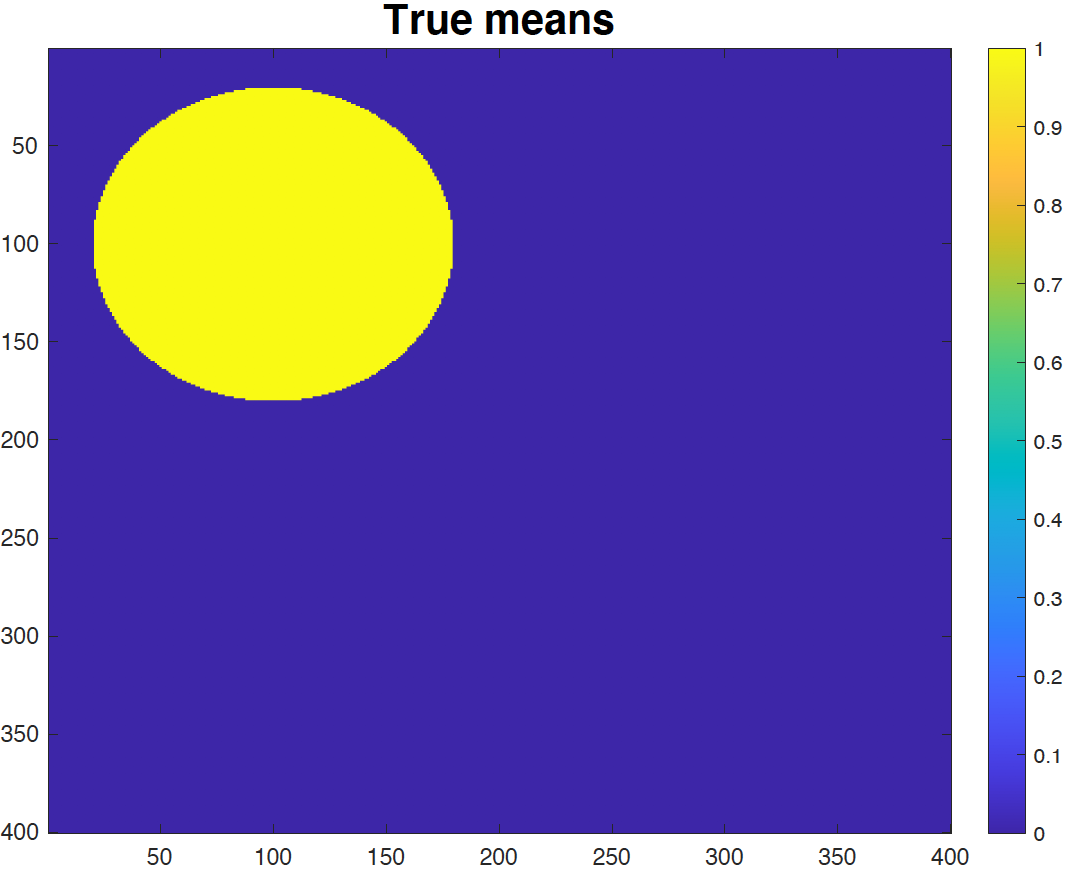}
		\includegraphics[width =1.95in,height = 1.95in]{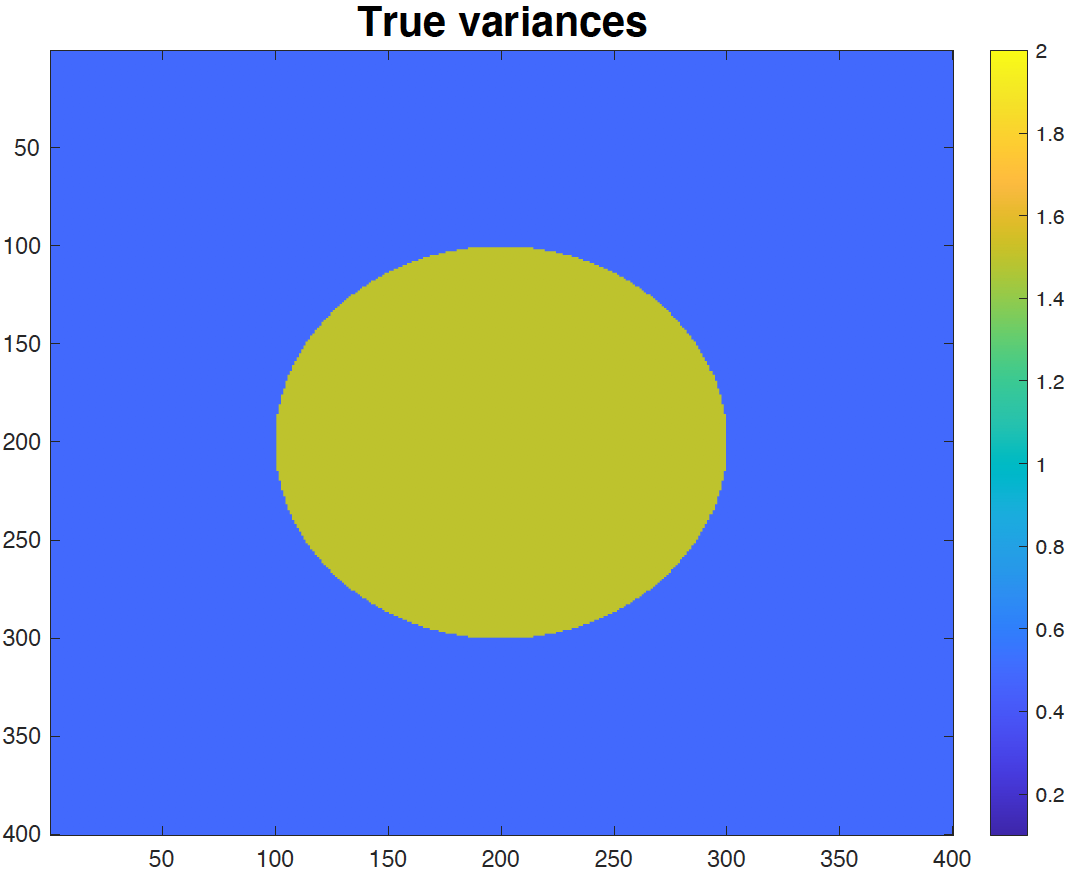}
		\includegraphics[width =1.95in,height = 1.95in]{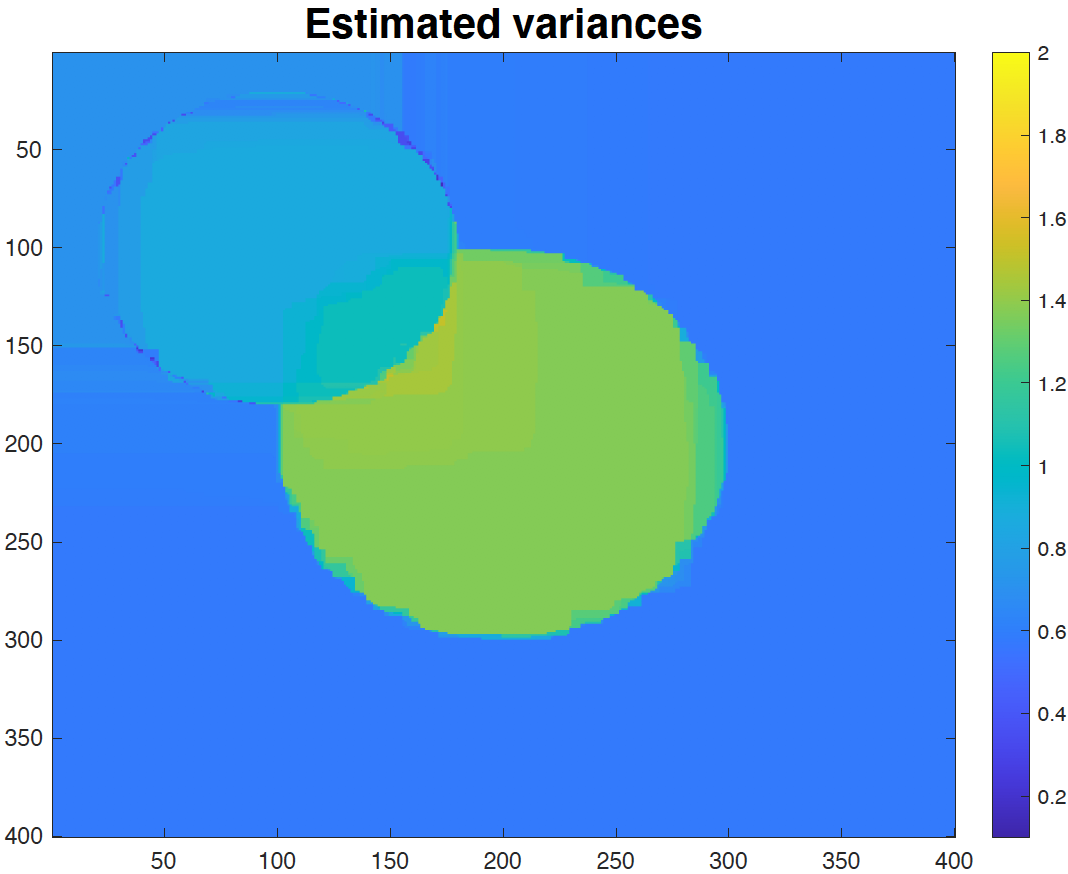}
		\includegraphics[width =1.95in,height = 1.95in]{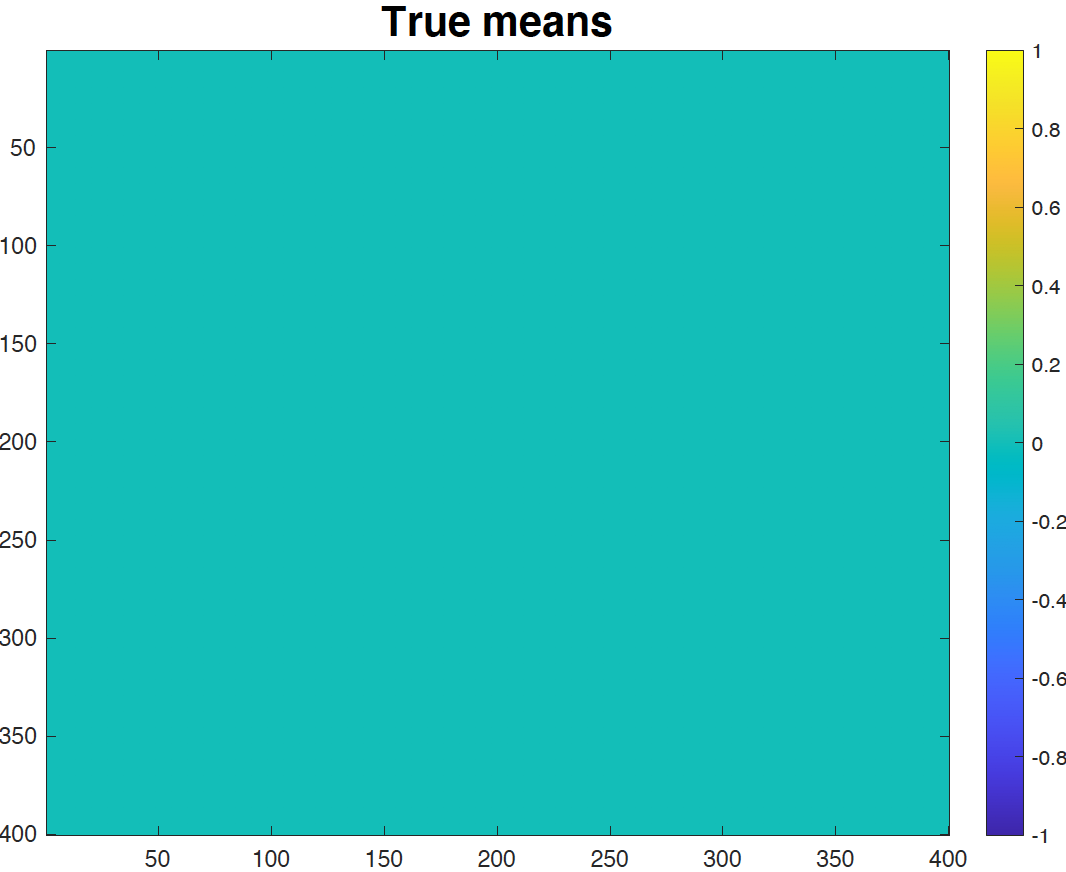}
		\includegraphics[width =1.95in,height = 1.95in]{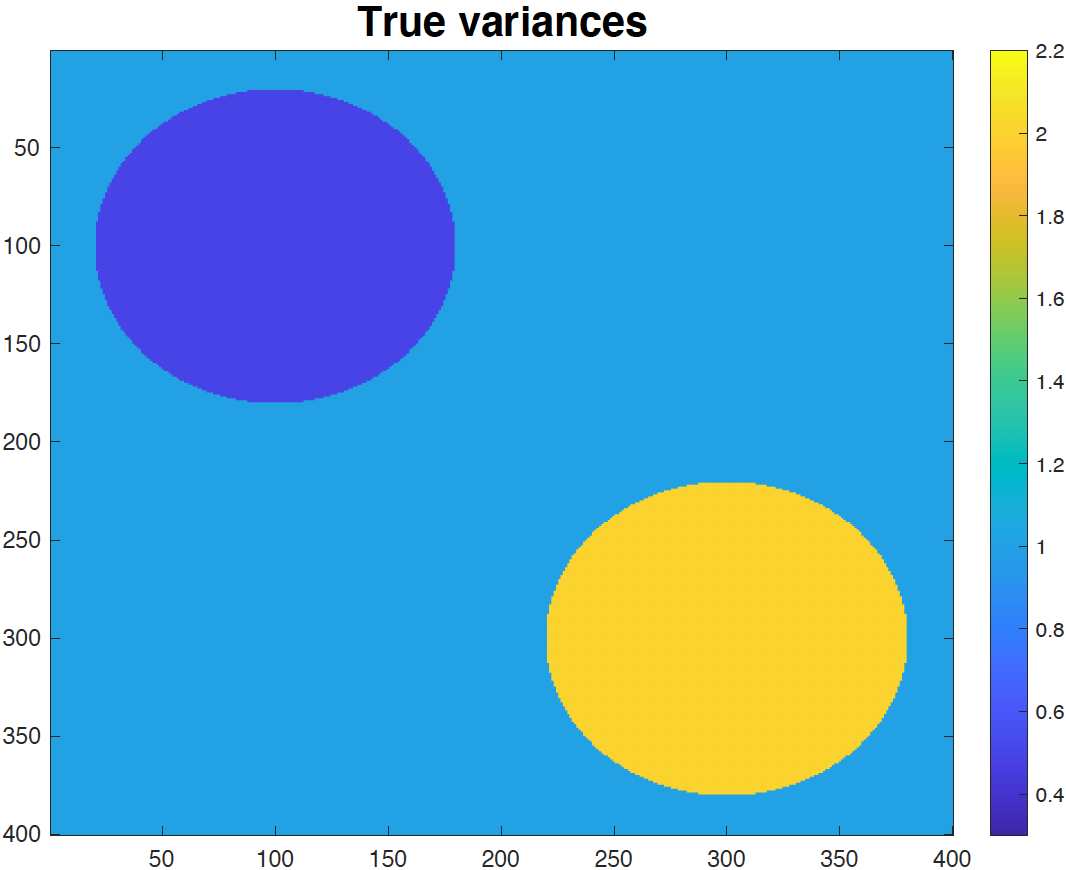}
		\includegraphics[width =1.95in,height = 1.95in]{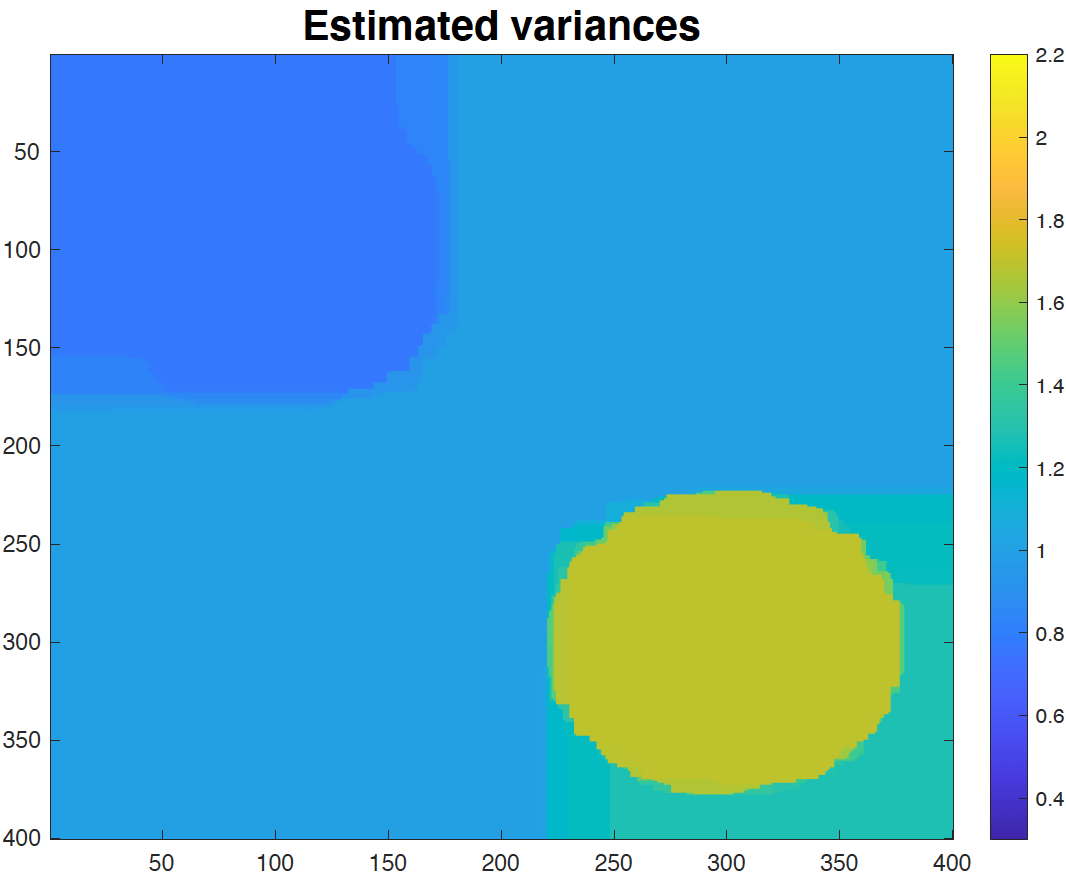}
		\caption{ 		\label{fig2} Each row corresponds to one scenario, with the top row corresponding to Scenario 4, the middle to Scenario 5, and the bottom to Scenario 6.  The left column depicts the signals $\theta^*$, the middle column the signals $v^*$, and the right column the estimated $\hat{v}$ with our method in (\ref{eqn:def})--(\ref{eqn:estimator2}). }
	\end{center}
\end{figure}

\begin{table}[t!]
	\centering
	\setlength{\tabcolsep}{7.7pt}
	\begin{scriptsize}
		\begin{tabular}{|c|ccc|ccc|ccc|}
			\toprule
			& \multicolumn{3}{c}{Scenario 4} \vline  & \multicolumn{3}{c}{Scenario 5} \vline & \multicolumn{3}{c}{Scenario 6} \vline \\
			\hline
			$n$ & L.  Pol. & Laplacian S.&  Het.&L.  Pol. &Laplacian S.& Het. &L.  Pol.  & Laplacian S. & Het.\\
			\midrule
			$100^2$ &  3.25&  11.02&  \textbf{1.34 }&2.18 & 4.91&  \textbf{1.57}& 2.40& 9.49&\textbf{1.22}\\ 
			$200^2$   & 3.27 & 10.71 &  \textbf{0.52}& 2.17&4.91 & \textbf{0.75}&2.44 &8.09 &\textbf{0.72}\\ 
			$300^2$  & 3.21 & 6.82 & \textbf{0.29}& 2.15& 4.43&  \textbf{0.43} & 2.41&7.87 &\textbf{0.42}\\ 
			$400^2$    &  3.15&  6.14&\textbf{0.18} & 2.14& 3.92&\textbf{0.28} &2.39 &7.41 &\textbf{0.29}\\ 
			\bottomrule
		\end{tabular}
	\end{scriptsize}
	\caption{
		\label{tab:ex2}  Performance evaluations of  the competing methods  for the different settings described in the text. We report  $10$ multiplied by the average
		mean squared error,  averaging over 200 Monte Carlo simulations.
	}
\end{table}

\textit{Scenario 5.} We set 
\[
v^*_{k,l} \,=\, \begin{cases}
	1.5 & \text{if}\,\,\, (k-n/2)^2 + (l-n/2)^2< (n/4)^2,\\
	0.5 & \text{otherwise.}
\end{cases}
\]
for $k,l \in \{ 1,\ldots, n^{1/2}\}$ and take $\theta^*$ as in Scenario 2 in Section \ref{sec:homos}.

\textit{Scenario 6.} We let  $\theta^* = 0 \in  \mathbb{R}^{ n^{1/2} \times n^{1/2} }$ and 
\[
v^*_{k,l} \,=\, \begin{cases}
	0.5 & \text{if}\,\,\,(k-n/4)^2 + (l-n/4)^2< (n/5)^2 ,\\
	2& \text{if}\,\,\,  (k-3n/4)^2 + (l-3n/4)^2< (n/5)^2,\\
	1 & \text{otherwise.}
\end{cases}
\]
for $k,l \in \{ 1,\ldots, n^{1/2}\}$. 

As for benchmarks, we compare our estimator Het. defined in  (\ref{eqn:def})--(\ref{eqn:estimator2})  with the  local polynomial regression (L. Pol.) method from \cite{fan1998efficient}, and the laplacian smoothing estimator (Laplacian S.) , see e.g.  \cite{smola2003kernels}. For our method Het. the tuning  parameters are selected as in Section \ref{sec:tuning}. As for  the method L. Pol., we use the  function \textit{loess} from the \textit{R} package \textit{stats} with the default choices of input. However, for large values of $n$ ($n \geq 10000$)  the computational complexity of this function becomes challeging, and hence we average 10 estimates each of which  is obtained by fitting  the estimator  to randomly selected subsets of the data with size $5000$. 

As for  the Laplacian S. estimator, we first define

\begin{equation}
	\label{eqn:laplacian_mean}  
	\tilde{\theta}  \,:=\, \underset{\theta \in  \mathbb{R}^n}{\arg\min} \,\left\{  \frac{1}{2}\|  y-\theta\|^2 \,+\,\eta \sum_{(i,j) \in G} \vert \theta_i - \theta_j\vert^2        \right\},
\end{equation}
where $\eta >0$ is a tuning parameter, and $G$ is the 2-dimensional grid graph. Thus,  the only difference with the estimator $\hat{\theta}$ defined in (\ref{eqn:gfl}) is in the penalty with (\ref{eqn:laplacian_mean}) using the square of the absolute value of the difference of the signal values, along the edges of  the graph. Once  $\tilde{\theta} $ has been constucted, we define
\begin{equation}
	\label{eqn:def2}
	\tilde{v}_i \,=\,  \hat{\gamma}_i - (\tilde{\theta}_i)^2,
\end{equation}
where
\begin{equation}
	\label{eqn:estimator3}
	\hat{\gamma} \,:=\,\underset{\gamma \in  \mathbb{R}^n}{\arg \min}\left\{ \frac{1}{2}\sum_{i=1}^{n}(y_i^2 - \gamma_i)^2   +   \eta^{\prime}  \sum_{ (i,j)\in E } \vert\gamma_i -\gamma_j\vert ^2  \right\}
\end{equation}
for a tuning parameter $\eta^{\prime} >0$. The final estimator  $\tilde{\gamma}$ is the one that we refer to as Laplacian S., where the tuning parameters are chosen with BIC as in Section \ref{sec:tuning}.

For each scenario and value of the tuning parameters, and for each data set, we compute the estimator $\hat{v}$  and choose the tuning parameters as in Section \ref{sec:tuning}. We then report 
\[
\frac{1}{n} \| \hat{v} -v^* \|^2
\]
averaging over 200  Monte Carlo simulations. The results in Table \ref{tab:ex2} show an excellent performance of our estimator, which becomes more evident as $n$ grows.  This  goes in line with our findings in Theorem \ref{thm5}.

Finally, Figure \ref{fig2} provides visualizations of Scenarios 4--6 and the corresponding estimates $\hat{v} $ for one instance of $n=400^2$. There, we can see that $\hat{v}$ is a reasonable esimator of $v^*$, although $\hat{v}$ is afected by the bias induced by $\hat{\theta}$ which comes from Equation (\ref{eqn:def}).

\begin{figure}[h!]
	\begin{center}
		\includegraphics[width = 2.9in]{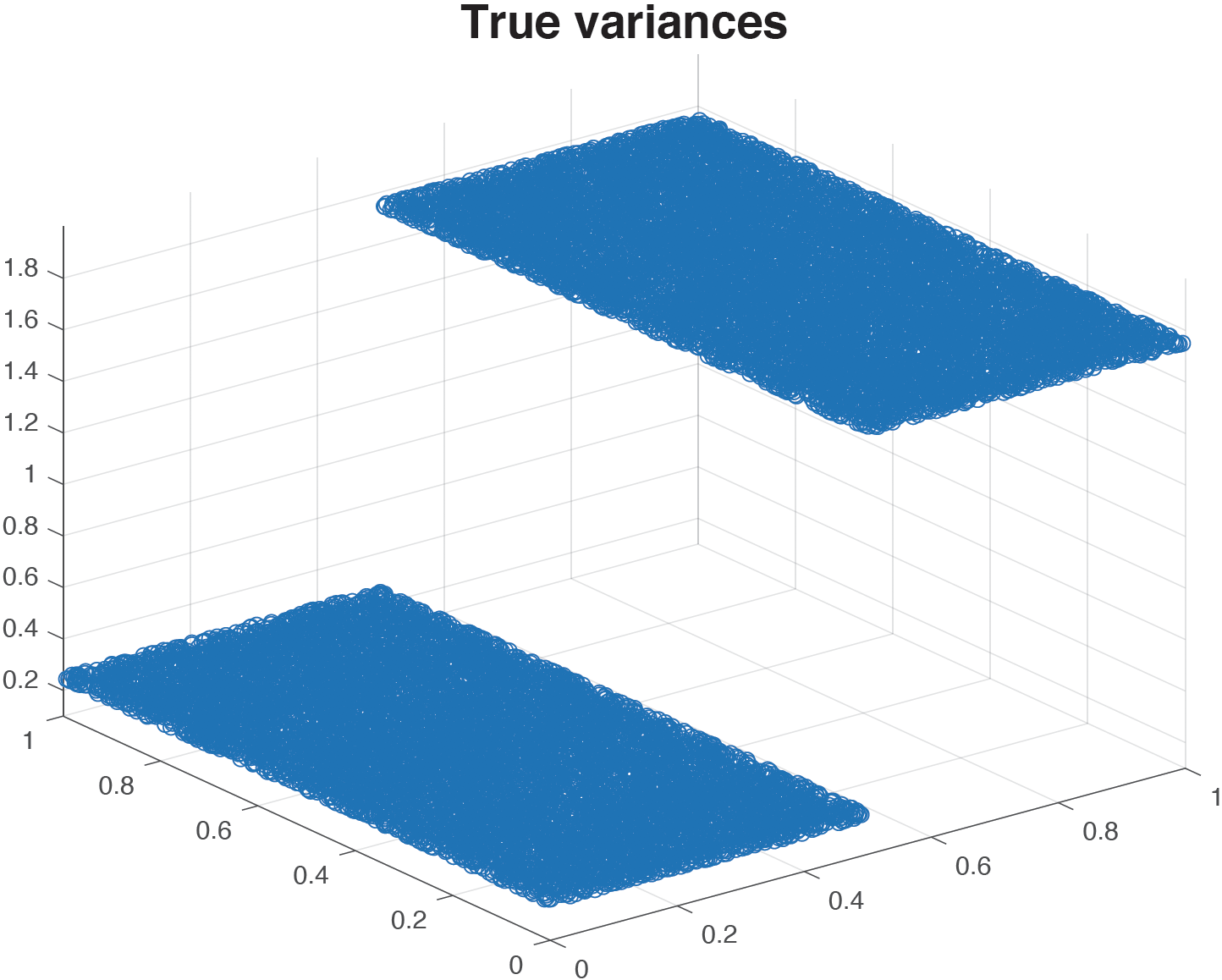}
		\includegraphics[width = 2.9in]{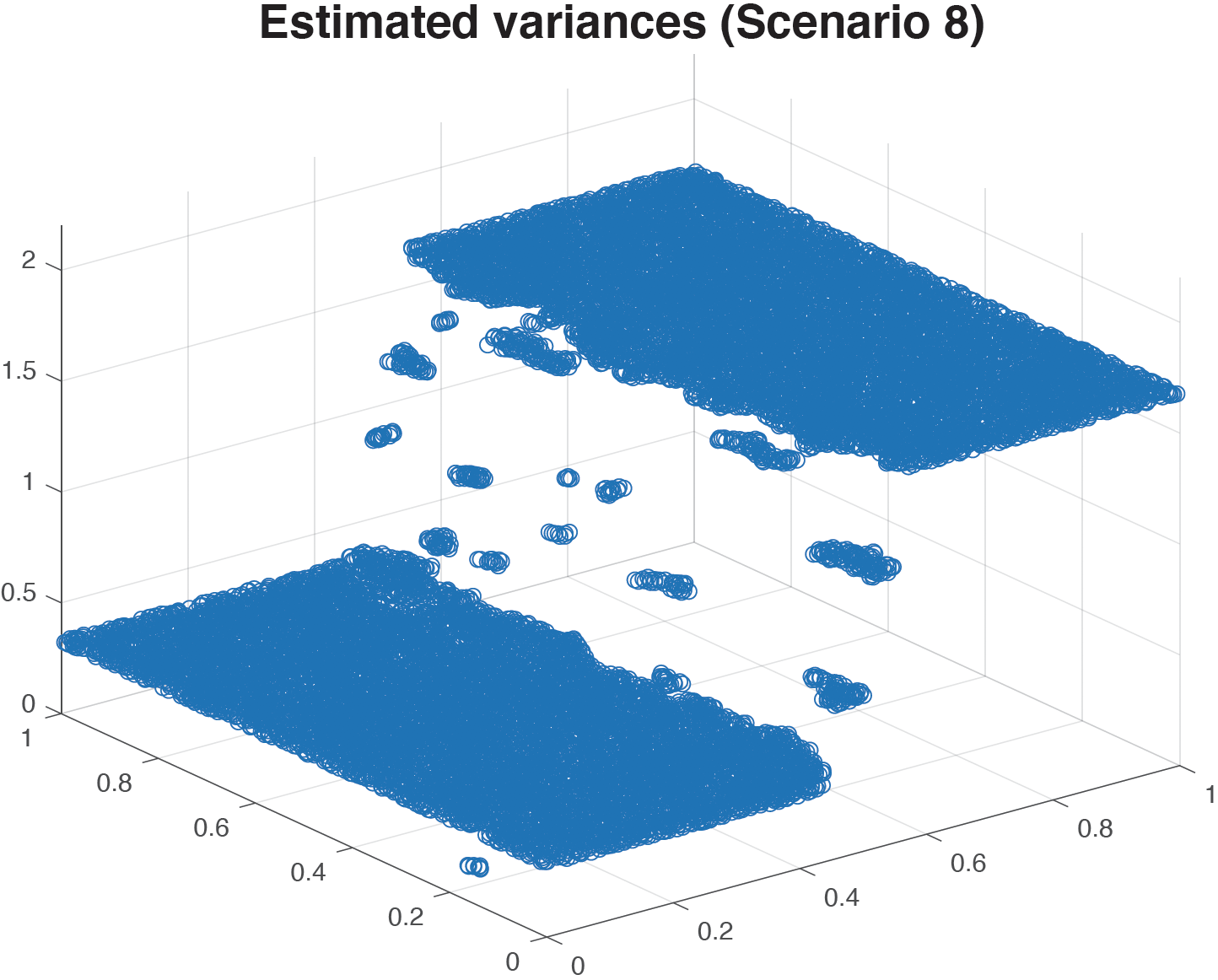}
		\includegraphics[width = 2.8in]{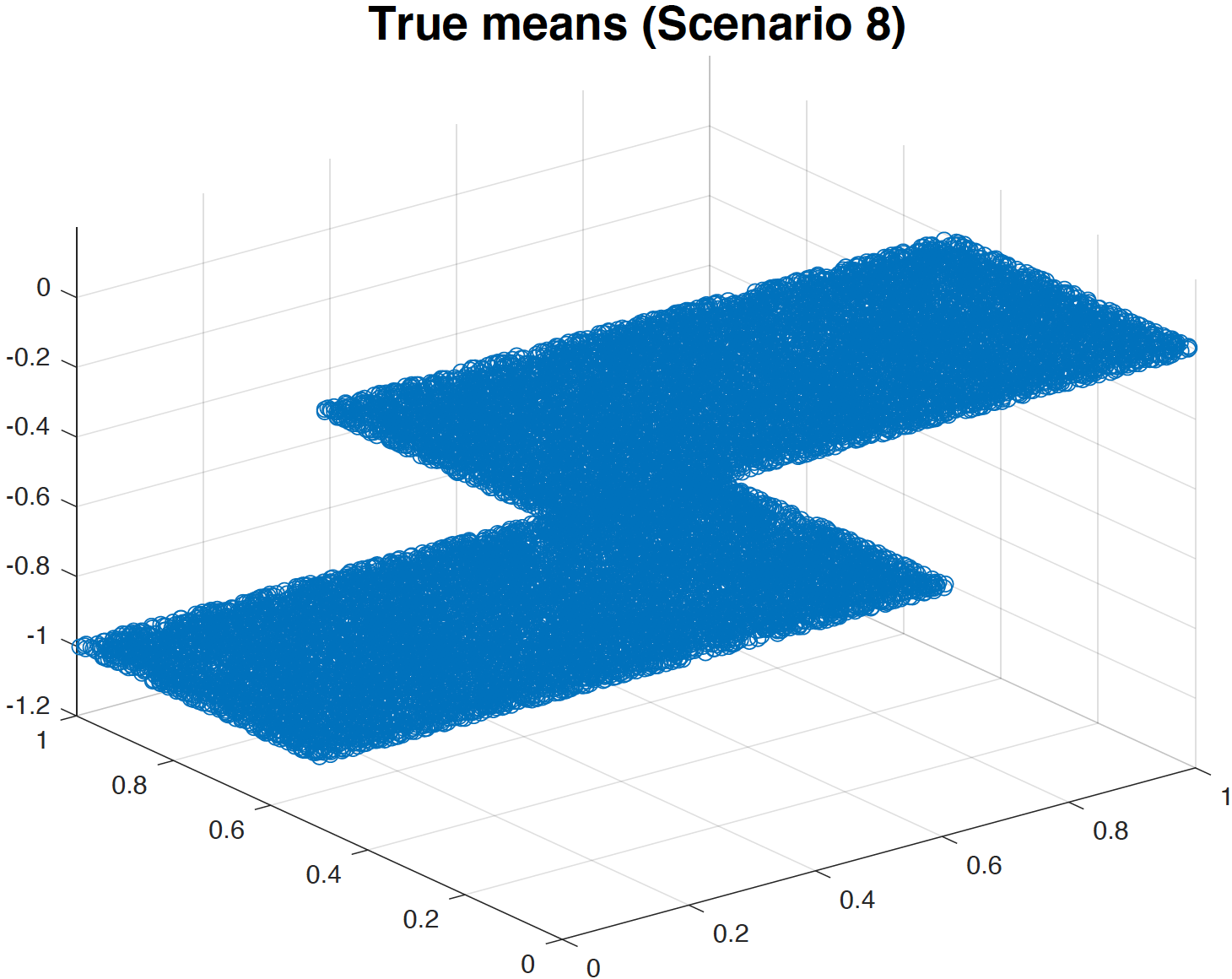}
		\includegraphics[width = 2.8in]{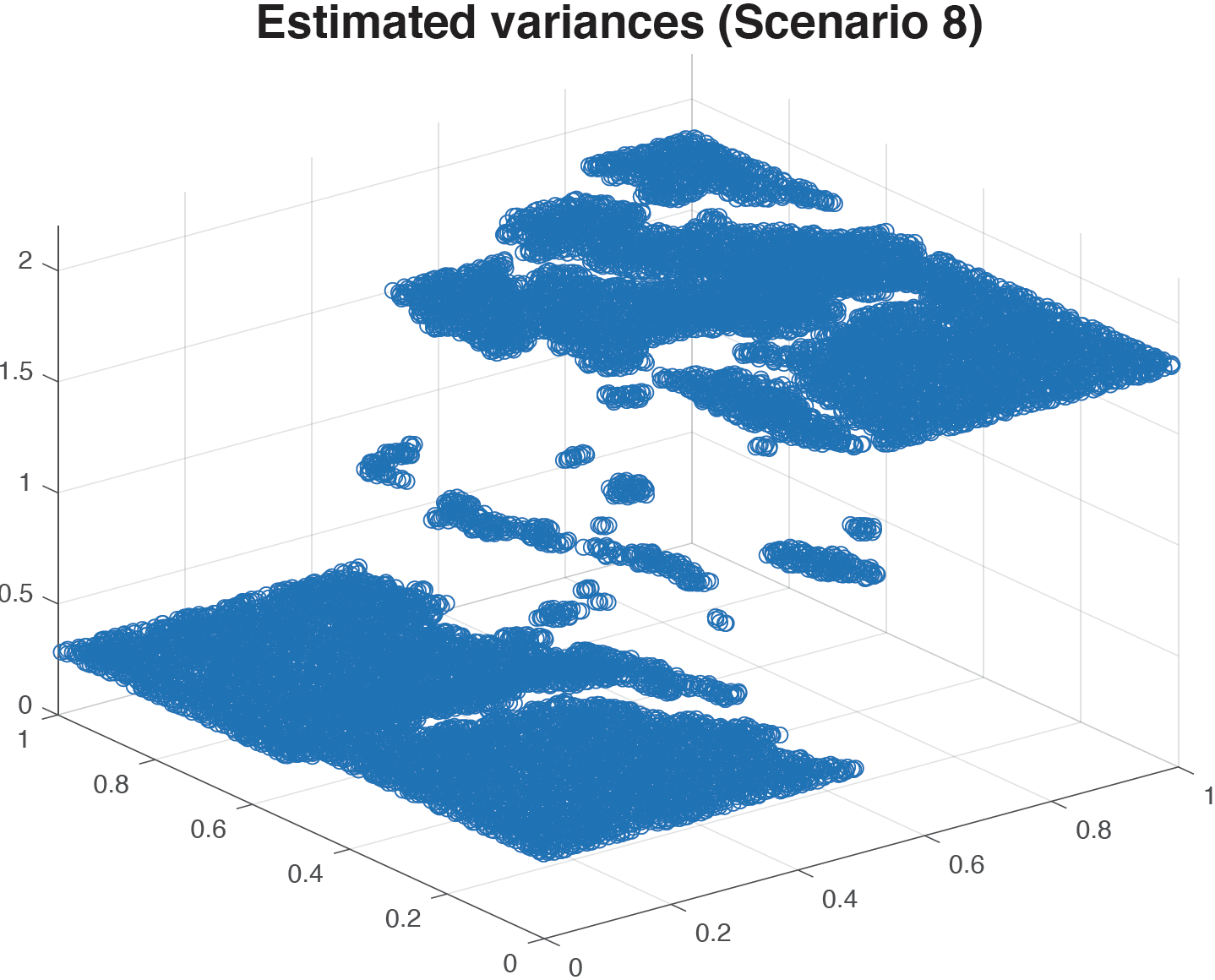}
		\caption{ 		\label{fig3}.   For $n = 20000$ and $d=2$, the top left panel shows a scatter plot of $\{(x_{i,1},x_{i,2},  v_i^*)\}_{i=1}^n$ for one instance of Scenarios 7 and 8. The top right panel displays the corresponding  scatter plot of $\{(x_{i,1},x_{i,2},  \hat{v}_i)\}_{i=1}^n$  for Scenario 7. The bottom left panel is the  scatter plot of  $\{(x_{i,1},x_{i,2},  f_0(x_i))\}_{i=1}^n$ for Scenario 8, and the bottom right panel shows the scatter plot of $\{(x_{i,1},x_{i,2},  \hat{v}_i)\}_{i=1}^n$  for Scenario 8.   Here, $\hat{v}$ is our Het. estimator defined in (\ref{eqn:def})--(\ref{eqn:estimator2})   with the $K$-NN graph.}
	\end{center}
\end{figure}

\subsection{Heteroscedastic case:  $K$-NN graphs}
\label{sec:knn}

In this experiment we consider a nonparametric regression setting. Specifically, we generate data from the model 
\[
\begin{array}{lll}
	y_i  &  =&   f_0(x_i)  +   (v_i^*)^{1/2}\epsilon_i   ,  \\ 
	v_i^* &=&	g_0(x_i)\\
	\epsilon_i	&\overset{\text{ind}}{\sim } & N(0,1),\\
	x_i & \overset{\text{ind}}{\sim }& U[0,1]^d,
\end{array}
\]
where $U[0,1]^d$ is the uniform distribution. In our simulations, we consider $d\in \{2,3\}$, $n \in \{500,10000,15000,20000\}$, and difference choices of $f_0$ and $g_0$. The functions $f_0$ and $g_0$ are taken from the following scenarios:

\textit{Scenario 7.} In this scenario, we let $f_0(z) = 0$  for all $z =  (z_1,\ldots,z_d)^{\top }\in \mathbb{R}^d$ and 
\[
g_0(z)  \,=\, \begin{cases}
	1.75	 & \text{if} \,\,\, z_1 >0.5,   \\
	0.25 & \text{otherwise.} 
\end{cases}
\]

\begin{table}[t!]
	\centering
	\setlength{\tabcolsep}{10.7pt}
	\begin{scriptsize}
		\begin{tabular}{|cc|ccc|ccc|}
			\toprule
			&	& \multicolumn{3}{c}{Scenario 7} \vline  & \multicolumn{3}{c}{Scenario 8}  \vline \\
			\hline
			$n$ & $d$& L. Pol. & Laplacian S. & Het. &L. Pol. &Laplacian S.& Het\\
			\midrule
			$5000$  & $2$&1.06 &1.73  &    \textbf{0.59} & 1.06&5.27 & \textbf{0.87}\\ 
			$10000$  & $2$& 1.04& 1.65 &    \textbf{0.40} & 1.01& 5.11& \textbf{0.58}\\ 
			$15000$  & $2$&1.01 & 1.57 &\textbf{0.34}     &1.04 & 5.26& \textbf{0.45}\\ 
			$20000$  & $2$& 1.25& 1.53 &  \textbf{0.27} & 1.08& 5.05& \textbf{0.39}\\ 
			\midrule
			$5000$  & $3$& 1.48&1.49  &    \textbf{1.45} &   \textbf{1.86}& 4.21&  1.88\\ 
			$10000$  & $3$&1.38 & 1.40 & \textbf{1.05}    & 1.61& 4.11&   \textbf{1.54}\\ 
			$15000$  & $3$&1.34 & 1.48 &\textbf{0.92}     & 1.40&4.11 &  \textbf{1.22} \\ 
			$20000$  & $3$&1.42 &1.26  & \textbf{0.89}   & 1.39& 4.37& \textbf{1.12}\\ 
			\bottomrule
		\end{tabular}
	\end{scriptsize}
	\caption{
		\label{tab:ex3}  Performance evaluations of  the competing methods  for the different settings described in the text. We report  $10$ multiplied by the average
		mean squared error,  averaging over 200 Monte Carlo simulations.
	}
\end{table}

\textit{Scenario 8.}  We let  $g_0$  as in Scenario 7, and let
\[
f_0(z)  \,=\, \begin{cases}
	0	 & \text{if} \,\,\, z_2 >0.5,   \\
	-1& \text{otherwise,} 
\end{cases}
\]
for $z  \in R^d$

Based on the above scenarios, we generate 200 data sets and compute the mean squared error of our estimator in (\ref{eqn:def})--(\ref{eqn:estimator2}) averaging over all the repetitions. Our estimator is computed using the $K$-NN graph   with $K=5$. Table \ref{tab:ex3} seems to corroborate our findings in Theorem \ref{thm5} as our method's performance appears to improve with sample size but worsens when $d$ increases.  

Finally, Figure \ref{fig3} provides a visualization of  the true signals and the estimated variances for one instance with $n=12000$ and $d=2$.


\subsection{ Ion channels data}
\label{sec:real}

\begin{figure}[h!]
	\begin{center}
		\includegraphics[width = 4.4in,height = 4in]{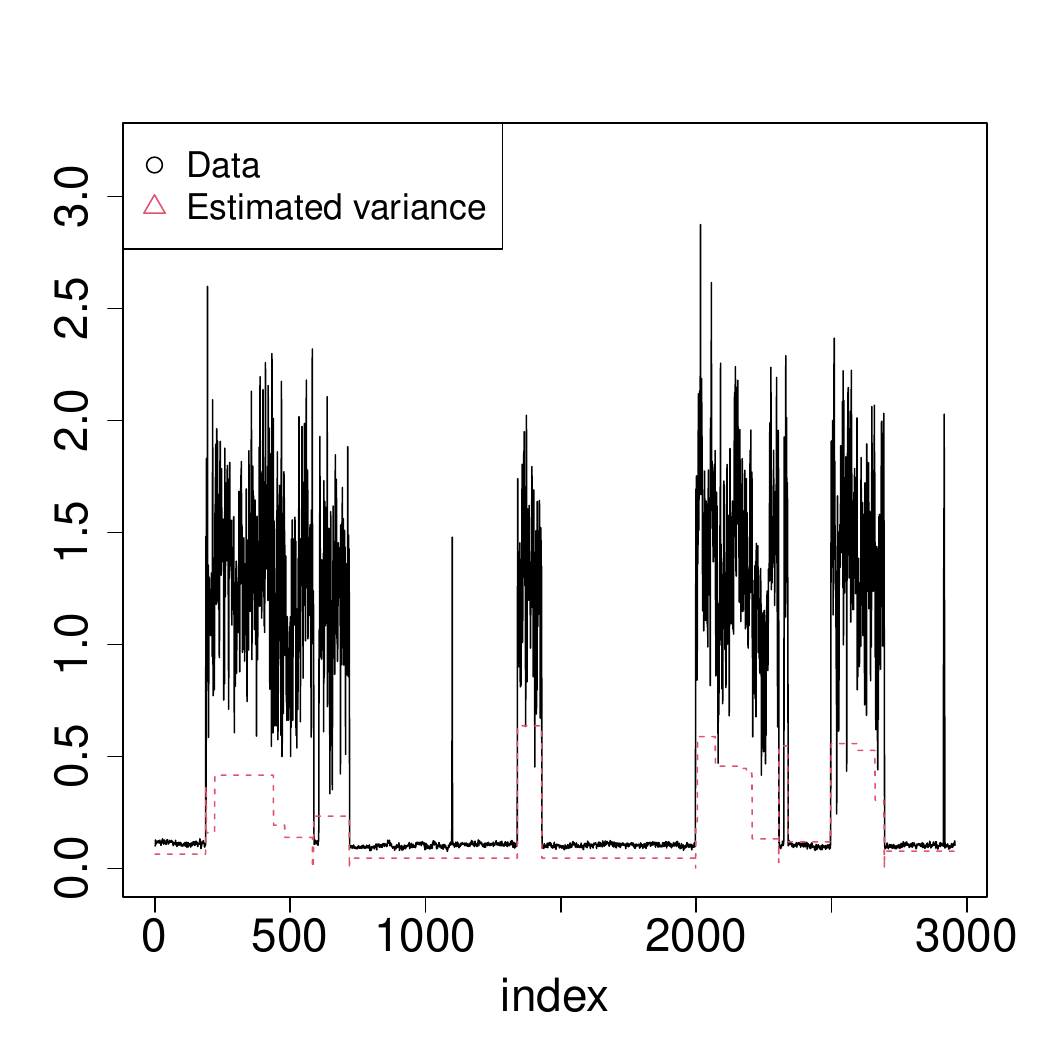}
		\caption{ 		\label{fig4} Ion channels data and estimated variances  }
	\end{center}
\end{figure}

We now validate our method using a real data example. Specifically, we consider the Ion channels data used by  \cite{jula2021multiscale}. The original data  was produced by the Steinem Lab (Institute of Organic and Biomolecular Chemistry, University of
Gottingen). As explained by the \cite{jula2021multiscale},  Ion channels are
a class of proteins expressed by all cells that create pathways for ions to pass through the cell membrane. The data consist of a single ion channel of
the bacterial porin PorB, a bacterium related to  Neisseria gonorrhoeae.

Although the original data consists of   600000 time instances. We proceed as in \cite{cappello2021scalable} and construct a signal $y \in \mathbb{R}^{2048}$. The resulting data are plotted in Figure \ref{fig4}. There, we also see the estimated variances using our proposed method for the heteroscedastic case on the 1D chain graph. We see that our method seems to capture the heteroscedastic nature of the data.

\section{Conclusion}
\label{sec:conclude}

In this paper, we have studied the problem of estimating the variance in general graph denoising problems. We have proposed and analyzed estimators for both the homoscedastic and heteroscedastic cases. In studying the latter, we also proved generalizations of known bounds for the fused lasso estimator to models beyond sub-Gaussian errors. 

Many research directions are left open in this work. One particular problem is to generalize our results to higher order versions of total variation for estimating the vector of variances. Constructing higher order versions of total variation is challenging in the case of estimating the mean in general graph-structured problems, and we expect it to be even more challenging for the variance case. Therefore, we leave this for future work.

\appendix

\section{Canonical scaling}
\label{sec:canonical}

In previous sections of the paper we made reference to the fact that the canonical scaling of the total variation in   a $d$-dimensional  grid graph  is $O(n^{1-1/d})$. Thus, for the 1D chain graph we obtain the canonical scaling is $O(1)$
and $O(n^{1/2})$ for the 2D grid graph. We now justify this by following the discussion from \cite{sadhanala2017higher}.  

To start, consider a $d$-dimensional grid graph given as $G =(V,E)$, with $V = \{1,\ldots,n\}$. We let $N= n^{1/d}$ and construct  the $d$-dimensional lattice $Z_d  =\{ (i_1/N,\ldots,i_d/N)\,:\,  i_1,\ldots,i_d \in \{1,\ldots,N\} \} \subset [0,1]^d$. Then we can index the components of a vector $\theta \in \mathbb{R}^n$ by the lattice locations, $\theta(a) $, $a\in Z_d$.  Then, the total variation of $\theta$ along the graph $G$ is given by
\[
\| \nabla_G \theta\|_1 \,=\, \frac{1}{2}\sum_{a \in Z_d} \sum_{ b \in  Z_d }  \vert \theta(a) -\theta(b ) \vert  1_{ \{   \|a-b\|=\frac{1}{N}    \} } .
\]
Notice that the factor $1/2$ appears because we are counting every edge exactly twice.  Next assume that $\theta(a)  \,=\, f(a)$ for 
a function $f \,:\,[0,1]^d \,\rightarrow \mathbb{R} $  such that
\[
\| f(a) - f(b) \| \,\leq \, L\|a-b\|,
\]
for all $a, b \in [0,1]^d$  and for a constant $L>0$.  Thus, $f$ is an $L$-Lipschitz function.  It follows that 
\[
\begin{array}{lll}
	\displaystyle  \| \nabla_G \theta\|_1    & \leq &\displaystyle  \frac{1}{2}\sum_{a \in Z_d} \sum_{ b \in  Z_d }  L\|a -b \|  1_{ \{   \|a-b\|=  \frac{1}{N}   \} }\\
	& \lesssim &\displaystyle \frac{1}{2}\sum_{a \in Z_d}  \frac{dL }{N }\\
	& =  &\displaystyle   \frac{d  n L}{2 N}   \\
	& = & O(n^{1-1/d}),
\end{array}
\]
and so $O(n^{1-1/d})$ makes sense as a canonical scaling for $\|\nabla_G \theta\|_1$ when $G$ is a $d$-dimensional grid graph.

\section{Additional experiments}

\subsection{Model selection}
\label{sec:model_selection}

Recall that in Section \ref{sec:homo} we motivated our homoscedastic estimator as being potentially useful for  model selection. In this section, we evaluate a BIC criterion for model selection  for the purpose of mean estimation in a 2-dimensional grid graph.  For this evaluation, we consider 
Scenarios 1--3 from Section \ref{sec:homos}. For each scenario, and $\sigma \in \{ 0.25, 0.5,0.75, 1 \}$, we generate 200 data sets and compare the performance of  $\hat{\theta}(\lambda)$, the solution to (\ref{eqn:gfl}) with two choices of $\lambda$. The first choice of $\lambda$ is taken as optimal, thus, as 
\[
\lambda \,:=\, \, \underset{\lambda\in \Lambda  }{\arg \min}\|  \theta^* -  \theta(\lambda) \|^2,
\]
where $\Lambda   =\{   10^1,10^2, 10^3,10^4,10^5 \}$.    The second choice of  $\lambda$ is set to
\[ 
\lambda \,:=\,\underset{\lambda\in \Lambda  }{\arg \min}\{ \|  y - \hat{\theta}_{\lambda}\|^2 \,+\, \log(n) \hat{v} \widehat{\mathrm{df}}_{\lambda} \},
\] 
where $\widehat{\mathrm{df}}_{\lambda}$ is the number of connected components induced by $\hat{\theta}(\lambda)$ in the 2-dimensional grid graph, and where $\hat{v} $ is the homoscedastic estimator defined in (\ref{eqn:dfs_est}). 

We find that in all the instances considered, both choices of $\lambda$ coincide, suggesting that in practice the estimator $\hat{v}$ can be useful for mean estimation with a BIC criterion that typically produces the optimal choice of $\lambda$.


\subsection{Mean estimation wih Laplace errors}

Notice that in the second step of our heteroscedastic estimator, Equation (\ref{eqn:estimator2}), we actually estimate $\gamma^*_i = \mathbb{E}(y_i^2)$ for $i=1,\ldots,n$. Thus, we estimate the mean of the random variables $\{y_i^2\}_{i=1}^{n}$, which are not sub-Gaussian, even if the $\{y_i\}_{i=1}^{n}$ are sub-Gaussian random variables. In our experiments, the key  is that our analysis in Theorem \ref{thm4} can be used when $\{y_i\}_{i=1}^n$  are sub-Gaussian as in such case the random variables $\{y_i^2\}_{i=1}^{n}$ are sub-Exponential.

We now evaluate the validity of Thereom \ref{thm4}  in a simulation setting where we estimate the mean of the random variables but with $y_i - \theta_i^* $    following a Laplace distribution. Specifically,  we consider the same setting as in Section \ref{sec:2d}, but focusing on estimating $\theta^* \in \mathbb{R}^n$, and with data generated as
\[
y_i \,=\,   \theta^*_i  \,+\, \sqrt{ v_i^* }\epsilon_i
\]
with  $\epsilon_i \overset{ind} {\sim} \text{Laplace}(0,1)$.   The methods that we compare are the gaph fused lasso (GFL) defined in (\ref{eqn:gfl})  with the  Laplacian S. estimator defined in (\ref{eqn:estimator3}), and where we choose the tuning parameters as in  Section \ref{sec:tuning}.

The results on Table \ref{tab:ex4} seem to provide additional evidence in favor of Theorem  \ref{thm4}. In particular, the GFL outperforms Laplacian S., and the performance of GFL improves as $n$ increases which is what it is expected in light of Theorem \ref{tab:ex4}.

\begin{table}[t!]
	\centering
	\setlength{\tabcolsep}{7.7pt}
	\begin{scriptsize}
		\begin{tabular}{|c|cc|cc|cc|}
			\toprule
			& \multicolumn{2}{c}{Scenario 4} \vline  & \multicolumn{2}{c}{Scenario 5} \vline & \multicolumn{2}{c}{Scenario 6} \vline \\
			\hline
			$n$ &  Laplacian S.&  GFL&  Laplacian S.& GFL    & Laplacian S. & GFL\\
			\midrule
			$100^2$ & 1.12& \textbf{1.08}& 1.12 &\textbf{0.91}& $    2.7\times 10^{-3}$&  $\mathbf{9.87\times 10^{-4}}$\\ 
			$200^2$  & 1.13&   \textbf{0.52} &1.12 &\textbf{0.54}&$       2.7\times 10^{-3}$ & $\mathbf{2.89\times 10^{-4}}$\\ 
			$300^2$   & 1.10&   \textbf{0.23}& 1.12 &\textbf{0.32}& $ 2.7\times 10^{-3}$&  $\mathbf{1.44\times 10^{-4}}$\\ 
			$400^2$   &1.11 & \textbf{0.13}  & 1.11 &\textbf{0.18}&$2.7\times 10^{-3}$ &   $\mathbf{6.82\times 10^{-5}}$\\ 
			\bottomrule
		\end{tabular}
	\end{scriptsize}
	\caption{
		\label{tab:ex4}  Performance evaluations of  the competing methods  for the different settings described in the text. We report  $10$ multiplied by the average
		mean squared error,  averaging over 200 Monte Carlo simulations.
	}
\end{table}

\section{Proof of Theorem \ref{thm2}}

\begin{proof}
	First, we observe that
	\[
	\begin{array}{lll}
		\vert  v_0^* - \hat{v}\vert     & \leq & \displaystyle \bigg\vert    v_0^*    \,-\,    \frac{v^*_0}{2(\floor{n/2}-1)}  \sum_{i=1}^{  \floor{n/2}-1}  \{ \epsilon_{ \sigma(2i) }  -    \epsilon_{ \sigma(2i-1) }\}^2   \bigg\vert\,+\,\\
		& & \displaystyle \bigg\vert     \frac{v_0^*}{2(\floor{n/2}-1)}  \sum_{i=1}^{  \floor{n/2}-1}  \{ \epsilon_{ \sigma(2i) }  -  \epsilon_{ \sigma(2i-1) }\}^2   \,-\,   \\
		&&\displaystyle \frac{1}{2(\floor{n/2}-1)}  \sum_{i=1}^{  \floor{n/2}-1}  \{(v_0^*)^{1/2}\epsilon_{ \sigma(2i) }  -   (v_0^*)^{1/2}\epsilon_{ \sigma(2i-1) }  +   \theta^*_{ \sigma(2i) }   -    \theta^*_{ \sigma(2i-1) }  \}^2    \bigg\vert.
	\end{array}
	\]
	Next, using the identity $a^2 -  (a+b)^2 = -b(2a+b)$ we obtain that 
	\begin{equation}
		\label{eqn:e4}
		\begin{array}{lll}
			\vert  v_0^* - \hat{v}\vert     & \leq & \displaystyle \bigg\vert    v_0^*    \,-\,    \frac{v^*_0}{2(\floor{n/2}-1)}  \sum_{i=1}^{  \floor{n/2}-1}  \{\epsilon_{ \sigma(2i) }  -    \epsilon_{ \sigma(2i-1) }\}^2   \bigg\vert\,+\,\\
			& & \displaystyle    \frac{1}{2(\floor{n/2}-1)}  \sum_{i=1}^{  \floor{n/2}-1}  \vert\theta^*_{ \sigma(2i) }  -    \theta^*_{ \sigma(2i-1) }\vert\, \vert2 (v^*_0)^{1/2}(\epsilon_{ \sigma(2i) }  -    \epsilon_{ \sigma(2i-1) }) +(\theta^*_{ \sigma(2i) }  -    \theta^*_{ \sigma(2i-1) }) \vert\\
			& \leq& \displaystyle \bigg\vert    v_0^*    \,-\,    \frac{v^*_0}{2(\floor{n/2}-1)}  \sum_{i=1}^{  \floor{n/2}-1}  \{ \epsilon_{ \sigma(2i) }  -    \epsilon_{ \sigma(2i-1) }\}^2   \bigg\vert   \,+\,\\
			& & \displaystyle    \frac{\{ 4\|\epsilon\|_{\infty}(v_0^*)^{1/2}  + 2\|\theta^*\|_{\infty}  \}}{2(\floor{n/2}-1)}  \sum_{i=1}^{  \floor{n/2}-1}  \vert\theta^*_{ \sigma(2i) }  -    \theta^*_{ \sigma(2i-1) }\vert\\
			& \leq& \displaystyle v_0^*\bigg\vert    1    \,-\,    \frac{1}{2(\floor{n/2}-1)}  \sum_{i=1}^{  \floor{n/2}-1}  ( \epsilon_{ \sigma(2i) }  -    \epsilon_{ \sigma(2i-1) })^2   \bigg\vert   \,+\,\\
			& & \displaystyle    \frac{\{ 4\|\epsilon\|_{\infty}    (v_0^*)^{1/2}  + 4\|\theta^*\|_{\infty}   \}\|\nabla_G\theta^*\|_1  }{2(\floor{n/2}-1)}  \\ 	
		\end{array}
	\end{equation}
	where the last inequality follows from Lemma 1 in \cite{padilla2016dfs}. Finally, notice that 
	\begin{equation}
		\label{eqn:e5}
		\mathbb{E}\left[  \frac{1}{2(\floor{n/2}-1)}  \sum_{i=1}^{  \floor{n/2}-1}  \{ \epsilon_{ \sigma(2i) }  -    \epsilon_{ \sigma(2i-1) }\}^2   \right]\,=\, 1
	\end{equation}
	and 
	\begin{equation}
		\label{eqn:e6}
		\begin{array}{lll}
			\displaystyle \text{var}\left[ \frac{1}{2(\floor{n/2}-1)}  \sum_{i=1}^{  \floor{n/2}-1}  \{\epsilon_{ \sigma(2i) }  -    \epsilon_{ \sigma(2i-1) }\}^2   \right] &=  &\displaystyle \frac{1}{4(\floor{n/2}-1)^2}  \sum_{i=1}^{  \floor{n/2}-1} \text{var}\left[   \{ \epsilon_{ \sigma(2i) }  -    \epsilon_{ \sigma(2i-1) }\}^2   \right] \\
			& \leq &\displaystyle \frac{1}{4(\floor{n/2}-1)^2}  \sum_{i=1}^{  \floor{n/2}-1} \mathbb{E}\left[   \{ \epsilon_{ \sigma(2i) }  -    \epsilon_{ \sigma(2i-1) }\}^4  \right] \\
			& \leq &\displaystyle \frac{1}{4(\floor{n/2}-1)^2}  \sum_{i=1}^{  \floor{n/2}-1}  \mathbb{E}\left\{  8 \epsilon_{ \sigma(2i) }^4  +8  \epsilon_{ \sigma(2i-1) }^4  \right\} \\
			& \leq  &\displaystyle \frac{4   }{(\floor{n/2}-1)}   \underset{i=1,\ldots,n}{\sup}\,  \mathbb{E}(\epsilon_i^4) \\
		\end{array}
	\end{equation}
	where the second inequality follows from the inequality $(a+b)^4 \leq  8 a^4 + 8b^4$. Combining (\ref{eqn:e4})--(\ref{eqn:e6}) with the Chebyshev's inequality we conclude the proof.
\end{proof}

\section{A general upper bound}
\label{sec:general_bound}

\begin{theorem}
	\label{thm3}
	Consider data $\{o_i \}_{i=1}^n$  generated as $o_i = \beta_i^*+\varepsilon_i$ for some $\beta^*\in  \mathbb{R}^n$ and $\varepsilon_1,\ldots,\varepsilon_n$ independent random variables satisfying  satisfying $\mathbb{E}(\varepsilon_i) = 0$ and  
	\[
	\underset{i = 1,\ldots, n}{\max}\,\mathbb{E}(\varepsilon_i^4) \,=\, O(1).
	\]
	Let  $\hat{\beta}$ be defined as 
	\[
	\hat{\beta} \,:=\,\underset{\beta \in \mathbb{R}^n}{\arg \min}\left\{ \frac{1}{2}\sum_{i=1}^{n}(o_i - \beta_i)^2   +   \lambda\sum_{ (i,j)\in E } \vert\beta_i -\beta_j\vert   \right\}.
	\]
	Let $\eta > 0$. Then   for any  sequence $U_n>0$ and for  any $\eta>0$,  it holds that 
	
	\[
	\begin{array}{lll}
		\mathrm{pr}( \| \hat{\beta} -\beta^*\|  >\eta)    & \leq&\displaystyle\frac{16 n^{1/2}   \underset{i=1,\ldots,n}{\max}   \{\mathbb{E}( \varepsilon_i^4  )\}^{1/4}      \{  \mathrm{pr}(\vert  \varepsilon_i\vert > U_n) \}^{1/4}}{\eta}\,+\,\\
		& &\displaystyle \frac{16}{\eta^2} \,\mathbb{E}\left[     \underset{\beta \in  \mathbb{R}^n \,:\, \|\beta - \beta^*\|\leq \eta , \, \| \nabla_G\beta\|_1 \leq 5 \|\nabla_G\beta^*\|_1 }{\sup}  \sum_{i=1}^n   \xi_i\varepsilon_i  1_{  \{\vert  \varepsilon_i\vert \leq U_n\}   } (\beta_i - \beta_i^*)    \right],
	\end{array}
	\]
	where  $\xi_1,\ldots,\xi_n$ are independent Rademacher random variables independent of $\{\varepsilon_i\}_{i=1}^n$, provided that
	\begin{equation}
		\label{eqn:lambda}
		\lambda \,=\,   \frac{\eta^2}{4\| \nabla_G \beta^*\|_1 }.
	\end{equation}
\end{theorem}

\begin{proof}
	First, notice that by convexity and the basic inequality we have that 
	\begin{equation}
		\label{eqn:e7}
		\frac{1}{2}\sum_{i=1}^{n}(o_i - \beta_i)^2   +   \lambda\| \nabla_G \beta\|_1\,\leq \, \frac{1}{2}\sum_{i=1}^{n}(o_i - \beta_i^*)^2   +   \lambda\| \nabla_G \beta^*\|_1
	\end{equation}
	for any $\beta \in \Lambda \,:=\, \{  s  \hat{\beta} +(1-s)\beta^* \,:\, s\in [0,1]  \}$.  Then
	\begin{equation}
		\label{eqn:e8}
		\| \nabla_G \beta\|_1 \,\leq \, 	 \| \nabla_G \beta\|_1 \,+\,  \frac{\|\beta - \beta^*\|^2}{2\lambda}  \,\leq \, \frac{\varepsilon^{\top}  (\beta - \beta^*) }{\lambda} +     \| \nabla_G \beta^*\|_1
	\end{equation}
	for all $\beta \in \Lambda$.  Ths implies
	\begin{equation}
		\label{eqn:e9}
		\begin{array}{lll}
			\displaystyle 	 \| \nabla_G (\beta -\beta^*)\|_1   & \leq & \displaystyle  	 \| \nabla_G \beta \|_1    \,+\,	 \| \nabla_G \beta^*\|_1  \\ 
			&\leq &  \displaystyle  \frac{\varepsilon^{\top}  (\beta - \beta^*) }{\lambda} +     2\| \nabla_G \beta^*\|_1,
		\end{array}
	\end{equation}
	for all $\beta \in \Lambda$. 
	
	Next, let $\beta \in \Lambda$ and suppose that   $\|  \beta - \beta^*\|^2 \leq  \eta^2 $, and $\| \nabla_G \beta\|_1   \geq  5 \| \nabla_G \beta^*\|_1  $. Then 
	\[
	\| \nabla_G (\beta -\beta^*)\|_1   \,\geq \, \| \nabla_G \beta \|_1  - \| \nabla_G \beta^*\|_1  \, \geq  \,4 \| \nabla_G \beta^*\|_1 .
	\]
	Hence, setting 
	\[
	s \,:=\, \frac{4  \| \nabla_G \beta^*\|_1  }{	 \| \nabla_G (\beta -\beta^*)\|_1 } ,   
	\]
	clearly $s\in [0,1]$, and we let 
	\[
	\beta^{\prime}  :=  s \beta + (1-s)\beta^* \in \Lambda.  
	\]
	Then
	\[
	\| \beta^{\prime} - \beta^*\|^2 \,\leq \,    \| \beta - \beta^*\|^2 \,\leq \, \eta^2,
	\]
	and 
	\[
	\begin{array}{lll}
		\| \nabla_G (\beta^{\prime} -\beta^*)\|_1  &=&   s\| \nabla_G (\beta -\beta^*)\|_1\\ 
		&=  &4 \| \nabla_G \beta^*\|_1.
	\end{array}
	\]
	Therefore, from (\ref{eqn:e9}), 
	\[
	4 \| \nabla_G \beta^*\|_1 \,  =\,  	\| \nabla_G (\beta^{\prime} -\beta^*)\|_1 \,\leq \, \frac{\varepsilon^{\top}  (\beta^{\prime} - \beta^*) }{\lambda} +     2\| \nabla_G \beta^*\|_1
	\]
	which implies 
	\[
	2 \| \nabla_G \beta^*\|_1 \,  \,\leq \, \frac{\varepsilon^{\top}  (\beta^{\prime} - \beta^*) }{\lambda}.
	\]
	Hence, if we take
	\[
	\lambda \,=\,   \frac{\eta^2}{4\| \nabla_G \beta^*\|_1 }
	\]
	we obtain 
	\[
	\frac{\eta^2}{2}\,\leq\, \varepsilon^{\top} ( \beta^{\prime} - \beta^* ).
	\]
	As a result, the events
	\[
	\Omega_1 \,:=\, \left\{  \underset{\beta \in \Lambda \,:\,   \|\beta- \beta^*\| \leq \eta }{\sup}\,   \|\nabla_G \beta\|_1   \,\geq \, 5\|\nabla_G \beta^*\|_1 \right\}
	\]
	and 
	\[
	\Omega_2 \,:=\,\left\{  \underset{\beta \in \Lambda \,:\,   \|\beta- \beta^*\| \leq \eta,\,  \|\nabla_G(\beta-\beta^*)\|_1\leq 4  \|\nabla_G \beta^*\|_1 }{\sup}\,   \varepsilon^{\top}(\beta-\beta^*)\,\geq \, \frac{\eta^2}{2} \right\}
	\]
	satisfy that  $\Omega_1 \subset  \Omega_2$. And so, 
	\begin{equation}
		\label{eqn:omega}
		\mathrm{pr}(\Omega_1 )\,\leq\, \mathrm{pr}(\Omega_2).
	\end{equation}
	
	Next, suppose that    $\| \hat{\beta} - \beta^*\| >\eta$. Then there exists  $\beta \in \Lambda$ such that $\| \beta - \beta^*\| =\eta$ and so (\ref{eqn:e8})  implies  that
	\[
	\frac{\eta^2}{2} \,\leq \,  \varepsilon^{\top} (\beta - \beta^*) \,+\,  \lambda \|\nabla_G \beta^*\|_1 \,-\,\lambda \|\nabla_G \beta \|_1.
	\]
	Hence, given our choice of $\lambda$, we obtain that 
	\[
	\frac{\eta^2}{4} \,\leq \,  \varepsilon^{\top}(\beta - \beta^*),
	\]
	for some $\beta\in \Lambda$, provided that $\| \hat{\beta} - \beta^*\| >\eta$.

	The above  implies that 
	\begin{equation}
		\label{eqn:e10}
		\begin{array}{lll}
			\mathrm{pr}( \| \hat{\beta} -\beta^*\|  >\eta)   & \leq &\displaystyle   \mathrm{pr}( \{  \| \hat{\beta} -\beta^*\|  >\eta \}  \cap  \Omega_1^c )   \,+ \, \mathrm{pr}(\Omega_1)     \\
			& \leq&\displaystyle  \mathrm{pr}( \{  \| \hat{\beta} -\beta^*\|  >\eta \}  \cap  \Omega_1^c )   \,+ \, \mathrm{pr}(\Omega_2)     \\
			& \leq& \displaystyle   \mathrm{pr}\bigg\{\underset{\beta \in \Lambda   \,:\, \|\beta - \beta^*\|\leq \eta , \, \| \nabla_G\beta\|_1 \leq 5 \|\nabla_G\beta^*\|_1 }{\sup}  \varepsilon^{\top}(\beta - \beta^*) \,\geq \, \frac{\eta^2}{4} \bigg\}   \,+ \, \mathrm{pr}(\Omega_2)     \\
			& \leq&\displaystyle   2\,\mathrm{pr}\bigg\{ \underset{\beta \in \Lambda   \,:\, \|\beta - \beta^*\|\leq \eta , \, \| \nabla_G\beta\|_1 \leq 5 \|\nabla_G\beta^*\|_1 }{\sup}  \varepsilon^{\top}(\beta - \beta^*)  \,\geq\, \frac{\eta^2}{4}\bigg\} \\
			& \leq& \displaystyle \frac{8}{\eta^2} \,\mathbb{E}\left\{    \underset{\beta \in \Lambda   \,:\, \|\beta - \beta^*\|\leq \eta , \, \| \nabla_G\beta\|_1 \leq 5 \|\nabla_G\beta^*\|_1 }{\sup}  \varepsilon^{\top}(\beta - \beta^*) \right\}\\
			& =:&  A_1,
		\end{array}
	\end{equation}
	where the second inequality follows from (\ref{eqn:omega}),  and third from the discussion above,  the fourth from the definition of $\Omega_2$, and the last inequality from Markov's inequality. Next,  notice that 
	\[
	\begin{array}{lll}
		A_1     & \leq&    \displaystyle \frac{8}{\eta^2} \,\mathbb{E}\left[     \underset{\beta \in \Lambda   \,:\, \|\beta - \beta^*\|\leq \eta , \, \| \nabla_G\beta\|_1 \leq 5 \|\nabla_G\beta^*\|_1 }{\sup}  \sum_{i=1}^n  \varepsilon_i  1_{  \{\vert  \varepsilon_i\vert \leq  U_n\}   } (\beta_i - \beta_i^*)    \right]\,+\,\\
		& &  \displaystyle \frac{8}{\eta^2} \,\mathbb{E}\left[    \underset{\beta \in \Lambda   \,:\, \|\beta - \beta^*\|\leq \eta , \, \| \nabla_G\beta\|_1 \leq 5 \|\nabla_G\beta^*\|_1 }{\sup}  \sum_{i=1}^n  \varepsilon_i  1_{  \{\vert  \varepsilon_i\vert >  U_n\}   } (\beta_i - \beta_i^*)    \right]\\
		& \leq&    \displaystyle \frac{8}{\eta^2} \,\mathbb{E}\left\{    \underset{\beta \in \Lambda   \,:\, \|\beta - \beta^*\|\leq \eta , \, \| \nabla_G\beta\|_1 \leq 5 \|\nabla_G\beta^*\|_1 }{\sup}  \sum_{i=1}^n  (\varepsilon_i  1_{  \{\vert  \varepsilon_i\vert \leq U_n\}   } -   \mathbb{E}[\varepsilon_i  1_{  \{\vert  \varepsilon_i\vert \leq  U_n\}   } ]  )(\beta_i - \beta_i^*)    \right\}\,+\,\\
		&&   \displaystyle \frac{8}{\eta^2} \,   \underset{\beta \in \Lambda   \,:\, \|\beta - \beta^*\|\leq \eta , \, \| \nabla_G\beta\|_1 \leq 5 \|\nabla_G\beta^*\|_1 }{\sup}  \sum_{i=1}^n  \mathbb{E}[\varepsilon_i  1_{  \{\vert  \varepsilon_i\vert \leq U_n\}   } ]  (\beta_i - \beta_i^*)   \,+\,\\
		& &\displaystyle \frac{8}{\eta^2} \,\mathbb{E}\left[     \underset{\beta \in \Lambda   \,:\, \|\beta - \beta^*\|\leq \eta , \, \| \nabla_G\beta\|_1 \leq 5 \|\nabla_G\beta^*\|_1 }{\sup}  \sum_{i=1}^n  \varepsilon_i  1_{  \{\vert  \varepsilon_i\vert >  U_n\}   } (\beta_i - \beta_i^*)    \right]\\
		& =:&A_2+A_3+A_4,
	\end{array}
	\]
	next we proceed to bound $A_2$, $A_3$ and $A_4$.  To bound $A_3$, notice that since $\mathbb{E}(\varepsilon_i)=0$  then 
	\[
	\begin{array}{lll}
		A_3    & =& \displaystyle \frac{8}{\eta^2} \,   \underset{\beta \in \Lambda   \,:\, \|\beta - \beta^*\|\leq \eta , \, \| \nabla_G\beta\|_1 \leq 5 \|\nabla_G\beta^*\|_1 }{\sup}  \sum_{i=1}^n  - \mathbb{E}[\varepsilon_i  1_{  \{\vert  \varepsilon_i\vert > U_n\}   } ]  (\beta_i - \beta_i^*). 
	\end{array}
	\]
	Hence,
	\begin{equation}
		\label{eqn:e11}
		\begin{array}{lll}
			A_3    & \leq& \displaystyle \frac{8n^{1/2}}{\eta} \,  \underset{i=1,\ldots,n}{\max}  \vert \mathbb{E}[\varepsilon_i  1_{  \{\vert  \varepsilon_i\vert >U_n\}   } ] \vert \\
			& \leq&\displaystyle \frac{8n^{1/2}}{\eta} \,   \underset{i=1,\ldots,n}{\max}  \left(\mathbb{E}(\varepsilon_i^2  ) \,\mathbb{E}[1_{  \{\vert  \varepsilon_i\vert >U_n\}   } ]  \right)^{1/2} \\
			& = & \displaystyle   \frac{8 n^{1/2}}{\eta}\left\{  \underset{i=1,\ldots,n}{\max}  \mathbb{E}(\varepsilon_i^2  )  \,  \mathrm{pr}(\vert  \varepsilon_i\vert >U_n)   \right\}^{1/2},
		\end{array}
	\end{equation}
	where the first and second inequalities follow from  Cauchy–Schwarz inequality.
	
	To bound $A_4$, we observe that 
	\begin{equation}
		\label{eqn:e12}
		\begin{array}{lll}
			A_4&\leq &\displaystyle \frac{8}{\eta^2} \, \mathbb{E}\bigg(     \underset{\beta \in \Lambda   \,:\, \|\beta - \beta^*\|\leq \eta , \, \| \nabla_G\beta\|_1 \leq 5 \|\nabla_G\beta^*\|_1 }{\sup}   \bigg[\sum_{i=1}^n  \varepsilon_i^2 1_{  \{\vert  \varepsilon_i\vert > U_n\}   }  \bigg]^{1/2} \|\beta - \beta^*\|  \bigg)\\
			& \leq& \displaystyle\frac{8}{\eta} \, \mathbb{E}\bigg(      \bigg[\sum_{i=1}^n  \varepsilon_i^2 1_{  \{\vert  \varepsilon_i\vert > U_n\}   }  \bigg]^{1/2}  \bigg)\\
			& \leq&\displaystyle\frac{8}{\eta} \,  \bigg( \mathbb{E}\bigg[\sum_{i=1}^n  \varepsilon_i^2 1_{  \{\vert  \varepsilon_i\vert >  U_n\}   } \bigg]   \bigg)^{1/2}\\
			& \leq&\displaystyle\frac{8 n^{1/2}}{\eta} \,  \bigg( \underset{i=1,\ldots,n}{\max} \mathbb{E}\bigg[  \varepsilon_i^2 1_{  \{\vert  \varepsilon_i\vert >  U_n\}   } \bigg]   \bigg)^{1/2}\\
			& \leq&\displaystyle\frac{8 n^{1/2}    \max_{i=1,\ldots,n}\{\mathbb{E}( \varepsilon_i^4  )\}^{1/4}      \{ \mathrm{pr}(\vert  \varepsilon_i\vert >  U_n) \}^{1/4}}{\eta}. 
		\end{array}
	\end{equation}
	
	Let us now proceed to bound $A_2$.   Let $\varepsilon_1^{\prime},\ldots,\varepsilon_n^{\prime}$ independent copies of $\varepsilon_1,\ldots,\varepsilon_n$. Then 
	for independent Rademacher random variables $\xi_1,\ldots,\xi_n$, it holds that 
	\begin{equation}
		\label{eqn:e13}
		\begin{array}{lll}
			A_2   & \leq&\displaystyle \frac{8}{\eta^2} \, \mathbb{E}\left(     \underset{\beta \in \Lambda   \,:\, \|\beta - \beta^*\|\leq \eta , \, \| \nabla_G\beta\|_1 \leq 5 \|\nabla_G\beta^*\|_1 }{\sup}  \sum_{i=1}^n  [ \varepsilon_i  1_{  \{\vert  \varepsilon_i\vert \leq U_n\}   }  -   \varepsilon_i^{\prime }  1_{  \{\vert  \varepsilon_i^{\prime}\vert \leq U_n\}   }   ](\beta_i - \beta_i^*)    \right)\\
			& =&\displaystyle \frac{8}{\eta^2} \,\mathbb{E}\left(     \underset{\beta \in \Lambda   \,:\, \|\beta - \beta^*\|\leq \eta , \, \| \nabla_G\beta\|_1 \leq 5 \|\nabla_G\beta^*\|_1 }{\sup}  \sum_{i=1}^n  \xi_i[ \varepsilon_i  1_{  \{\vert  \varepsilon_i\vert \leq U_n\}   }  -   \varepsilon_i^{\prime }  1_{  \{\vert  \varepsilon_i^{\prime}\vert \leq U_n\}   }   ](\beta_i - \beta_i^*)    \right)\\
			& \leq &\displaystyle \frac{8}{\eta^2} \,\mathbb{E}\left(     \underset{\beta \in \Lambda   \,:\, \|\beta - \beta^*\|\leq \eta , \, \| \nabla_G\beta\|_1 \leq 5 \|\nabla_G\beta^*\|_1 }{\sup}  \sum_{i=1}^n  \xi_i[ \varepsilon_i  1_{  \{\vert  \varepsilon_i\vert \leq U_n\}   }    ](\beta_i - \beta_i^*)    \right)   \,+\,\\
			& &\displaystyle \frac{8}{\eta^2} \,\mathbb{E}\left(     \underset{\beta \in \Lambda   \,:\, \|\beta - \beta^*\|\leq \eta , \, \| \nabla_G\beta\|_1 \leq 5 \|\nabla_G\beta^*\|_1 }{\sup}  \sum_{i=1}^n  -\xi_i[    \varepsilon_i^{\prime }  1_{  \{\vert  \varepsilon_i^{\prime}\vert \leq U_n\}   }   ](\beta_i - \beta_i^*)    \right)\\
			& = &\displaystyle \frac{16}{\eta^2} \,\mathbb{E}\left[    \underset{\beta \in \Lambda   \,:\, \|\beta - \beta^*\|\leq \eta , \, \| \nabla_G\beta\|_1 \leq 5 \|\nabla_G\beta^*\|_1 }{\sup}  \sum_{i=1}^n  \xi_i \varepsilon_i  1_{  \{\vert  \varepsilon_i\vert \leq U_n\}   }    (\beta_i - \beta_i^*)    \right] \\
			& \leq& \displaystyle \frac{16}{\eta^2} \,\mathbb{E}\left[     \underset{\beta \in  \mathbb{R}^n \,:\, \|\beta - \beta^*\|\leq \eta , \, \| \nabla_G\beta\|_1 \leq 5 \|\nabla_G\beta^*\|_1 }{\sup}  \sum_{i=1}^n   \xi_i\varepsilon_i  1_{  \{\vert  \varepsilon_i\vert \leq U_n\}   } (\beta_i - \beta_i^*)    \right].
		\end{array}
	\end{equation}
	The claim then follows. 
	
\end{proof}

\section{Assumptions for $K$-NN graph for  Theorem \ref{thm4} }
\label{sec:assump}

We start by explicitly defining the construction of the $K$-NN graph. Specifically, $(i,j) \in E$ if and only if   $x_j$ is among the $K$-nearest neighbors  (with respect to the metric  $\mathrm{dist}(\cdot$) of $x_i$, or vice versa.

We now state the assumptions from \cite{padilla2018adaptive} needed for Theorem \ref{thm4}. Throughout $(\mathcal{X},\mathrm{dist})$ is a metric space with Borel sets $\mathcal{B}(\mathcal{X})$.

\begin{assumption}
	\label{cond1}
	The covariates  $\{x_i\}_{i=1}^n$ are independent draws from a density $p$,  with respect to the measurable space $(\mathcal{X},\mathcal{B}(\mathcal{X}), \mu )$, with support $\mathcal{X}$. Furthermore, the density $p$ satisfies $0< p_{\min} < p(x) < p_{\max} $ for all $x \in \mathcal{X}$, where $p_{\min}$ and $p_{\max}$ are constants. 
\end{assumption}

\begin{assumption}
	\label{cond2}
	The base measure satisfies 
	\[
	c_1 r^d \,\leq \,\mu\left[ \{  q  \in  \mathcal{X}\,:\, \mathrm{dist}(q,x)\leq r   \} \right]\,\leq \, c_2 r^d
	\]
	for all $x \in \mathcal{X}$, and all $0<r<r_0$, where  $c_1$, $c_2$ and $r_0$ are all positive constants, and $d \in \mathbb{N} \backslash \{0\}$ is the intrinsic dimension of $\mathcal{X}$. 
\end{assumption}

\begin{assumption}
	\label{cond3}
	There exists a homeomorphism (a continuous bijection with a continuous inverse) $h \,:\, \mathcal{X} \rightarrow  [0,1]^d$ such that 
	\[
	L_{\min} \mathrm{dist}(x,x^{\prime}) \,\leq \,  \| h(x) -h(x^{\prime})\| \,\leq \,L_{\max} \mathrm{dist}(x,x^{\prime}),\,\,\,\forall x,x^{\prime} \in \mathcal{X}, 
	\]
	for some positive constants $L_{\min}$ and $L_{\max}$.
\end{assumption}

For a set $S \subset [0,1]^d$, we let 
\[
B_{t}(S) \,:=\,\{  q \in [0,1]^d \,:\, \|q-q^{\prime}\|\leq t \,\,\text{for some }\,\, q^{\prime } \in S    \}. 
\]
With this notation, we state our next assumption.

\begin{assumption}
	\label{cond4}
	\textbf{[Piecewise Lipschitz].}	The parameter $\beta^*$ satisfies that $\beta_i^* = f_0(x_i)$ for $i=1,\ldots,n$ for some function $f_0$, where the following holds for the function $f_0$.
	\begin{enumerate}
		\item  $f_0$ is bounded.
		\item Let $\partial [0,1]^d$   be the boundary of $[0,1]^d$, and let $\Omega_t = [0,1]^d \backslash B_{t}(\partial [0,1]^d) $. We assume that there exists a set $\mathcal{S}$ such that: 
		\begin{enumerate}
			\item The set $\mathcal{S}$ has Lebesgue measure zero.
			
			\item For some constants $C_{\mathcal{S}}, t_0>0$, we have that 
			\[
			\mu\left[   h^{-1}\left\{B_t(\mathcal{S})\cup   ([0,1]^d \backslash \Omega_t)\right\} \right] \,\leq \,  C_{\mathcal{S}}t
			\] 
			for all $0<t<t_0$.
			\item There exists a positive constant $L_0$ such that if $z$ and $z^{\prime}$ belong to the same connected component of $\Omega_t \backslash B_t(\mathcal{S})$ then 
			\[
			\vert  f_0\circ h^{-1}(z)- f_0\circ h^{-1}(z^{\prime}) \vert \,\leq \, L_0 \|z-z^{\prime}\|.
			\]
		\end{enumerate}
	\end{enumerate}
\end{assumption}

\textbf{Notation:   } We denote as $\mathcal{F}(L_0)$ the set of functions $f \,:\,[0,1]^d \rightarrow \mathbb{R}$ that satisfy Assumption \ref{cond4} with $L_0$ and such that 
\[
\underset{x \in [0,1]^d}{\sup}\,\vert f(x)\vert   \,\leq \, L_0,
\]
and $C_{\mathcal{S}} \leq L_0$.

\section{Assumptions for $K$-NN graph for  Theorem \ref{thm5} }
\label{sec:assump2}

We assume that the covariates $\{x_i \}_{i=1}^n$ satisfy Assumptions \ref{cond1}--\ref{cond3}. In addition, we assume that Assumption \ref{cond4} holds replacing $\beta^*$ with both $v^*$ and $\theta^*$.

\section{Proof of Theorem \ref{thm4} }

\begin{proof}
	\textbf{Proof  of (\ref{eqn:rate1}):}  First, let $G^{\prime}$  be a chain graph corresponding to a DFS ordering in $G$. Based of Theorem \ref{thm3}, we first need to bound
	\begin{equation}
		\label{eqn:B1}
		B_1:=  \displaystyle \frac{16}{\eta^2} \, \mathbb{E}\left[     \underset{\beta \in \mathbb{R}^n  \,:\, \|\beta - \beta^*\|\leq \eta , \, \| \nabla_G\beta\|_1 \leq 5 \|\nabla_G\beta^*\|_1 }{\sup}  \sum_{i=1}^n   \xi_i\varepsilon_i  1_{  \{\vert  \varepsilon_i\vert \leq  U_n\}   } (\beta_i - \beta_i^*)    \right].
	\end{equation}
	To bound this, we recall Lemma 1 in \cite{padilla2016dfs} which implies that $\| \nabla_{G^{\prime}} \beta\|_1 \leq  2 \| \nabla_G\beta\|_1$ for all $\beta \in \mathbb{R}^n$. Hence, 
	\begin{equation}
		\label{eqn:e14}
		\begin{array}{lll}
			B_1 &\leq& \displaystyle \frac{16}{\eta^2} \,\mathbb{E}\left[     \underset{\beta \in \mathbb{R}^n \,:\, \|\beta - \beta^*\|\leq \eta , \, \| \nabla_{G^{\prime}}\beta\|_1 \leq 10 \|\nabla_G\beta^*\|_1 }{\sup}  \sum_{i=1}^n   \xi_i\varepsilon_i  1_{  \{\vert  \varepsilon_i\vert \leq U_n\}   } (\beta_i - \beta_i^*)    \right]\\
			&\leq& \displaystyle \frac{16}{\eta^2} \mathbb{E}\left(  \mathbb{E}\left[     \underset{\beta \in \mathbb{R}^n   \,:\, \|\beta - \beta^*\|\leq \eta , \, \| \nabla_{G^{\prime}}\beta\|_1 \leq 10 \|\nabla_G\beta^*\|_1 }{\sup}  \sum_{i=1}^n   \xi_i\varepsilon_i  1_{  \{\vert  \varepsilon_i\vert \leq U_n\}   } (\beta_i - \beta_i^*)  \, \bigg|  \varepsilon   \right]  \right)\\
			&=& \displaystyle \frac{16U_n}{\eta^2} \mathbb{E}\left(  \mathbb{E}\left[     \underset{\beta \in \mathbb{R}^n   \,:\, \|\beta - \beta^*\|\leq \eta , \, \| \nabla_{G^{\prime}}\beta\|_1 \leq 10 \|\nabla_G\beta^*\|_1 }{\sup}  \sum_{i=1}^n   \xi_i\frac{\varepsilon_i }{U_n} 1_{  \{\vert  \varepsilon_i\vert \leq  U_n\}   } (\beta_i - \beta_i^*)  \, \bigg|  \varepsilon   \right]  \right).
		\end{array}
	\end{equation}
	Then 
	\[
	\begin{array}{l}
		\displaystyle 	 \mathbb{E}\bigg[    \underset{\beta \in \mathbb{R}^n   \,:\, \|\beta - \beta^*\|\leq \eta , \, \| \nabla_{G^{\prime}}\beta\|_1 \leq 10 \|\nabla_G\beta^*\|_1 }{\sup}  \sum_{i=1}^n   \xi_i\frac{\varepsilon_i }{U_n} 1_{  \{\vert  \varepsilon_i\vert \leq  U_n\}   } (\beta_i - \beta_i^*)  \, \bigg|  \varepsilon   \bigg] \\
		\displaystyle  \leq \displaystyle 	 \mathbb{E}\bigg[     \underset{\beta \in \mathbb{R}^n \,:\, \|\beta - \beta^*\|\leq \eta , \, \| \nabla_{G^{\prime}}(\beta -\beta^*)\|_1 \leq 11 \|\nabla_G\beta^*\|_1 }{\sup}  \sum_{i=1}^n   \xi_i\frac{\varepsilon_i }{U_n} 1_{  \{\vert  \varepsilon_i\vert \leq  U_n\}   } (\beta_i - \beta_i^*)  \, \bigg|  \varepsilon   \bigg] \\
		\displaystyle  \leq \displaystyle 	 \mathbb{E}\bigg[     \underset{\beta \in \mathbb{R}^n  \,:\, \|\beta \|\leq \eta , \, \| \nabla_{G^{\prime}}\beta \|_1 \leq 11 \|\nabla_G\beta^*\|_1 }{\sup}  \sum_{i=1}^n   \xi_i\frac{\varepsilon_i }{U_n} 1_{  \{\vert  \varepsilon_i\vert \leq  U_n\}   } \beta_i   \, \bigg|  \varepsilon   \bigg] \\
		\displaystyle  \leq \displaystyle 	 \mathbb{E}\bigg[     \underset{\beta \in  \mathbb{R}^n  \,:\, \|\beta \|\leq \eta , \, \| \nabla_{G^{\prime}}\beta \|_1 \leq 11 \|\nabla_G\beta^*\|_1 }{\sup}  \sum_{i=1}^n   \xi_i \beta_i   \, \bigg|  \varepsilon   \bigg] \\
	\end{array}
	\]
	where  the last inequality follows from Theorem 4.12 in \cite{ledoux1991probability}. As a result, letting  $\tilde{\xi}_i \overset{\text{ind}}{\sim} N(0,1)$ for $i=1,\ldots,n$, we obtain that 
	\[
	\begin{array}{l}
		\displaystyle 	 \mathbb{E}\bigg[     \underset{\beta \in \mathbb{R}^n  \,:\, \|\beta - \beta^*\|\leq \eta , \, \| \nabla_{G^{\prime}}\beta\|_1 \leq 10 \|\nabla_G\beta^*\|_1 }{\sup}  \sum_{i=1}^n   \xi_i\frac{\varepsilon_i }{U_n} 1_{  \{\vert  \varepsilon_i\vert \leq  U_n\}   } (\beta_i - \beta_i^*)  \, \bigg|  \varepsilon   \bigg] \\
		\displaystyle  \leq \displaystyle 	\mathbb{E}\bigg(     \underset{\beta \in  \mathbb{R}^n  \,:\, \|\beta \|\leq \eta , \, \| \nabla_{G^{\prime}}\beta \|_1 \leq 11 \|\nabla_G\beta^*\|_1 }{\sup}  \sum_{i=1}^n   \xi_i \beta_i   \,   \bigg) \\
		\displaystyle  \leq \displaystyle 	 \left(\frac{\pi}{2}\right)^{1/2}\mathbb{E}\bigg(     \underset{\beta \in \mathbb{R}^n  \,:\, \|\beta \|\leq \eta , \, \| \nabla_{G^{\prime}}\beta \|_1 \leq 11 \|\nabla_G\beta^*\|_1 }{\sup}  \sum_{i=1}^n   \tilde{\xi}_i\beta_i   \,  \bigg) \\
		\displaystyle \leq   \, C_1 \bigg\{ \eta\bigg(  \frac{11 \|\nabla_G\beta^*\|_1   n^{1/2} }{\eta} \bigg)^{1/2}   \,+\,\eta \{\log (en)\}^{1/2}\bigg\}
	\end{array}
	\]
	where the second inequality follows fromt a well known fact bounding Rademacher Width by Gaussian Width; e.g see Page
	132 in \cite{wainwright2019high}, and  the last by Lemma B.1 from \cite{guntuboyina2020adaptive}. This implies that
	\[
	B_1 \,\leq\, \frac{16U_n}{\eta^2} \, C_1 \bigg[  \eta\bigg(  \frac{11 \|\nabla_G\beta^*\|_1   n^{1/2} }{\eta} \bigg)^{1/2}   \,+\,\eta\{ \log (en)\}^{1/2}\bigg].
	\]
	
	Therefore, given $a \in (0,1)$, we let  
	\[
	\eta \,:= \,     \frac{4}{a} \left[16^{2/3}n^{1/6}(\log n)^{1/6}  (11)^{1/3}\|\nabla_G\beta^*\|_1^{1/3}    U_n^{2/3}   (C_1)^{2/3}\,+\,  16C_1U_n \{\log (en )\}^{1/2}\right]
	\]
	and $\lambda $ as in (\ref{eqn:lambda}).  Hence,  
	\begin{equation}
		\label{eqn:e15}
		\begin{array}{lll}
			B_1&\leq & \displaystyle \frac{16  C_1 U_n (11\|\nabla_G\beta^*\|_1 )^{1/2} n^{1/4}}{ \eta^{3/2}  }\,+\,  \frac{16 C_1 U_n  \{\log(en)\}^{1/2} }{\eta}\\
			& \leq& \displaystyle  \frac{a^{3/2}}{4^{3/2}} \frac{16 C_1 U_n (11\|\nabla_G\beta^*\|_1 )^{1/2} n^{1/4}}{\left\{16^{2/3}n^{1/6}(\log n)^{1/6}  (11)^{1/3}\|\nabla_G\beta^*\|_1^{1/3}    U_n^{2/3}   (C_1)^{2/3}\right\}^{3/2} } \,+\, \frac{a}{4}
			\\
			&\leq &\displaystyle \frac{a}{2}.
		\end{array}
	\end{equation}

	Furthermore, by Theorem \ref{thm3}, we must bound
	\begin{equation}
		\label{eqn:b2}
		B_2 \,:=\,\displaystyle\frac{16 n^{1/2} \,  \underset{i=1,\ldots,n}{\max} \{\mathbb{E}( \varepsilon_i^4  )\}^{1/4}      \{  \mathrm{pr}(\vert  \varepsilon_i\vert >  U_n) \}^{1/4}}{\eta}. 
	\end{equation}
	However, given our definition of $\eta$, we obtain that 
	\[
	\begin{array}{lll}
		B_2  & \leq&\displaystyle \frac{a}{4} \,\frac{16n^{1/2}\,\underset{i=1,\ldots,n}{\max}   \{\mathbb{E}( \varepsilon_i^4  )\}^{1/4}      \{  \mathrm{pr}(\vert  \varepsilon_i\vert > U_n) \}^{1/4}}{\left[16^{2/3}n^{1/6}(\log n)^{1/6}  (11)^{1/3}\|\nabla_G\beta^*\|_1^{1/3}    U_n^{2/3}   (C_1)^{2/3}\,+\,  16C_1U_n \{\log (en )\}^{1/2}\right]}\\
		& \leq& \displaystyle \frac{a}{4}
	\end{array}
	\]
	where the last inequality follows from (\ref{eqn:c1}). The conclusion of the Theorem follows from Theorem \ref{thm3}.

	\textbf{Proof of rate (\ref{eqn:rate2}):}  As before, we first bound $B_1$ as defined in (\ref{eqn:B1}). Towards that end, let  $\nabla_G^{+}$ the pseudo inverse of $\nabla_G$, and  $\Pi$ the orthogonal projection onto the span of $(1,\ldots,1)^{\top} \in \mathbb{R}^n$.  Then
	notice that
	\[
	\begin{array}{lll}
		B_1&\leq &  \displaystyle \frac{16U_n}{\eta^2} \,\mathbb{E}\left(     \underset{\delta \in \mathbb{R}^n \,:\, \|\delta\|\leq \eta , \, \| \nabla_G\delta\|_1 \leq 6 \|\nabla_G\beta^*\|_1 }{\sup}  \sum_{i=1}^n   \frac{\xi_i\varepsilon_i  1_{  \{\vert  \varepsilon_i\vert \leq  U_n\}   } }{U_n}\delta_i    \right)\\
		& \leq& \displaystyle \frac{16U_n}{\eta^2} \,\mathbb{E}\left(     \underset{\delta \in \mathbb{R}^n \,:\, \|\delta\|\leq \eta , \, \| \nabla_G\delta\|_1 \leq 6 \|\nabla_G\beta^*\|_1 }{\sup}   \tilde{\varepsilon}^{\top }\nabla_G^{+} \nabla_G \delta   \right) \,+\,\\
		&&\displaystyle \frac{16U_n}{\eta^2} \,\mathbb{E}\left(     \underset{\delta \in \mathbb{R}^n \,:\, \|\delta\|\leq \eta , \, \| \nabla_G\delta\|_1 \leq 6 \|\nabla_G\beta^*\|_1 }{\sup}   \tilde{\varepsilon}^{\top }\Pi\delta   \right),\\
	\end{array}
	\]
	where   $\tilde{\varepsilon}_i  = \xi_i\varepsilon_i  1_{  \{\vert  \varepsilon_i\vert \leq  U_n\}   } /U_n$ for $i=1,\ldots,n$. Next, we observe that by  H\"{o}lder’s inequality and Cauchy–Schwarz inequality, it holds that
	,
	\[
	\begin{array}{lll}
		B_1  & \leq&\displaystyle  \frac{96U_n    \|  \nabla_G\beta^*\|_1 }{\eta^2} \mathbb{E}\left(  \|  \big(\nabla_G^{+}\big)^{\top} \tilde{\varepsilon}\|_{\infty}  \right)\,+\, \frac{16U_n}{\eta^2} \, \mathbb{E}\left\{     \underset{\delta \in  \mathbb{R}^n \,:\, \|\delta\|\leq \eta , \, \| \nabla_G\delta\|_1 \leq 6 \|\nabla_G\beta^*\|_1 }{\sup}   \tilde{\varepsilon}^{\top }\Pi\delta   \right\}\\
		& =&\displaystyle  \frac{96U_n    \|  \nabla_G\beta^*\|_1 }{\eta^2} \mathbb{E}\left(  \|  \big(\nabla_G^{+}\big)^{\top} \tilde{\varepsilon}\|_{\infty}  \right)\,+\, \frac{16 U_n  }{\eta^2}  \underset{\delta\,:\, \|\delta\|\leq \eta }{\sup} \bigg(  \frac{1}{n^{1/2}   }\sum_{i=1}^{n}\delta_i  \bigg)    \mathbb{E}\bigg(\bigg\vert\frac{1}{n^{1/2} }\sum_{i=1}^{n} \tilde{\varepsilon}_i \bigg\vert \bigg) \\
		&\leq&	 \displaystyle  \frac{96U_n    \|  \nabla_G\beta^*\|_1 }{\eta^2}\mathbb{E}\left(  \|  \big(\nabla_G^{+}\big)^{\top} \tilde{\varepsilon}\|_{\infty}  \right)\,+\, \frac{16 U_n }{\eta} \mathbb{E}\bigg(\bigg\vert\frac{1}{n^{1/2}  }\sum_{i=1}^{n} \tilde{\varepsilon}_i \bigg\vert \bigg) \\
		& \leq&   \displaystyle  \frac{96U_n    \|  \nabla_G\beta^*\|_1 }{\eta^2}\mathbb{E}\left(  \|  \big(\nabla_G^{+}\big)^{\top} \tilde{\varepsilon}\|_{\infty}  \right)\,+\, \frac{16 U_n }{\eta} \\
		& \leq& \displaystyle  \frac{96U_n    \|  \nabla_G\beta^*\|_1 }{\eta^2} \underset{j}{\max} \|(\nabla_G^{+})_{,j}\|   \,+\, \frac{16 U_n  }{\eta}\\
	\end{array}
	\]
	where the third and last inequalities follow from the  Sub-Gaussian maximal inequality. Next,  by Propositions 4 and 6 from  \cite{hutter2016optimal}, we obtain that 
	\begin{equation}
		\label{eqn:upper}
		\underset{j}{\max} \|(\nabla_G^{+})_{,j}\| \,\leq \, \phi_n \,:=\,\begin{cases}
			C (\log n)^{1/2} & \text{if} \,\,\,d=2,\\
			C,
		\end{cases}
	\end{equation}
	for some constant $C>0$. Therefore,
	\[
	B_1 \,\leq \,\frac{96U_n    \|  \nabla_G\beta^*\|_1 \phi_n }{\eta^2}  \,+\, \frac{16 U_n  }{\eta}.\\
	\]
	Hence, for a given $a \in (0,1)$, we let 
	\[
	\eta \,:=\,\frac{2}{a^{1/2}} (96U_n    \|  \nabla_G\beta^*\|_1 \phi_n )^{1/2} +  \frac{4}{a} \cdot16 U_n 
	\]
	and $\lambda $ as in (\ref{eqn:lambda}).
	
	Therefore, 
	\[
	B_1 \, \leq\,\frac{a^2}{4}+ \frac{a}{4} \,\leq \, \frac{a}{2}.
	\]
	Moreover, from (\ref{eqn:b2}), we have that 
	\[
	\begin{array}{lll}
		B_2  & =&\displaystyle\frac{16 n^{1/2} \,\underset{i=1,\ldots,n}{\max}   \{\mathbb{E}( \varepsilon_i^4  )\}^{1/4}      \{ \mathrm{pr}(\vert  \varepsilon_i\vert >  U_n) \}^{1/4}}{\eta}\\
		& \leq& \displaystyle \frac{a}{4}\frac{  n^{1/2} \,\underset{i=1,\ldots,n}{\max}   \{\mathbb{E}( \varepsilon_i^4  )\}^{1/4}      \{ \mathrm{pr}(\vert  \varepsilon_i\vert >  U_n) \}^{1/4}}{  U_n}\\
		& \leq& \displaystyle \frac{a}{4}.
	\end{array}
	\]
	Therefore,
	\[
	\mathrm{pr}(\|\hat{\beta} -\beta^*\|> \eta   )\,\leq\, \frac{3a}{4}.
	\]
	This proves (\ref{eqn:rate2}).

	\textbf{Proof of  (\ref{eqn:rate3}):}  First, by \cite{madrid2020adaptive},  there exists $N$ satisfying  $N \asymp (n/K)^{1/d}$   and functions $I \,:\,\mathbb{R}^n  \,\rightarrow  \mathbb{R}^n$, and $\tilde{I} \,:\ \mathbb{R}^n  \,\rightarrow \mathbb{R}^{N^d}$ satisfying the properties below.
	
	\begin{itemize}
		\item  \textbf{[Lemma 8 in \cite{padilla2018adaptive}].} Let $\mathcal{E}_1$ be the event  such that 
		\begin{equation}
			\label{eqn:lem8.1}
			\vert   e^{\top}\{\beta - I(\beta) \}\vert\,\leq \,  2\|e\|_{\infty} \|\nabla_G \beta\|_1, \,\,\,\,\,\forall \beta, e \in  \mathbb{R}^n,		
		\end{equation}
		and there exists a  $d$-dimensional lattice $G^{\prime}$ with $N^d$ nodes such that 
		\begin{equation}
			\label{eqn:lem8.2}
			\|\nabla_{G^{\prime}}  \tilde{I}(\beta) \|_1 \,\leq \, \|\nabla_G \beta\|_1, \,\,\,\,\,\forall \beta \in \mathbb{R}^n.		
		\end{equation}
		Then $\mathrm{pr}(\mathcal{E}_1) \rightarrow  1$.

		\item  \textbf{[Lemmas 7, 8 and 10 in \cite{padilla2018adaptive}].} Le $e $ be any  vector of mean zero independent subg-Gaussian($\sigma^2$), then there exists $\tilde{e}$ a  vector of mean zero independent sub-Gaussian($\sigma^2$)  random variables  and a constant $C>0$(not depending on $e$) such that the event $\mathcal{E}_2$ given as 
		\begin{equation}
			\label{eqn:lem10.1}
			\begin{array}{lll}
				\mathcal{E}_2&\,:=\,&\bigg\{	e^{\top} \{  I(\beta) -I(\beta^*)  \} \,\leq \,CK^{1/2}\bigg\{  \|\Pi \tilde{e}\|_2\|\beta- \beta^*\|  \,+\,\| (\nabla_{G^{\prime}})^{+}  \tilde{e}\|_{\infty} (\|\nabla_G \beta^*\|_1+\|\nabla_G\beta\|_1)  \bigg\} ,\\
				& &\,\,\,\,\,\,\,\,\forall \beta \in  \mathbb{R}^n,\,\,\,\,\forall e \,\,    \text{vector of mean zero independent sub-Gaussian}(\sigma^2)
				\bigg\}
			\end{array}
		\end{equation}
		satisfies $\mathrm{pr}(\mathcal{E}_2) \rightarrow 1$.
		
		\item  \textbf{Theorem 2 in \cite{padilla2018adaptive}].}  It holds that for some constant $C_2>0$ the event
		\[
		\mathcal{E}_3 \,:=\, \left\{\|\nabla_G\beta^*\|_1 \,\leq \, C_2   \text{poly}(\log n) n^{1-1/d}\right\}
		\]
		satisfies  $\mathrm{pr}(\mathcal{E}_3) \rightarrow 1$, where $\text{poly}(\cdot)$ is a polynomial function.
		
		
	\end{itemize}

	Let  $\mathcal{E}_4 = \mathcal{E}_1 \cap \mathcal{E}_2 \cap \mathcal{E}_3$
	Notice that as in Theorem \ref{thm3}, with the choice
	\[
	\lambda \,=\,   \frac{\eta^2}{4\| \nabla_G \beta^*\|_1 },
	\]
	we have that 
	\begin{equation}
		\label{eqn:bound}
		\begin{array}{lll}
			\mathrm{pr}( \| \hat{\beta} -\beta^*\|  >\eta | \mathcal{E}_4)    & \leq&\displaystyle\frac{16 n^{1/2}   \,\underset{i=1,\ldots,n}{\max}    \{\mathbb{E}( \varepsilon_i^4  )\}^{1/4}      \{ \mathrm{pr}(\vert  \varepsilon_i\vert > U_n) \}^{1/4}}{\eta}\,+\,\\
			& &\displaystyle \frac{16U_n}{\eta^2} \,\mathbb{E}\left[     \underset{\beta \in \mathbb{R}^n   \,:\, \|\beta - \beta^*\|\leq \eta , \, \| \nabla_G\beta\|_1 \leq 5 \|\nabla_G\beta^*\|_1 }{\sup}  \sum_{i=1}^n   \frac{\xi_i\varepsilon_i  1_{  \{\vert  \varepsilon_i\vert \leq U_n\}   } }{U_n}(\beta_i - \beta_i^*)   \,\bigg| \mathcal{E}_4\right]\\
			& =:&  T_1 +T_2.
		\end{array}
	\end{equation}
	
	Next we bound $T_1$ and $T_2$. To bound $T_2$,  we define 
	$$e_i \,:=\,    \frac{\xi_i\varepsilon_i  1_{  \{\vert  \varepsilon_i\vert \leq U_n\}   } }{U_n},$$ 
	and notice that $\mathbb{E}(e_{i} |\Omega_4)  =  \mathbb{E}(e_{i} ) =0 $,  and $e_i $ is sub-Gaussian($1$) for $i=1,\ldots,n$.  It follows that if $\mathcal{E}_4$ holds, then
	\[
	\begin{array}{lll}
		e^{\top}(\beta-\beta^*)  & = & \displaystyle e^{\top}\{\beta-  I(\beta)\}   \,+\,e^{\top}\{I(\beta)-   I(\beta^*) \}   \,-\,   e^{\top}\{\beta^*-  I(\beta^*)\}   \\ 
		& \leq &\displaystyle  2\|\nabla_G \beta^*\|_1  \,+\, CK^{1/2}\bigg[  \|\Pi \tilde{e}\|_2\|\beta- \beta^*\|  \,+\,\| (\nabla_{G^{\prime}})^{+}  \tilde{e}\|_{\infty} (\|\nabla_G \beta^*\|_1+\|\nabla_G\beta\|_1)  \bigg] \,+\,\\ 
		& &\displaystyle  2\|\nabla_G \beta\|_1,
	\end{array}
	\]
	where 
	Therefore,
	\[
	\begin{array}{lll}
		T_2  & \leq&\displaystyle \frac{16U_n}{\eta^2}\bigg\{ 12C_2 \text{poly}(\log n)  n^{1-1/d}  \,+\, CK^{1/2}\bigg[    \eta  \mathbb{E}( \| \Pi \tilde{e}\|  )    \,+\, 6  C_2 \text{poly}(\log n)n^{1-1/d}   E(\| (\nabla_{G^{\prime}})^{+}  \tilde{e}\|_{\infty} )  \bigg] \bigg\}\\
		& \leq&  \displaystyle \frac{16U_n}{\eta^2}\bigg(  12C_2 \text{poly}(\log n)n^{1-1/d}  \,+\, C K^{1/2}\bigg[    \eta \mathbb{E} \bigg\vert \frac{1}{n^{1/2} }\sum_{i=1}^{n} \tilde{e}_i \bigg\vert    \,+\, 6  C_2 \text{poly}(\log n)n^{1-1/d} \underset{j}{\max} \|(\nabla_G^{+})_{,j}\|   \bigg] \bigg)\\
		& \leq&\displaystyle \frac{16U_n}{\eta^2}\bigg[  12C_2 \text{poly}(\log n)n^{1-1/d}  \,+\, CK^{1/2}\bigg[    \eta  \,+\, 6  C_2 \text{poly}(\log n)n^{1-1/d}\underset{j}{\max} \|(\nabla_G^{+})_{,j}\|   \bigg] \bigg]\\
		& \leq&\displaystyle \frac{16U_n}{\eta^2}\bigg[  12C_2\text{poly}(\log n) n^{1-1/d}  \,+\, CK^{1/2}\bigg[    \eta  \,+\, 6  C_2\text{poly}(\log n)n^{1-1/d}\phi_n  \bigg] \bigg]\\
	\end{array}
	\]
	where the second and third inequalities follow from Sub-Gaussian maximal inequality, and the last from (\ref{eqn:upper}).  Then for a given $a\in(0,1)$, we set 
	\[
	\begin{array}{lll}
		\eta &= & \frac{6^{1/2}}{a^{1/2}  }   (16 \times 12 C_2  \text{poly}(\log n)n^{1-1/d}  U_n)^{1/2} \,+\, \frac{6}{a}(16 C K^{1/2}  U_n)\,+\,\\
		& &\frac{6}{a} (16 \times 6 C C_2 U_n  K^{1/2}  \text{poly}(\log n)\,n^{1-1/d}\phi_n)^{1/2}
	\end{array}
	\]
	and so
	\begin{equation}
		\label{eqn:e16}
		T_2 \,\leq \,  \frac{a}{2}.
	\end{equation}
	Furthermore,
	\begin{equation}
		\label{eqn:e17}
		\begin{array}{lll}
			T_1  & \leq&  \displaystyle  \frac{a}{6} \frac{16 n^{1/2}    \,\underset{i=1,\ldots,n}{\max}   \{\mathbb{E}( \varepsilon_i^4  )\}^{1/4}      \{  \mathrm{pr}(\vert  \varepsilon_i\vert > U_n) \}^{1/4}}{16 C  K^{1/2}  U_n}\\
			& \leq& \displaystyle \frac{a}{6}.
		\end{array}
	\end{equation}
	The claim then follows. 
\end{proof}

\section{Auxiliary lemmas for proof of  Theorem \ref{thm5} }

\begin{lemma}
	\label{lem1}
	Let $\gamma_i^* =  \mathbb{E}(y_i^2)$ for  $i = 1,\ldots,n$. Then
	\[
	\sum_{ (i,j)\in E } \vert  \gamma_i^*  - \gamma_j^* \vert  \,\leq \, 	\sum_{ (i,j)\in E } \vert   v_i^* \,- \, v_j^*\vert  \,+\, 2\|\theta^*\|_{\infty}\sum_{ (i,j)\in E } \vert   \theta_i^* \,- \, \theta_j^*\vert.
	\]
\end{lemma}

\begin{proof}
	Notice that \[
	\begin{array}{lll}
		\displaystyle 	\sum_{ (i,j)\in E } \vert  \gamma_i^*  - \gamma_j^* \vert  & =& 	\displaystyle 	\sum_{ (i,j)\in E } \vert   \{v_i^* +(\theta_i^*)^2 \}   \,-\, \{v_j^* +(\theta_j^*)^2 \} \vert  \\
		& \leq & 	\displaystyle 	\sum_{ (i,j)\in E } \vert   v_i^* \,- \, v_j^*\vert  \,+\,\sum_{ (i,j)\in E } \vert   \theta_i^* \,- \, \theta_j^*\vert  (  \vert \theta_i^*\vert  +  \vert \theta_j^*\vert   )\\
		& \leq&\displaystyle 	\sum_{ (i,j)\in E } \vert   v_i^* \,- \, v_j^*\vert  \,+\, 2\|\theta^*\|_{\infty}\sum_{ (i,j)\in E } \vert   \theta_i^* \,- \, \theta_j^*\vert  \\
	\end{array}
	\]
	and the claim follows.
\end{proof}

\begin{lemma}
	\label{lem2}
	For any $U_n>0$ we have that
	\[
	\mathrm{pr}\left(\,\vert y_i^2 - E(y_i^2)\vert  \,>\,2 \|v^*\|_{\infty}^{1/2}\|\theta^*\|_{\infty} U_n   \,+\,  \| v^*\|_{\infty}(1+U_n^2)  \right)\,\leq \, \mathrm{pr}(\vert \epsilon_i\vert >U_n).
	\]
	for $i=1,\ldots,n$.
	
\end{lemma}

\begin{proof}
	Simply observe that
	\[
	\begin{array}{lll}
		\displaystyle 	y_i^2 -   \mathbb{E}(y_i^2)  & = &\displaystyle \{ \theta_i^*   +  (v_i ^*)^{1/2}\epsilon_i  \}^2 -   \mathbb{E}[\{\theta_i^*   + ( v_i ^*)^{1/2}\epsilon_i\}^2]  \\
		&  = &   \displaystyle 2 (v_i^*)^{1/2}\theta_i^* \epsilon_i   +  v_i ^*\epsilon_i ^2  - v_i^* 
	\end{array}
	\]
	and hence
	\[
	\begin{array}{lll}
		\displaystyle \vert   	y_i^2 -    \mathbb{E}(y_i^2)  \vert   & \leq& 2\|  v^*\|_{\infty}^{1/2}\|\theta^*\|_{\infty}    \vert \epsilon_i\vert+\, \|  v^*\|_{\infty}\vert \epsilon_i\vert^2 \,+\, \| v^*\|_{\infty}
	\end{array}
	\]
	and so the claim follows.

\end{proof}

\section{Proof of Theorem \ref{thm5}}

\begin{proof}
	First notice that
	\[
	\begin{array}{lll}
		\displaystyle 	\frac{1}{n}\|\hat{v} -v^*\|^2   & =&   \displaystyle \frac{1}{n}\sum_{i=1}^{n} \big[\{\hat{\gamma}_i - (\hat{\theta}_i)^2\}- \{\gamma^*_i - (\theta^*_i)^2\}\big]^2\\
		& \leq& \displaystyle  \frac{2}{n}\sum_{i=1}^{n}\big(\hat{\gamma}_i - \gamma^*_i \big)^2   \,+\, \frac{2}{n}\sum_{i=1}^{n}   \big\{   (\hat{\theta}_i)^2-(\theta^*_i)^2\big\}^2 \\
		& \leq& \displaystyle \frac{2}{n}\sum_{i=1}^{n}\big(\hat{\gamma}_i - \gamma^*_i \big)^2   \,+\, \frac{8\| \theta^*\|_{\infty}^2}{n}\sum_{i=1}^{n}   \big(   \hat{\theta}_i-\theta^*_i\big)^2 \\
	\end{array}
	\]
	and so each conclusion of the theorem follows applying  Theorem \ref{thm4}, Lemma \ref{lem1}, and Lemma \ref{lem2}. Specifically, it is clear that the generative model and $\theta^*$ satisfy the  conditions of Theorem \ref{thm4}. As for the estimation of $\gamma^*$,  letting  $r_i   =   y_i^2 -  \mathbb{E}(y_i^2)$, for $i=1,\ldots,n$, we need to verify (\ref{eqn:c1})  for  $\{r_i\}_{i=1}^n$. 
	
	\textbf{Proof of  (\ref{eqn:rate4}).} Notice that by Lemma \ref{lem2}, 
	\[
	\begin{array}{lll}
		\displaystyle 	\frac{n^{1/4}  \,\underset{i=1,\ldots,n}{\max}   \{ \mathrm{pr}(\vert  r_i\vert >  U_n^{\prime}) \}^{1/4}}{U_n^{\prime} \{\log (en )\}^{1/2}\ } & \leq & 	\displaystyle 	\frac{n^{1/4}  \,\underset{i=1,\ldots,n}{\max}   \{ \mathrm{pr}(\vert  r_i\vert >  2\|  v^*\|_{\infty}^{1/2}\|\theta^*\|_{\infty}U_n+\, \|  v^*\|_{\infty}U_n^2 \,+\, \| v^*\|_{\infty}   ) \}^{1/4}}{U_n\{\log (en )\}^{1/2}\ } \\
		& \leq&\displaystyle 	\frac{n^{1/4}  \,\underset{i=1,\ldots,n}{\max}   \{ \mathrm{pr}(\vert  \epsilon_i\vert >  U_n)\}^{1/4}}{U_n\{\log (en )\}^{1/2}\ } \,\rightarrow\,0.\\
	\end{array}
	\]
	Therefore, by Lemma \ref{lem1} and Theorem \ref{thm4},
	\[
	\frac{1}{n}\sum_{i=1}^{n}\big(\hat{\theta}_i - \theta^*_i \big)^2  \,=\, 	 O_{\mathrm{pr}}\left\{\frac{U_n^{4/3}  (\log n)^{1/3} \|\nabla_G\theta^*\|_1^{2/3}    }{n^{2/3}}\,+\, \frac{U_n^2 \log n}{n} \right\},
	\]
	and 
	\[
	\frac{1}{n}\sum_{i=1}^{n}\big(\hat{\gamma}_i - \gamma^*_i \big)^2  \,=\, 	 O_{\mathrm{pr}}\left\{\frac{(U_n^{\prime})^{4/3}  (\log n)^{1/3} \left(\|\nabla_Gv^*\|_1+ \|\theta^*\|_{\infty}\|\nabla_G\theta^*\|_1\right)^{2/3}    }{n^{2/3}}\,+\, \frac{(U_n^{\prime})^2 \log n}{n} \right\},
	\]
	and so the claim (\ref{eqn:rate4}) follows.

	The proof of (\ref{eqn:rate5}) and (\ref{eqn:rate6}) follow similarly. 
\end{proof}

\section{Lower bounds}

\subsection{Proof of Lemma \ref{lem:lower1}}

\begin{proof}

	We notice that
	\begin{equation}
		\label{eqn:lower_bound}
		\begin{array}{l}
			\underset{ \tilde{v} \in  \mathcal{F} }{\inf}\,\,\,\underset{ \theta^*,v^* \in K,   \,\,\, v_i^* \in (\frac{c^2}{8} ,\frac{3c^2}{8}), \,\,\,\, y_i =  \theta^*_i   +  \sqrt{v^*_i} \epsilon_i , \,\,\epsilon_i    \overset{\text{ind} } {\sim}\text{N}(0,1)   }{\sup} \,\mathbb{E}\left(   \frac{1}{n}\|   \tilde{v}(y) - v^*  \|^2  \right)    \\
			\geq 	\underset{ \tilde{v} \in  \mathcal{F} }{\inf}\,\,\,\underset{ \theta^* \in K,   \,\,\, \theta_i^* \in ( \frac{c}{\sqrt{8} },   \frac{\sqrt{3}c}{\sqrt{8} } ), \,\,\,\, y_i =  \theta^*_i   +  \sqrt{\frac{c^2}{2}   -   (\theta^*_i)^2 } \epsilon_i , \,\,\epsilon_i    \overset{\text{ind} } {\sim}\text{N}(0,1)   }{\sup} \,\mathbb{E}\left(   \frac{1}{n}  \sum_{i=1}^n   (   \tilde{v}_i(y)   -    (c^2/2-  (\theta^*_i)^2)   )^2 \right)    \\
			= 	\underset{ \tilde{v} \in  \mathcal{F} }{\inf}\,\,\,\underset{ \theta^*\in K,   \,\,\, \theta_i^* \in ( \frac{c}{\sqrt{8} },   \frac{\sqrt{3}c}{\sqrt{8} } ), \,\,\,\, y_i =  \theta^*_i   +  \sqrt{\frac{c^2}{2}   -   (\theta^*_i)^2 } \epsilon_i , \,\,\epsilon_i    \overset{\text{ind} } {\sim}\text{N}(0,1)   }{\sup} \,\mathbb{E}\left(   \frac{1}{n}  \sum_{i=1}^n   (   \tilde{v}_i(y)   -      (\theta^*_i)^2   )^2 \right)    \\
			\geq \underset{ \tilde{v} \in  \mathcal{F},  \, \tilde{v}_i(\cdot)  \in [\frac{c}{\sqrt{8} },   \frac{\sqrt{3}c}{\sqrt{8} } ]  }{\inf}\,\,\,\underset{ \theta^*\in K,   \,\,\, \theta_i^* \in (\frac{c}{\sqrt{8} },   \frac{\sqrt{3}c}{\sqrt{8} } ), \,\,\,\, y_i =  \theta^*_i   +  \sqrt{\frac{c^2}{2}   -   (\theta^*_i)^2 } \epsilon_i , \,\,\epsilon_i    \overset{\text{ind} } {\sim}\text{N}(0,1)   }{\sup} \,\mathbb{E}\left(   \frac{1}{n}  \sum_{i=1}^n   (   \tilde{v}_i(y)^2   -      (\theta^*_i)^2   )^2 \right)    \\
			\geq \underset{ \tilde{v} \in  \mathcal{F},  \, \tilde{v}_i(\cdot)  \in [\frac{c}{\sqrt{8} },   \frac{\sqrt{3}c}{\sqrt{8} } ] }{\inf}\,\,\,\underset{ \theta^*\in K,   \,\,\, \theta_i^* \in (\frac{c}{\sqrt{8} },   \frac{\sqrt{3}c}{\sqrt{8} } ), \,\,\,\, y_i =  \theta^*_i   +  \sqrt{\frac{c^2}{2}   -   (\theta^*_i)^2 } \epsilon_i , \,\,\epsilon_i    \overset{\text{ind} } {\sim}\text{N}(0,1)   }{\sup} \,\mathbb{E}\bigg(   \frac{1}{n}  \sum_{i=1}^n   (   \tilde{v}_i(y) - \theta^*_i   )^2   \cdot \\
			\,\,\,\,\,\,\,\,\,\,\,\,\,\,\,\,\,\,\,\,\,\,\,\,\,\,\,\,\,\,\,\,\,\,\,\,\,\,\,\,\,\,\,\,\,\,\,\,\,\,\,\,\,  \,\,\,\,\,\,\,\,\,\,\,\,\,\,\,\,\,\,\,\,\,\,\,\,\,\,\,\,\,\,\,\,\,\,\,\,\,\,\,\,\,\,\,\,\,\,\,\,\,\,\,\,\, \,\,\,\,\,\,\,\,\,\,\,\,\,\,\,\,\,\,\,\,\,\,\,\,\,\,\,\,\,\,\,\,\,\,\,\,\,\,\,\,\,\,\,\,\,\,\,\,\,\,\,\,\, \,\,\,\,\, \,\,\,\,\underset{j=1,\ldots\,n}{\min}(     \tilde{v}_j(y) +\theta^*_j  )^2   \bigg)    \\
			\geq  \frac{c^2}{32} \,\underset{ \tilde{v} \in  \mathcal{F},  \, \tilde{v}_i(\cdot)  \in [\frac{c}{\sqrt{8} },   \frac{\sqrt{3}c}{\sqrt{8} } ]  }{\inf}\,\,\,\underset{ \theta^*\in K,   \,\,\, \theta_i^* \in (\frac{c}{\sqrt{8} },   \frac{\sqrt{3}c}{\sqrt{8} } ), \,\,\,\, y_i =  \theta^*_i   +  \sqrt{\frac{c^2}{2}    -   (\theta^*_i)^2 } \epsilon_i , \,\,\epsilon_i    \overset{\text{ind} } {\sim}\text{N}(0,1)   }{\sup} \,\mathbb{E}\left(   \frac{1}{n}  \sum_{i=1}^n   (   \tilde{v}_i(y)   -      \theta^*_i)  )^2 \right)    \\
			\geq \frac{c^2}{32} \,\underset{ \tilde{v} \in  \mathcal{F}  \,  }{\inf}\,\,\,\underset{ \theta^*\in K,   \,\,\, \theta_i^* \in (\frac{c}{\sqrt{8} },   \frac{\sqrt{3}c}{\sqrt{8} }), \,\,\,\, y_i =  \theta^*_i   +  \sqrt{\frac{c^2}{2}   -   (\theta^*_i)^2 } \epsilon_i , \,\,\epsilon_i    \overset{\text{ind} } {\sim}\text{N}(0,1)   }{\sup} \,\mathbb{E}\left(   \frac{1}{n}  \sum_{i=1}^n   (   \tilde{v}_i(y)   -      \theta^*_i)  )^2 \right)    \\
			\geq \frac{c^2}{32} \,\underset{ \tilde{v} \in  \mathcal{F}  \,  }{\inf}\,\,\,\underset{ \theta^*\in K,   \,\,\, \theta_i^* \in (\frac{c}{\sqrt{8} },   \frac{\sqrt{3}c}{\sqrt{8} } ), \,\,\,\, y_i =  \theta^*_i   +  \epsilon_i , \,\,\epsilon_i    \overset{\text{ind} } {\sim}\text{N}(0,\frac{c^2}{8})   }{\sup} \,\mathbb{E}\left(   \frac{1}{n}  \sum_{i=1}^n   (   \tilde{v}_i(y)   -      \theta^*_i)  )^2 \right).    \\
		\end{array}
	\end{equation}

	Next let $d_{\max}$ be the maximum degree of any node in $G$ and  consider  distinct $a_1, \ldots, a_{m} \in [n^{1/d} ] \times \ldots \times   [ n^{1/d}]$, for   $m \in \mathbb{N}$ with $m \asymp n^{1-1/d}$ and $d_{\max}\cdot m \leq  n^{1-1/d}$, such that for all $j,j^{\prime} \in \{1,\ldots,m\} $ it holds that  $a_j$ and $a_{j^{\prime}}$ are not connected by an edge in the $d$-dimensional grid graph associated with $[n^{1/d} ] \times \ldots \times   [ n^{1/d}]$ . Then for  $\eta \in \{-1,1\}^{ m}$ let $\theta_{\eta} \in \mathbb{R}^n$ be given as
	\[
	(\theta_{\eta})_{i}  \,=\,  \begin{cases}
		\frac{c}{ \sqrt{8}}\left[\frac{  \eta_{a_j }  (\sqrt{3}-1) }{4}   +   \frac{   (1+\sqrt{3})}{2} \right] & \text{if }\,\,\, i =     a_{_j} ,\,\,\,\, j  \in \{1,\ldots,m \}  \\
		\frac{c}{ \sqrt{8}} \cdot\frac{    (1+\sqrt{3})}{2} & \text{otherwise.}
	\end{cases}
	\]
	Notice that by construction $(\theta_{\eta})_i \in (\frac{c}{\sqrt{8} },   \frac{\sqrt{3}c}{\sqrt{8} })$ for all $i$ and $\eta \in \{-1,1\}^m$.  Moreover,
	\[
	\| \nabla_G \theta_{\eta} \|_1 \,\leq \,        \frac{ d_{\max}\cdot c  }{ \sqrt{8}} \cdot\frac{m  (\sqrt{3}-1) }{4} \,\leq \, c n^{1-1/d}.
	\]
	In addition, if  $\eta, \eta^{\prime}   \in \{-1,1\}^m$ such that $\|\eta - \eta^{\prime}\|_1= 2$, then 
	\[
	\|   \theta_{\eta} - \theta_{\eta^{\prime} } \| \,=\,     \frac{c}{ \sqrt{8}} \cdot\frac{ (\sqrt{3}-1)}{2}.
	\]
	Also, denoting by $P_{\eta}$ and $P_{\eta^{\prime}}$ the distributions $N( \theta_{\eta},  \frac{c^2}{8}I_n  )$ and $N( \theta_{\eta^{\prime}},  \frac{c^2}{8} I_n  )$, respectively, we obtain that 
	\[
	\mathrm{TV}(P_{\eta}, P_{\eta^{\prime}}) \,\leq\, \sqrt{  \frac{1}{2}  D_{\text{KL} }(P_{\eta},P_{\eta^{\prime}})  }  \,\leq\,  \frac{\sqrt{8}}{2 c } \| \theta_{\eta} - \theta_{\eta^{\prime}}\| \,=\, \frac{(\sqrt{3}-1) }{4}
	\]
	where the first  inequality follows from Pinsker's inequality.  Therefore, by Assouad's lemma, Lemma \ref{lem2} in  \cite{yu1997assouad}, we obtain that

	\begin{equation}
		\label{eqn:lower_bound2}
		\begin{array}{lll}
			\underset{ \tilde{v} \in  \mathcal{F}  \,  }{\inf}\,\,\,\underset{ \theta^*\in K,   \,\,\, \theta_i^* \in (\frac{c}{\sqrt{8} },   \frac{\sqrt{3}c}{\sqrt{8} } ), \,\,\,\, y_i =  \theta^*_i   +  \epsilon_i , \,\,\epsilon_i    \overset{\text{ind} } {\sim}\text{N}(0,\frac{c^2}{8})   }{\sup} \,\mathbb{E}\left(     \sum_{i=1}^n   (   \tilde{v}_i(y)   -      \theta^*_i)^2 \right)  \\
			\gtrsim \frac{m }{2}  \cdot \left(  1-   \frac{(\sqrt{3}-1) }{4}\right)\\
			\gtrsim   n^{1-1/d}.
		\end{array}
	\end{equation}
	Hence, from (\ref{eqn:lower_bound}) and (\ref{eqn:lower_bound2}) imply 
	\[
	\underset{ \tilde{v} \in  \mathcal{F} }{\inf}\,\,\,\underset{ \theta^*,v^* \in K,   \,\,\, v_i^* \in (\frac{c^2}{8} ,\frac{3c^2}{8}), \,\,\,\, y_i =  \theta^*_i   +  \sqrt{v^*_i} \epsilon_i , \,\,\epsilon_i    \overset{\text{ind} } {\sim}\text{N}(0,1)   }{\sup} \,\mathbb{E}\left(   \frac{1}{n}\|   \tilde{v}(y) - v^*  \|^2  \right)    \,\gtrsim   \frac{1}{ n^{1/d}   }.
	\]
\end{proof}

\subsection{Proof of Lemma \ref{lem:lower2}}

\begin{proof}
	
	We observe that 
	\begin{equation}
		\label{eqn:lower_bound3}
		\begin{array}{l}
			\underset{ \tilde{g}  }{\inf}\,\,\,\underset{ f_0,g_0 \in \mathcal{F}(L_0),   \,\,\,  \frac{c^2}{8} \leq g_0 \leq\frac{3c^2}{8}, \,\,\,\, y_i =  f_0(x_i)  +  \sqrt{g_0(x_i)} \epsilon_i , \,\,\epsilon_i    \overset{\text{ind} } {\sim}\text{N}(0,1)   }{\sup} \,\mathbb{E}\left(   \|   \tilde{g} - g_0  \|_2^2  \right)    \\
			\geq 	\underset{ \tilde{g}  }{\inf}\,\,\,\underset{ f_0 \in \mathcal{F}(L_0),   \,\,\, \frac{c}{ \sqrt{8}} \leq f_0 \leq\frac{  c \sqrt{3} }{8}, \,\,\,\, y_i =  f_0(x_i)   +  \sqrt{\frac{c^2}{2}   -   (f_0(x_i))^2 } \epsilon_i , \,\,\epsilon_i    \overset{\text{ind} } {\sim}\text{N}(0,1)   }{\sup} \,\mathbb{E}\left(    \|  \tilde{g}   -    (c^2/2-  (f_0)^2)   \|_2^2 \right)    \\
			= 	\underset{ \tilde{g} }{\inf}\,\,\,\underset{ f_0 \in \mathcal{F}(L_0),   \,\,\, \frac{c}{ \sqrt{8}} \leq f_0 \leq\frac{  c \sqrt{3} }{8}, \,\,\,\, y_i =  f_0(x_i)   +  \sqrt{\frac{c^2}{2}   -   (f_0(x_i))^2 } \epsilon_i , \,\,\epsilon_i    \overset{\text{ind} } {\sim}\text{N}(0,1)   }{\sup} \,\mathbb{E}\left(  \|  \tilde{g}  -      (f_0)^2   \|_2^2 \right)    \\
			\geq \underset{  \frac{c}{ \sqrt{8}} \leq \tilde{g} \leq\frac{  c \sqrt{3} }{8} }{\inf}\,\,\,\underset{ f_0 \in \mathcal{F}(L_0),   \,\,\, \frac{c}{ \sqrt{8}} \leq f_0 \leq\frac{  c \sqrt{3} }{8}, \,\,\,\, y_i =  f_0(x_i)   +  \sqrt{\frac{c^2}{2}   -   (f_0(x_i))^2 } \epsilon_i , \,\,\epsilon_i    \overset{\text{ind} } {\sim}\text{N}(0,1)   }{\sup} \,\mathbb{E}\left(    \|  \tilde{g}^2  -      (f_0)^2   \|_2^2 \right)    \\
			\geq \underset{  \, \frac{c}{ \sqrt{8}} \leq \tilde{g} \leq\frac{  c \sqrt{3} }{8} }{\inf}\,\,\,\underset{ f_0 \in \mathcal{F}(L_0),   \,\,\, \frac{c}{ \sqrt{8}} \leq f_0 \leq\frac{  c \sqrt{3} }{8}, \,\,\,\, y_i =  f_0(x_i)   +  \sqrt{\frac{c^2}{2}   -   (f_0(x_i))^2 } \epsilon_i , \,\,\epsilon_i    \overset{\text{ind} } {\sim}\text{N}(0,1)   }{\sup} \,\mathbb{E}\bigg(    \|  \tilde{g}  -      f_0  \|_2^2  \cdot \\
			\,\,\,\,\,\,\,\,\,\,\,\,\,\,\,\,\,\,\,\,\,\,\,\,\,\,\,\,\,\,\,\,\,\,\,\,\,\,\,\,\,\,\,\,\,\,\,\,\,\,\,\,\,  \,\,\,\,\,\,\,\,\,\,\,\,\,\,\,\,\,\,\,\,\,\,\,\,\,\,\,\,\,\,\,\,\,\,\,\,\,\,\,\,\,\,\,\,\,\,\,\,\,\,\,\,\, \,\,\,\,\,\,\,\,\,\,\,\,\,\,\,\,\,\,\,\,\,\,\,\,\,\,\,\,\,\,\,\,\,\,\,\,\,\,\,\,\,\,\,\,\,\,\,\,\,\,\,\,\, \,\,\,\,\, \,\,\,\,\underset{x \in [0,1]^d}{\inf}(     \tilde{g}(x) +f_0(x) )^2   \bigg)    \\
			\geq \frac{c^2}{32}\,\underset{ \, \frac{c}{ \sqrt{8}} \leq \tilde{g} \leq\frac{  c \sqrt{3} }{8}   }{\inf}\,\,\,\underset{f_0 \in \mathcal{F}(L_0),   \,\,\, \frac{c}{ \sqrt{8}} \leq f_0 \leq\frac{  c \sqrt{3} }{8}, \,\,\,\, y_i =  f_0(x_i)   +  \sqrt{\frac{c^2}{2}   -   (f_0(x_i))^2 } \epsilon_i , \,\,\epsilon_i    \overset{\text{ind} } {\sim}\text{N}(0,1)     }{\sup} \,\mathbb{E}\left(   \|   \tilde{g}   -     f_0  \|_2^2 \right)    \\
			\geq \frac{c^2}{32}\,\underset{  \tilde{g}  }{\inf}\,\,\,\underset{f_0 \in \mathcal{F}(L_0),   \,\,\, \frac{c}{ \sqrt{8}} \leq f_0 \leq\frac{  c \sqrt{3} }{8}, \,\,\,\, y_i =  f_0(x_i)   +  \sqrt{\frac{c^2}{2}   -   (f_0(x_i))^2 } \epsilon_i , \,\,\epsilon_i    \overset{\text{ind} } {\sim}\text{N}(0,1)     }{\sup} \,\mathbb{E}\left(   \|   \tilde{g}   -     f_0  \|_2^2 \right)    \\
			\geq \frac{c^2}{32}\,\underset{  \tilde{g}  }{\inf}\,\,\,\underset{f_0 \in \mathcal{F}(L_0),   \,\,\, \frac{c}{ \sqrt{8}} \leq f_0 \leq\frac{  c \sqrt{3} }{8}, \,\,\,\, y_i =  f_0(x_i)   + \epsilon_i , \,\,\epsilon_i    \overset{\text{ind} } {\sim}\text{N}(0, \frac{c^2}{8} )     }{\sup} \,\mathbb{E}\left(   \|   \tilde{g}   -     f_0  \|_2^2 \right)    \\
			\gtrsim \frac{1}{n^{1/d}}
		\end{array}
	\end{equation}
	where the last inequality follows from Proposition  2 in \cite{castro2005faster}.
\end{proof}

\bibliographystyle{plainnat}
\bibliography{references}	

\begin{thebibliography}{64}
\providecommand{\natexlab}[1]{#1}
\providecommand{\url}[1]{\texttt{#1}}
\expandafter\ifx\csname urlstyle\endcsname\relax
  \providecommand{\doi}[1]{doi: #1}\else
  \providecommand{\doi}{doi: \begingroup \urlstyle{rm}\Url}\fi

\bibitem[Barbero and Sra(2014)]{barbero2014modular}
{\'A}lvaro Barbero and Suvrit Sra.
\newblock Modular proximal optimization for multidimensional total-variation
  regularization.
\newblock \emph{arXiv preprint arXiv:1411.0589}, 2014.

\bibitem[Blanchard et~al.(2007)Blanchard, Sch{\"a}fer, Rozenholc, and
  M{\"u}ller]{blanchard2007optimal}
Gilles Blanchard, Christin Sch{\"a}fer, Yves Rozenholc, and K-R M{\"u}ller.
\newblock Optimal dyadic decision trees.
\newblock \emph{Machine Learning}, 66\penalty0 (2):\penalty0 209--241, 2007.

\bibitem[Cai and Wang(2008)]{cai2008adaptive}
T~Tony Cai and Lie Wang.
\newblock Adaptive variance function estimation in heteroscedastic
  nonparametric regression.
\newblock \emph{The Annals of Statistics}, 36\penalty0 (5):\penalty0
  2025--2054, 2008.

\bibitem[Cai et~al.(2009)Cai, Levine, and Wang]{cai2009variance}
T~Tony Cai, Michael Levine, and Lie Wang.
\newblock Variance function estimation in multivariate nonparametric regression
  with fixed design.
\newblock \emph{Journal of Multivariate Analysis}, 100\penalty0 (1):\penalty0
  126--136, 2009.

\bibitem[Cappello et~al.(2021)Cappello, Padilla, and
  Palacios]{cappello2021scalable}
Lorenzo Cappello, Oscar Hernan~Madrid Padilla, and Julia~A Palacios.
\newblock Scalable bayesian change point detection with spike and slab priors.
\newblock \emph{arXiv preprint arXiv:2106.10383}, 2021.

\bibitem[Castro et~al.(2005)Castro, Willett, and Nowak]{castro2005faster}
Rui~M Castro, Rebecca Willett, and Robert Nowak.
\newblock Faster rates in regression via active learning.
\newblock \emph{Tech. Rep., University of Wisconsin, Madison, June 2005,
  ECE-05-3 Technical Report (available at http://homepages.cae.wisc.edu/
  rcastro/ECE-05-3.pdf)}, 2005.

\bibitem[Chambolle and Darbon(2009)]{chambolle2009total}
Antonin Chambolle and J{\'e}r{\^o}me Darbon.
\newblock On total variation minimization and surface evolution using
  parametric maximum flows.
\newblock \emph{International Journal of Computer Vision}, 84\penalty0
  (3):\penalty0 288--307, 2009.

\bibitem[Chatterjee and Goswami(2021{\natexlab{a}})]{chatterjee2021adaptive}
Sabyasachi Chatterjee and Subhajit Goswami.
\newblock Adaptive estimation of multivariate piecewise polynomials and bounded
  variation functions by optimal decision trees.
\newblock \emph{The Annals of Statistics}, 49\penalty0 (5):\penalty0
  2531--2551, 2021{\natexlab{a}}.

\bibitem[Chatterjee and Goswami(2021{\natexlab{b}})]{chatterjee2021new}
Sabyasachi Chatterjee and Subhajit Goswami.
\newblock New risk bounds for 2d total variation denoising.
\newblock \emph{IEEE Transactions on Information Theory}, 67\penalty0
  (6):\penalty0 4060--4091, 2021{\natexlab{b}}.

\bibitem[Coifman and Maggioni(2006)]{coifman2006diffusion}
Ronald~R Coifman and Mauro Maggioni.
\newblock Diffusion wavelets.
\newblock \emph{Applied and computational harmonic analysis}, 21\penalty0
  (1):\penalty0 53--94, 2006.

\bibitem[Crovella and Kolaczyk(2003)]{crovella2003graph}
Mark Crovella and Eric Kolaczyk.
\newblock Graph wavelets for spatial traffic analysis.
\newblock In \emph{IEEE INFOCOM 2003. Twenty-second Annual Joint Conference of
  the IEEE Computer and Communications Societies (IEEE Cat. No. 03CH37428)},
  volume~3, pages 1848--1857. IEEE, 2003.

\bibitem[Dalalyan et~al.(2017)Dalalyan, Hebiri, and
  Lederer]{dalalyan2017prediction}
Arnak~S Dalalyan, Mohamed Hebiri, and Johannes Lederer.
\newblock On the prediction performance of the lasso.
\newblock 2017.

\bibitem[Dallakyan and Pourahmadi(2022)]{dallakyan2022fused}
Aramayis Dallakyan and Mohsen Pourahmadi.
\newblock Fused-lasso regularized cholesky factors of large nonstationary
  covariance matrices of replicated time series.
\newblock \emph{Journal of Computational and Graphical Statistics}, \penalty0
  (just-accepted):\penalty0 1--27, 2022.

\bibitem[Dette et~al.(1998)Dette, Munk, and Wagner]{dette1998estimating}
Holger Dette, Axel Munk, and Thorsten Wagner.
\newblock Estimating the variance in nonparametric regression—what is a
  reasonable choice?
\newblock \emph{Journal of the Royal Statistical Society: Series B (Statistical
  Methodology)}, 60\penalty0 (4):\penalty0 751--764, 1998.

\bibitem[Donoho(1997)]{donoho1997cart}
David~L Donoho.
\newblock Cart and best-ortho-basis: a connection.
\newblock \emph{The Annals of Statistics}, 25\penalty0 (5):\penalty0
  1870--1911, 1997.

\bibitem[Fan and Yao(1998)]{fan1998efficient}
Jianqing Fan and Qiwei Yao.
\newblock Efficient estimation of conditional variance functions in stochastic
  regression.
\newblock \emph{Biometrika}, 85\penalty0 (3):\penalty0 645--660, 1998.

\bibitem[Fan and Guan(2018)]{fan2018approximate}
Zhou Fan and Leying Guan.
\newblock Approximate l0-penalized estimation of piecewise-constant signals on
  graphs.
\newblock \emph{The Annals of Statistics}, 46\penalty0 (6B):\penalty0
  3217--3245, 2018.

\bibitem[Gavish et~al.(2010)Gavish, Nadler, and Coifman]{gavish2010multiscale}
Matan Gavish, Boaz Nadler, and Ronald~R Coifman.
\newblock Multiscale wavelets on trees, graphs and high dimensional data:
  Theory and applications to semi supervised learning.
\newblock In \emph{ICML}, 2010.

\bibitem[Guntuboyina et~al.(2020)Guntuboyina, Lieu, Chatterjee, and
  Sen]{guntuboyina2020adaptive}
Adityanand Guntuboyina, Donovan Lieu, Sabyasachi Chatterjee, and Bodhisattva
  Sen.
\newblock Adaptive risk bounds in univariate total variation denoising and
  trend filtering.
\newblock 2020.

\bibitem[Hall and Carroll(1989)]{hall1989variance}
Peter Hall and Raymond~J Carroll.
\newblock Variance function estimation in regression: the effect of estimating
  the mean.
\newblock \emph{Journal of the Royal Statistical Society: Series B
  (Methodological)}, 51\penalty0 (1):\penalty0 3--14, 1989.

\bibitem[Hammond et~al.(2011)Hammond, Vandergheynst, and
  Gribonval]{hammond2011wavelets}
David~K Hammond, Pierre Vandergheynst, and R{\'e}mi Gribonval.
\newblock Wavelets on graphs via spectral graph theory.
\newblock \emph{Applied and Computational Harmonic Analysis}, 30\penalty0
  (2):\penalty0 129--150, 2011.

\bibitem[H{\"u}tter and Rigollet(2016)]{hutter2016optimal}
Jan-Christian H{\"u}tter and Philippe Rigollet.
\newblock Optimal rates for total variation denoising.
\newblock In \emph{Conference on Learning Theory}, pages 1115--1146. PMLR,
  2016.

\bibitem[Johnson(2013)]{johnson2013dynamic}
Nicholas Johnson.
\newblock A dynamic programming algorithm for the fused lasso and
  $l_0$-segmentation.
\newblock \emph{Journal of Computational and Graphical Statistics}, 22\penalty0
  (2):\penalty0 246--260, 2013.

\bibitem[Jula~Vanegas et~al.(2021)Jula~Vanegas, Behr, and
  Munk]{jula2021multiscale}
Laura Jula~Vanegas, Merle Behr, and Axel Munk.
\newblock Multiscale quantile segmentation.
\newblock \emph{Journal of the American Statistical Association}, pages 1--14,
  2021.

\bibitem[Ledoux and Talagrand(1991)]{ledoux1991probability}
Michel Ledoux and Michel Talagrand.
\newblock \emph{Probability in Banach Spaces: isoperimetry and processes},
  volume~23.
\newblock Springer Science \& Business Media, 1991.

\bibitem[Lin et~al.(2017)Lin, Sharpnack, Rinaldo, and Tibshirani]{lin2017sharp}
Kevin Lin, James~L Sharpnack, Alessandro Rinaldo, and Ryan~J Tibshirani.
\newblock A sharp error analysis for the fused lasso, with application to
  approximate changepoint screening.
\newblock In \emph{Advances in Neural Information Processing Systems}, pages
  6884--6893, 2017.

\bibitem[Madrid-Padilla and Chatterjee(2020)]{madrid2020risk}
Oscar~Hernan Madrid-Padilla and Sabyasachi Chatterjee.
\newblock Risk bounds for quantile trend filtering.
\newblock \emph{Biometrika}, 2020.

\bibitem[Madrid~Padilla et~al.(2020{\natexlab{a}})Madrid~Padilla, Sharpnack,
  Chen, and Witten]{madrid2020adaptive}
Oscar~Hernan Madrid~Padilla, James Sharpnack, Yanzhen Chen, and Daniela~M
  Witten.
\newblock Adaptive nonparametric regression with the k-nearest neighbour fused
  lasso.
\newblock \emph{Biometrika}, 107\penalty0 (2):\penalty0 293--310,
  2020{\natexlab{a}}.

\bibitem[Madrid~Padilla et~al.(2020{\natexlab{b}})Madrid~Padilla, Sharpnack,
  Chen, and Witten]{padilla2018adaptive}
Oscar~Hernan Madrid~Padilla, James Sharpnack, Yanzhen Chen, and Daniela~M
  Witten.
\newblock Adaptive nonparametric regression with the k-nearest neighbour fused
  lasso.
\newblock \emph{Biometrika}, 107\penalty0 (2):\penalty0 293--310,
  2020{\natexlab{b}}.

\bibitem[Madrid-Padilla et~al.(2021{\natexlab{a}})Madrid-Padilla, Chen, and
  Ruiz]{padilla2021causal}
Oscar~Hernan Madrid-Padilla, Yanzhen Chen, and Gabriel Ruiz.
\newblock A causal fused lasso for interpretable heterogeneous treatment
  effects estimation.
\newblock \emph{arXiv preprint arXiv:2110.00901}, 2021{\natexlab{a}}.

\bibitem[Madrid-Padilla et~al.(2021{\natexlab{b}})Madrid-Padilla, Yu, and
  Rinaldo]{madrid2021lattice}
Oscar~Hernan Madrid-Padilla, Yi~Yu, and Alessandro Rinaldo.
\newblock Lattice partition recovery with dyadic cart.
\newblock \emph{Advances in Neural Information Processing Systems},
  34:\penalty0 26143--26155, 2021{\natexlab{b}}.

\bibitem[Mammen and van~de Geer(1997)]{mammen1997locally}
Enno Mammen and Sara van~de Geer.
\newblock Locally apadtive regression splines.
\newblock \emph{Annals of Statistics}, 25\penalty0 (1):\penalty0 387--413,
  1997.

\bibitem[Ortelli and van~de Geer(2018)]{ortelli2018total}
Francesco Ortelli and Sara van~de Geer.
\newblock On the total variation regularized estimator over a class of tree
  graphs.
\newblock \emph{Electronic Journal of Statistics}, 12\penalty0 (2):\penalty0
  4517--4570, 2018.

\bibitem[Ortelli and van~de Geer(2020)]{ortelli2020adaptive}
Francesco Ortelli and Sara van~de Geer.
\newblock Adaptive rates for total variation image denoising.
\newblock \emph{Journal of Machine Learning Research}, 21:\penalty0 247, 2020.

\bibitem[Ortelli and van~de Geer(2021)]{ortelli2021prediction}
Francesco Ortelli and Sara van~de Geer.
\newblock Prediction bounds for higher order total variation regularized least
  squares.
\newblock \emph{The Annals of Statistics}, 49\penalty0 (5):\penalty0
  2755--2773, 2021.

\bibitem[Padilla(2018)]{padilla2018graphon}
Oscar Hernan~Madrid Padilla.
\newblock Graphon estimation via nearest neighbor algorithm and 2d fused lasso
  denoising.
\newblock \emph{arXiv preprint arXiv:1805.07042}, 2018.

\bibitem[Padilla et~al.(2018)Padilla, Sharpnack, Scott, and
  Tibshirani]{padilla2016dfs}
Oscar Hernan~Madrid Padilla, James Sharpnack, James~G Scott, and Ryan~J
  Tibshirani.
\newblock The dfs fused lasso: Linear-time denoising over general graphs.
\newblock \emph{Journal of Machine Learning Research}, 18:\penalty0 176--1,
  2018.

\bibitem[Rice(1984)]{rice1984bandwidth}
John Rice.
\newblock Bandwidth choice for nonparametric regression.
\newblock \emph{The Annals of Statistics}, pages 1215--1230, 1984.

\bibitem[Rudin et~al.(1992)Rudin, Osher, and Faterni]{rudin1992nonlinear}
Leonid Rudin, Stanley Osher, and Emad Faterni.
\newblock Nonlinear total variation based noise removal algorithms.
\newblock \emph{Physica {D}: Nonlinear Phenomena}, 60\penalty0 (1):\penalty0
  259--268, 1992.

\bibitem[Sadhanala and Tibshirani(2019)]{sadhanala2019additive}
Veeranjaneyulu Sadhanala and Ryan~J Tibshirani.
\newblock Additive models with trend filtering.
\newblock \emph{The Annals of Statistics}, 47\penalty0 (6):\penalty0
  3032--3068, 2019.

\bibitem[Sadhanala et~al.(2016)Sadhanala, Wang, and
  Tibshirani]{sadhanala2016total}
Veeranjaneyulu Sadhanala, Yu-Xiang Wang, and Ryan~J Tibshirani.
\newblock Total variation classes beyond 1d: Minimax rates, and the limitations
  of linear smoothers.
\newblock \emph{Advances in Neural Information Processing Systems}, 29, 2016.

\bibitem[Sadhanala et~al.(2017)Sadhanala, Wang, Sharpnack, and
  Tibshirani]{sadhanala2017higher}
Veeranjaneyulu Sadhanala, Yu-Xiang Wang, James~L Sharpnack, and Ryan~J
  Tibshirani.
\newblock Higher-order total variation classes on grids: Minimax theory and
  trend filtering methods.
\newblock \emph{Advances in Neural Information Processing Systems}, 30, 2017.

\bibitem[Sadhanala et~al.(2021)Sadhanala, Wang, Hu, and
  Tibshirani]{sadhanala2021multivariate}
Veeranjaneyulu Sadhanala, Yu-Xiang Wang, Addison~J Hu, and Ryan~J Tibshirani.
\newblock Multivariate trend filtering for lattice data.
\newblock \emph{arXiv preprint arXiv:2112.14758}, 2021.

\bibitem[Sharpnack et~al.(2013)Sharpnack, Singh, and
  Krishnamurthy]{sharpnack2013detecting}
James Sharpnack, Aarti Singh, and Akshay Krishnamurthy.
\newblock Detecting activations over graphs using spanning tree wavelet bases.
\newblock In \emph{Artificial intelligence and statistics}, pages 536--544.
  PMLR, 2013.

\bibitem[Shen et~al.(2020)Shen, Gao, Witten, and Han]{shen2020optimal}
Yandi Shen, Chao Gao, Daniela Witten, and Fang Han.
\newblock Optimal estimation of variance in nonparametric regression with
  random design.
\newblock \emph{The Annals of Statistics}, 48\penalty0 (6):\penalty0
  3589--3618, 2020.

\bibitem[Shuman et~al.(2013)Shuman, Narang, Frossard, Ortega, and
  Vandergheynst]{shuman2013emerging}
David~I Shuman, Sunil~K Narang, Pascal Frossard, Antonio Ortega, and Pierre
  Vandergheynst.
\newblock The emerging field of signal processing on graphs: Extending
  high-dimensional data analysis to networks and other irregular domains.
\newblock \emph{IEEE signal processing magazine}, 30\penalty0 (3):\penalty0
  83--98, 2013.

\bibitem[Smola and Kondor(2003)]{smola2003kernels}
Alexander~J Smola and Risi Kondor.
\newblock Kernels and regularization on graphs.
\newblock In \emph{Learning Theory and Kernel Machines: 16th Annual Conference
  on Learning Theory and 7th Kernel Workshop, COLT/Kernel 2003, Washington, DC,
  USA, August 24-27, 2003. Proceedings}, pages 144--158. Springer, 2003.

\bibitem[Tansey and Scott(2015)]{tansey2015fast}
Wesley Tansey and James Scott.
\newblock A fast and flexible algorithm for the graph-fused lasso.
\newblock \emph{arXiv preprint arXiv:1505.06475}, 2015.

\bibitem[Tansey et~al.(2017)Tansey, Koyejo, Poldrack, and
  Scott]{tansey2014false}
Wesley Tansey, Oluwasanmi Koyejo, Russell~A Poldrack, and James~G Scott.
\newblock False discovery rate smoothing.
\newblock \emph{To appear in Journal of the American Statistical Association},
  2017.

\bibitem[Tarjan(1972)]{tarjan1972depth}
Robert Tarjan.
\newblock Depth-first search and linear graph algorithms.
\newblock \emph{SIAM Journal on Computing}, 1\penalty0 (2):\penalty0 146--160,
  1972.

\bibitem[Tibshirani et~al.(2005)Tibshirani, Saunders, Rosset, Zhu, and
  Knight]{tibshirani2005sparsity}
Robert Tibshirani, Michael Saunders, Saharon Rosset, Ji~Zhu, and Keith Knight.
\newblock Sparsity and smoothness via the fused lasso.
\newblock \emph{Journal of the Royal Statistical Society: Series B},
  67\penalty0 (1):\penalty0 91--108, 2005.

\bibitem[Tibshirani(2014)]{tibshirani2014adaptive}
Ryan~J. Tibshirani.
\newblock Adaptive piecewise polynomial estimation via trend filtering.
\newblock \emph{The Annals of Statistics}, 42\penalty0 (1):\penalty0 285--323,
  2014.

\bibitem[Tibshirani and Taylor(2012)]{tibshirani2012degrees}
Ryan~J. Tibshirani and Jonathan Taylor.
\newblock Degrees of freedom in lasso problems.
\newblock \emph{The Annals of Statistics}, 40\penalty0 (2):\penalty0
  1198--1232, 2012.

\bibitem[Tong and Wang(2005)]{tong2005estimating}
Tiejun Tong and Yuedong Wang.
\newblock Estimating residual variance in nonparametric regression using least
  squares.
\newblock \emph{Biometrika}, 92\penalty0 (4):\penalty0 821--830, 2005.

\bibitem[Tran et~al.(2022)Tran, Wei, and Donnat]{tran2022generalized}
Huy Tran, Sansen Wei, and Claire Donnat.
\newblock The generalized elastic net for least squares regression with
  network-aligned signal and correlated design.
\newblock \emph{arXiv preprint arXiv:2211.00292}, 2022.

\bibitem[Vershynin(2018)]{vershynin2018high}
Roman Vershynin.
\newblock \emph{High-dimensional probability: An introduction with applications
  in data science}, volume~47.
\newblock Cambridge university press, 2018.

\bibitem[Wainwright(2019)]{wainwright2019high}
Martin~J Wainwright.
\newblock \emph{High-dimensional statistics: A non-asymptotic viewpoint},
  volume~48.
\newblock Cambridge University Press, 2019.

\bibitem[Wang et~al.(2008)Wang, Brown, Cai, and Levine]{wang2008effect}
Lie Wang, Lawrence~D Brown, T~Tony Cai, and Michael Levine.
\newblock Effect of mean on variance function estimation in nonparametric
  regression.
\newblock \emph{The Annals of Statistics}, 36\penalty0 (2):\penalty0 646--664,
  2008.

\bibitem[Wang et~al.(2016)Wang, Sharpnack, Smola, and
  Tibshirani]{wang2016trend}
Yu-Xiang Wang, James Sharpnack, Alex Smola, and Ryan~J Tibshirani.
\newblock Trend filtering on graphs.
\newblock \emph{Journal of Machine Learning Research}, 17\penalty0
  (105):\penalty0 1--41, 2016.

\bibitem[Ye and Padilla(2021)]{ye2021non}
Steven~Siwei Ye and Oscar Hernan~Madrid Padilla.
\newblock Non-parametric quantile regression via the k-nn fused lasso.
\newblock \emph{Journal of Machine Learning Research}, 22:\penalty0 111--1,
  2021.

\bibitem[Yu(1997)]{yu1997assouad}
Bin Yu.
\newblock Assouad, fano, and le cam.
\newblock In \emph{Festschrift for Lucien Le Cam: research papers in
  probability and statistics}, pages 423--435. Springer, 1997.

\bibitem[Yu et~al.(2022)Yu, Madrid, and Rinaldo]{yu2022optimal}
Yi~Yu, Oscar Madrid, and Alessandro Rinaldo.
\newblock Optimal partition recovery in general graphs.
\newblock In \emph{International Conference on Artificial Intelligence and
  Statistics}, pages 4339--4358. PMLR, 2022.

\bibitem[Zhou et~al.(2005)Zhou, Huang, and Sch{\"o}lkopf]{zhou2005learning}
Dengyong Zhou, Jiayuan Huang, and Bernhard Sch{\"o}lkopf.
\newblock Learning from labeled and unlabeled data on a directed graph.
\newblock In \emph{Proceedings of the 22nd international conference on Machine
  learning}, pages 1036--1043, 2005.

\bibitem[Zhu et~al.(2003)Zhu, Ghahramani, and Lafferty]{zhu2003semi}
Xiaojin Zhu, Zoubin Ghahramani, and John~D Lafferty.
\newblock Semi-supervised learning using gaussian fields and harmonic
  functions.
\newblock In \emph{Proceedings of the 20th International conference on Machine
  learning (ICML-03)}, pages 912--919, 2003.

\end{thebibliography}
	
\end{document}